\documentclass[12pt]{article}

\usepackage[utf8]{inputenc}

\usepackage{amsfonts,graphicx,amsmath,amssymb,amsthm,geometry,
hyperref,color,nicefrac,enumerate,bbm,verbatim,lmodern, url}
\usepackage{xcolor} 

\usepackage[draft]{todonotes}   

\allowdisplaybreaks[1]

\newtheorem{lemma}{Lemma}[section]
\newtheorem{corollary}[lemma]{Corollary}
\newtheorem{proposition}[lemma]{Proposition}

\newtheorem{theorem}[lemma]{Theorem}

\providecommand{\N}{{\ensuremath{\mathbbm{N}}}}
\providecommand{\R}{{\ensuremath{\mathbbm{R}}}}

\providecommand{\N}{{\ensuremath{\mathbbm{N}}}}
\providecommand{\R}{{\ensuremath{\mathbbm{R}}}}

\allowdisplaybreaks 
\begin{document}

\title{On the Alekseev-Gr\"obner formula in Banach spaces}
\author{Arnulf Jentzen$^{1} $, Felix Lindner$^{2} $, and Primo\v{z} Pu\v{s}nik$^{3} $
\bigskip 
\\
\small{$^1$ Seminar for Applied Mathematics, Department of Mathematics,}
\\
\small{ETH Zurich, Switzerland, 
email: arnulf.jentzen@sam.math.ethz.ch} 
\smallskip
\\
\small{$^2$  Institute of Mathematics, Faculty of Mathematics and Natural Sciences,}
\\
\small{University of Kassel, Germany,
email: lindner@mathematik.uni-kassel.de}
\smallskip
\\
\small{$^3$ Seminar for Applied Mathematics, Department of Mathematics,}
\\
\small{ETH Zurich, Switzerland,
	email: primoz.pusnik@sam.math.ethz.ch} 
}


\maketitle

\begin{abstract}
	The Alekseev-Gr\"obner formula 
	is a well known tool
	in numerical analysis
	for describing the effect 
	that a perturbation of an ordinary differential equation (ODE) has on its solution.  
	In this article we
	provide an extension of the
	Alekseev-Gr\"obner formula
	for Banach space valued ODEs
	under,
	loosely speaking, mild conditions on the perturbation 
	of the considered ODEs. 
\end{abstract}

\tableofcontents
%
\newpage  
\section{Introduction}
%
%
%
The Alekseev-Gr\"obner formula
(see, e.g., Alekseev~\cite{Alekseev1961},  
Gr\"obner~\cite{Grobner1960},
and
Hairer et al.\ \cite[Theorem~14.5 in Chapter~I]{HairerNonstiffProblems})
is a well known tool
in deterministic numerical analysis
for describing the effect 
that a perturbation of an ordinary differential equation (ODE) has on its solution.  
Considering numerical methods for ODEs
as  
appropriate perturbations
of the underlying equations
makes the Alekseev-Gr\"obner formula
applicable for estimating errors of numerical 
methods 
(see, e.g., 
Hairer et al.\ \cite[Theorem~7.9 in Chapter~II]{HairerNonstiffProblems}, 
Iserles~\cite[Theorem~3.7]{Iserles2008},
Iserles~\cite[Theorem~1]{Iserles1993}, 
and
Niesen~\cite[Theorem~1]{Niesen2003}).
It is the main contribution of this work to 
provide an extension of the
Alekseev-Gr\"obner formula
for Banach space valued ODEs 
under,
loosely speaking, mild conditions on the perturbation 
of the considered ODEs
(see Corollary~\ref{corollary:Alekseev_grobner}
in Section~\ref{subsection:AG}
and
Theorem~\ref{theorem:Alekseev_grobner}
below). 
As a consequence, our main result is well suited for the analysis of pathwise approximation errors between 
exact solutions of
stochastic partial differential equations (SPDEs) of evolutionary type
and
their numerical approximations.
 In particular, it can be used as a key ingredient for establishing strong convergence rates for numerical approximations of SPDEs with a
 non-globally Lipschitz continuous,
 non-globally monotone nonlinearity,
 and additive trace-class noise.
The precise result will be the subject of a future research article.
In this introductory section we now present 
our main result. 
Theorem~\ref{theorem:Alekseev_grobner} is proven as Corollary~\ref{corollary:Alekseev_grobner} 
in Section~\ref{subsection:AG}
below. 
\begin{theorem}
	\label{theorem:Alekseev_grobner}  
	Let 
	$ (V, \left\| \cdot \right\|_V) $ 
	be a nontrivial $ \R $-Banach space, 
	let
	$ T \in (0, \infty) $,
	let 
	$ f \colon [0, T] \times V \to V $
	be a continuous function,
	assume for all
	$ t \in [0,T] $ that
	$ V \ni x \mapsto f(t,x) \in V $
	is Fr\'echet differentiable, 
	assume that
	$ [0,T] \times V \ni (t,x) \mapsto 
	( \frac{ \partial }{ \partial x }
	f ) (t,x) \in L(V) $ is continuous,
	for every
	$ x \in V $,
	$ s \in [0,T] $ 
	let
	$ X^x_{s, (\cdot)} \colon [s,T] \to V $
	be a continuous function
	which satisfies
	for all 
	$ t \in [s,T] $ 
	that 
	$ X_{s,t}^x = 
	x
	+
	\int_s^t f( \tau , X_{s, \tau }^x) \,d\tau $,
	and
	let
	$ Y, E \colon [0,T] \to V $
	be strongly measurable functions
	which satisfy for all $ t \in [0,T] $ that
	$ \int_0^T 
	[
	\| f ( \tau, Y_\tau ) \|_V 
	+
	\| E_\tau \|_V 
	]
	\, d\tau < \infty $
	and
	$ Y_t 
	= Y_0 + \int_0^t  
	[ f(\tau, Y_\tau ) + E_\tau ]  
	\,
	d\tau $. 
	Then
	\begin{enumerate}[(i)]
		\item \label{item:Intro1} 
		it holds 
		for all
		$ s \in [0,T] $,
		$ t \in [s,T] $
		that 
		$ V \ni x
		\mapsto X_{s,t}^x \in V $ 
		is
		Fr\'echet
		differentiable,
		\item it holds 
		for all 
		$ t \in [0,T] $
		that
		$ [0,t] \ni \tau
		\mapsto
		 (
		\tfrac{\partial}{\partial x} X_{\tau, t}^{Y_\tau} 
		 ) E_\tau
		\in V $
		is strongly
		measurable,
		\item it holds for all 
		$ t \in [0,T] $
		that
		$ \int_{0}^t
		 \| 
		 (
		\tfrac{\partial}{\partial x} X_{\tau, t}^{Y_\tau} 
		 ) E_\tau
		 \|_V
		\, 
		d\tau < \infty $,
		and
		\item \label{item:Intro2}
		it holds
		for all $ s \in [0, T] $,
		$ t \in [s,T] $ 
		that
		\begin{equation}
		Y_t 
		= 
		X_{s, t}^{Y_s}
		+
		\int_{s}^t 
		(
		\tfrac{\partial}{\partial x} X_{\tau, t}^{Y_\tau} 
		 ) E_\tau
		\, 
		d\tau. 
		\end{equation}
	\end{enumerate}
\end{theorem}
The rest of this article is structured as follows.
In Section~\ref{subsection:ExtendedChainRule} 
we recall some elementary and well known properties for Banach space valued functions
(see
Lemmas~\ref{lemma:stronglyMeasurable}--\ref{lemma:integral_chain_rule},  Corollary~\ref{corollary:integral_chain_rule},
and
Lemma~\ref{lemma:ChangeOrder}).
Thereafter we 
combine these
elementary results
to prove
an abstract version of the Alekseev-Gr\"obner formula 
for Banach space valued ODEs
under, roughly speaking, 
restrictive conditions
on the solution as 
well as on the perturbation of the considered ODE;
see 
Proposition~\ref{proposition:Banach_space_alekseev_grobner}
in Section~\ref{subsection:ExtendedChainRule}
below for details.
Sections~\ref{subsection:Continuity of solutions} and~\ref{subsection:ContinuousDifferentiability}
are devoted to presenting in detail some
partially well known results
on continuous differentiability  
of solutions
to a class of Banach space valued ODEs 
with respect to initial value, initial time, and current time
(see Lemma~\ref{lemma:Differentiable} 
in 
Section~\ref{subsection:ContinuousDifferentiability}
below).
Finally, 
we
combine
Proposition~\ref{proposition:Banach_space_alekseev_grobner}, 
Lemma~\ref{lemma:flow_property}
(the flow property of solutions to ODEs),
and 
Lemma~\ref{lemma:Differentiable}
to establish
in Corollary~\ref{corollary:Alekseev_grobner} 
in Section~\ref{subsection:AG}
below
the main result of this article. 
\section{Extended chain rule property
for Banach space valued functions} 
\label{subsection:ExtendedChainRule}
In this section we
prove an abstract version of
the Alekseev-Gr\"obner  formula for  
Banach space valued
ODEs
under, loosely speaking,
restrictive conditions
 in Proposition~\ref{proposition:Banach_space_alekseev_grobner}.
This will be used
in Section~\ref{subsection:AG}
to prove 
in Corollary~\ref{corollary:Alekseev_grobner} an extension of the Alekseev-Gr\"obner formula (cf., e.g.,
Hairer et al.\ \cite[Theorem~14.5 in Chapter~I]{HairerNonstiffProblems}) for Banach space valued functions.
In order to prove 
Proposition~\ref{proposition:Banach_space_alekseev_grobner}
we first recall
some
elementary
auxiliary lemmas;
see
Lemmas~\ref{lemma:stronglyMeasurable}--\ref{lemma:integral_chain_rule},  Corollary~\ref{corollary:integral_chain_rule},
and
Lemma~\ref{lemma:ChangeOrder}. 
In particular, we recall the fundamental theorem of calculus for Banach space valued functions 
in Lemmas~\ref{lemma:Lebesgues_diff}--\ref{lemma:FundamentaTheoremOfCalculus}
(cf., e.g., Pr\'ev\^ot \& R\"ockner~\cite[Proposition A.2.3]{PrevotRoeckner2007})
and we derive 
suitable extensions thereof in 
Lemma~\ref{lemma:integral_chain_rule}
(cf., e.g., \cite[Lemma~2.1]{JentzenDiyoraWelti2017})
and
Corollary~\ref{corollary:integral_chain_rule}.
Thereafter, we combine Lemma~\ref{lemma:integral_chain_rule},
Corollary~\ref{corollary:integral_chain_rule}, 
and
Lemma~\ref{lemma:ChangeOrder}
(cf., e.g., Rudin~\cite[Theorem~7.17 and the remark thereafter]{Rudin76}) 
to establish
Proposition~\ref{proposition:Banach_space_alekseev_grobner}.
\begin{lemma}
	\label{lemma:stronglyMeasurable}
	Let $ ( X, \mathcal{X}) $ be a separable topological space,
	let
	$ ( Y, \mathcal{Y} ) $ be a topological space,
	and let
	$ f \in \mathcal{C}( X, Y ) $.
	Then $ f $ is strongly measurable.
\end{lemma}
\begin{proof}[Proof of Lemma~\ref{lemma:stronglyMeasurable}]
Throughout this proof let
$ A \subseteq X $ be a countable dense subset of $ X $.
Note that
the assumption that $ f $ is continuous ensures that for all
$ V \in \mathcal{Y} $ with
$ f(X) \cap V \neq \emptyset $
it holds that
\begin{equation} 
\emptyset \neq \{ x \in X \colon f(x) \in V \} \in \mathcal{X} 
.
\end{equation} 
This and the fact that
$ A \subseteq X $
is dense imply that
for every
$ V \in \mathcal{Y} $
with
$ f(X) \cap V \neq \emptyset $
there exists
$ a \in A $
such that
$ f(a) \in V $.
The fact that
$ A \subseteq X $ is countable therefore
implies that
$ f(A) \subseteq f(X) $ is a countable dense subset of $ f( X ) $.
This and the fact that
$ f $ is measurable complete the proof of
Lemma~\ref{lemma:stronglyMeasurable}.
\end{proof}
\begin{lemma}
	\label{lemma:Lebesgues_diff}
	Let $ ( V, \left \| \cdot
	\right\|_V ) $
	be
	an $ \R $-Banach space,
	let $ a \in \R $,
	$ b \in (a,\infty) $,
	and let
	$ f \in \mathcal{C}([a,b], V) $.
	Then 
	it holds
	for every $ t \in [a,b] $ that
	$ \limsup_{ ( [a-t, b-t] \backslash \{0\} ) \ni h \to 0}
	\| \frac{1}{h}
	\int_t^{ t + h } f(s) \, ds - f(t) \|_V 
	=
	0 
	$.
\end{lemma}
\begin{proof}[Proof of Lemma~\ref{lemma:Lebesgues_diff}]
	Throughout this proof
	let
	$ (c_n)_{n\in \N} \subseteq 
	\operatorname{im}(f) $
	be a dense subset of $ \operatorname{im}(f) $. 
	Note that the fundamental theorem of calculus
	assures that 
	for all
	$ t \in [a,b] $,
	$ n \in \N $
	it holds that
	\begin{equation}
	\begin{split}
	\limsup_{ ( [a-t, b-t] \backslash \{0\} ) \ni h \to 0}
	\Big|
	\tfrac{1}{ h }
	\int_t^{ t+h }
	\| f(s) - c_n \|_V \, ds
	-
	\| f(t) - c_n \|_V
	\Big|=
	0 
	.
	\end{split}
	\end{equation}
	This implies
	for all
	$ t \in  [a,b] $,
	$ n \in \N $
	that
	\begin{equation}
	\begin{split}
	&
	\limsup_{ ( [a-t, b-t] \backslash \{0\} ) \ni h \to 0}
	\tfrac{1}{ | h | }
	\int_{ \min\{t, t + h\} }^{ \max \{t,t + h\} }
	\| f(s) - f(t) \|_V \, ds
	\\
	&
	\leq
	\limsup_{ ( [a-t, b-t] \backslash \{0\} ) \ni h \to 0}
	\tfrac{1}{ | h | }
	\int_{ \min\{t, t + h\} }^{ \max \{t,t + h\} }
	\big(
	\| f(s) - c_n \|_V
	+
	\| c_n - f(t) \|_V
	\big)
	\, 
	ds
	\\
	&
	=
	\limsup_{ ( [a-t, b-t] \backslash \{0\} ) \ni h \to 0}
	\tfrac{1}{ | h | }
	\int_{ \min\{t, t + h\} }^{ \max \{t,t + h\} }
	\| f(s) - c_n \|_V \, ds
	\\
	&
	\quad
	+
	\limsup_{ ( [a-t, b-t] \backslash \{0\} ) \ni h \to 0}
	\tfrac{1}{ | h | }
	\int_{ \min\{t, t + h\} }^{ \max \{t,t + h\} }
	\| c_n - f(t) \|_V \, ds
	\\
	&
	=
	\| f(t) - c_n \|_V
	+
	\limsup_{ ( [a-t, b-t] \backslash \{0\} ) \ni h \to 0}
	\tfrac{ 
		| h |
	} { | h | }
	\| c_n - f(t) \|_V
	=
	2
	\| f (t) - c_n \|_V.
	\end{split}
	\end{equation}
	The fact that
	$ ( c_n )_{n \in \N} \subseteq 
	\operatorname{im}(f) $
	is dense
	hence
	ensures
	for all
	$ t \in [a,b] $
	that
	\begin{equation}
	\limsup_{ ( [a-t, b-t] \backslash \{0\} ) \ni h \to 0}
	\tfrac{1}{ h }
	\int_t^{ t + h }
	\| f(s) - f(t) \|_V \, ds
	=
	0.
	\end{equation}
	Therefore, we obtain for every $ t \in [a,b] $ that
	\begin{equation}
	\begin{split}
	\label{eq:left_limits}
	&
	\limsup_{ ( [a-t, b-t] \backslash \{0\} ) \ni h \to 0}
	\Big\|
	\tfrac{1}{ h }
	\int_t^{ t + h }
	f(s)  \, ds  
	- f(t)
	\Big\|_V
	\\
	&
	\leq
	\limsup_{ ( [a-t, b-t] \backslash \{0\} ) \ni h \to 0}
	\tfrac{1}{ h }
	\int_t^{ t+h }
	\|
	f(s)   
	- 
	f(t)
	\|_V
	\, ds 
	=
	0.
	\end{split}
	\end{equation}
	This
	completes the proof of 
	Lemma~\ref{lemma:Lebesgues_diff}.
\end{proof}
\begin{lemma}
	\label{lemma:extensions}
	Let $ ( V, \left\| \cdot \right \|_V ) $
	be an $ \R $-Banach space
	and
	let 
	$ a \in \R $,
	$ b \in (a,\infty) $,
	$ f \in \mathcal{C}([a,b], V) $,
	$ F \colon [a,b] \to V $
	satisfy for all
	$ t \in [a,b] $ that
	$ F(t) 
	= F(a) + \int_a^t f(s) \,ds $.
	Then it holds for every
	$ t \in [a,b] $ that 
	$ \limsup_{ ( [a-t, b-t] \backslash \{0\} ) \ni h \to 0}
	\| \frac{1}{h} ( F( t + h ) - F(t) - f(t) h ) \|_V 
	= 0 $.
\end{lemma}
\begin{proof}[Proof of Lemma~\ref{lemma:extensions}]
	Note that for all 
	$ t \in [a, b] $,
	$ h \in [a-t,b-t] \backslash \{0\} $
	it holds that
	\begin{equation}
	\begin{split}
	\tfrac{ F( t+h ) - F(t)}{ h }
	=
	\tfrac{1}{ h }
	\int_t^{t+h }
	f(s) \, ds.
	\end{split}
	\end{equation}
	Combining this with
	Lemma~\ref{lemma:Lebesgues_diff}
	(with 
	$ V = V $,
	$ a = a $,
	$ b = b $,
	$ f = f $
	in the notation of
	Lemma~\ref{lemma:Lebesgues_diff})
	implies  for all
	$ t \in [a,b] $
	that
	\begin{equation}
	\begin{split}
	&
	\limsup_{ ( [a-t, b-t] \backslash \{0\} ) \ni h \to 0}
	\big\| 
	\tfrac{ F( t+h ) - F(t)}{ h }
	- f(t)
	\big\|_V
	=
	\limsup_{ ( [a-t, b-t] \backslash \{0\} ) \ni h \to 0}
	\Big\|
	\tfrac{1}{ h }
	\int_t^{ t+h }
	f(s) \, ds - f(t)
	\Big\|_V
	=
	0 
	.
	\end{split}
	\end{equation}
	The proof of Lemma~\ref{lemma:extensions}
	is thus completed.
\end{proof}
\begin{lemma} 
	\label{lemma:FundamentaTheoremOfCalculus}
	Let $ ( V, \left \| \cdot \right \|_V ) $
	be an $ \R $-Banach space 
	and let 
	$ a \in \R $,
	$ b \in (a, \infty) $, 
	$ F \in \mathcal{C}^1 ( [a,b], V ) $.
	Then it holds for all
	$ t \in [a,b] $ 
	that
	\begin{equation}
	F(t) - F(a) 
	=
	\int_a^t F'(s) \, ds
	.
	\end{equation}
\end{lemma}
\begin{proof}[Proof of Lemma~\ref{lemma:FundamentaTheoremOfCalculus}]
	Throughout this proof let 
	$ G \colon [a,b] \to V $
	be the function which satisfies
	  for all
	$ t \in [a,b] $ that 
	$ G(t) - F(a) = \int_a^t F'(s) \, ds $
	and for every
	$ L \in L(V, \R) $ let
	$ p_L \colon [a,b]\to \R $
	be the function which satisfies for all
	$ t \in (a,b) $ that
	$ p_L(t)= L( F(t) - G(t) ) $.
	Observe that
	Lemma~\ref{lemma:extensions}
	(with
	$ V = V $,
	$ a = a $,
	$ b = b $,
	$ f = F' $,
	$ F = G $
	in the notation of Lemma~\ref{lemma:extensions})
	shows that
	for all
	$ t \in [a,b] $ it holds that
	\begin{equation} 
	G'(t) = F'(t)  
	.
	\end{equation}
	This and the fact that
	$ F' \in \mathcal{C}( [a, b], V ) $
	ensure that
	for all 
	$ L \in L(V, \R ) $,
	$ t \in (a,b) $
	it holds that
	$ p_L \in \mathcal{C}^1( [a,b], \R ) $
	and
	$ p_L'(t) = 0 $.
This proves that
for every $ L \in L(V, \R) $
	the function $ p_L $ is constant.
	The fact that the space
	$ L(V,\R) $ separates points 
(see, e.g., Brezis~\cite[Corollary~1.3 in Chaper~1]{Brezis2011})
	hence shows that
	$ [a,b] \ni t \mapsto F(t) - G(t) $ is a constant function.
	This demonstrates that for all
	$ t \in [a,b] $ it holds that
	\begin{equation}
	F(t) - F(a) = \int_a^t F'(s) \, ds.
	\end{equation}
	The proof of Lemma~\ref{lemma:FundamentaTheoremOfCalculus}
	is thus completed.
\end{proof}
\begin{lemma}
	\label{lemma:Help_cont}
	Let $ ( V, d_V ) $
	and
	$ ( W, d_W ) $
	be metric spaces
	and let 
	$ a \in \R $,
	$ b \in (a, \infty) $, 
	$ g \in \mathcal{C}( [a,b] \times V, W ) $.
	Then it holds that
	the function
	\begin{equation} 
	V \ni v \mapsto 
	( [a,b] \ni t \mapsto g(t,v) \in W ) 
	\in 
	\mathcal{C} ( [a,b], W ) 
	\end{equation}
	is continuous.
\end{lemma}
\begin{proof}[Proof of Lemma~\ref{lemma:Help_cont}]
	Throughout this proof let 
	$ \varepsilon \in (0,\infty) $.
	Observe that the assumption that
	$ g \in \mathcal{C}( [a,b] \times V, W ) $
	implies that
	there exists
	a function
	$ \delta \colon [ a,b] \times V \to (0,\infty) $ 
	such that for all
	$ x, x_0 \in V $,
	$ s, t \in [a,b] $ with $ \max \{ |s-t|, d_V(x, x_0) \}<\delta_{t,x_0} $ 
	it holds that
	\begin{equation}
	\label{eq:cont_condition}
	d_W( g(s,x), g(t, x_0) ) < \tfrac{ \varepsilon }{2} .
	\end{equation}
	Moreover, note that the fact that
	$ [a,b] $ is compact ensures that
	for every 
	$ x_0 \in V $
	there exist
	$ n \in \N $,
	$ t_1, \ldots, t_n \in [a,b] $
	such that
	$ a = t_1 < \ldots < t_n = b $
	and
	\begin{equation}
	[ a, b ] = \cup_{ i = 1 }^n  
	\{ r \in [a,b] \colon | r - t_i | < \delta_{t_i, x_0} \}
	.
	\end{equation}
	This and~\eqref{eq:cont_condition} demonstrate that
	for all
	$ x_0 \in V $
	there exist
	$ n \in \N $,
	$ t_1, \ldots, t_n \in [a,b] $
	such that for all
	$ x \in V $,
	$ t \in [a,b] $ 
	with
	$ d_V( x, x_0 ) < \min \{ \delta_{t_1, x_0},  \ldots, \delta_{t_n, x_0} \} $
	it holds that
	\begin{equation}
	\begin{split}
	d_W( g (t, x), g(t, x_0) )
	&
	\leq
	\min_{i\in [0,n]\cap \N} 
	\big| 
	d_W( g(t, x), g(t_i, x_0) )
	+
	d_W( g(t_i, x_0), g(t, x_0) )
	\big|
	%
	<
	\tfrac{\varepsilon}{2}
	+
	\tfrac{\varepsilon}{2}
	=
	\varepsilon
	.
	\end{split}
	\end{equation}
	As $ \varepsilon \in (0, \infty) $
	was arbitrary, the proof of Lemma~\ref{lemma:Help_cont}
	is completed.
\end{proof}
\begin{lemma}
	\label{lemma:integral_chain_rule} 
 	Let $ ( V, \left\| \cdot \right \|_V ) $
 	be a nontrivial $ \R $-Banach space,
 	let 
	$ ( W, \left\| \cdot \right \|_W ) $
	be an $ \R $-Banach space,
	 let $ a \in \R $,
	$ b \in (a,\infty) $, 
	$ \phi \in \mathcal{C}^1(V, W) $, 
	$ F \colon [a,b] \to V $,
	and let
	$ f \colon [a,b] \to V $
	be a strongly measurable function
	which 
	satisfies for all
	$ t \in [a,b] $ that
	$ \int_a^b \| f(s) \|_V \, ds < \infty $
	and
	$ F(t) - F(a)
	= \int_a^t f(s) \,ds $.
	Then 
	\begin{enumerate}[(i)] 
	\item \label{item:L1defined} it holds that
	$
	 [a,b] \ni s \mapsto  \phi'(F(s)) f(s) \in W $
	is strongly measurable,
	\item \label{item:L1definedB}
	it holds that
	$ \int_a^b
	\| \phi'(F(s)) f(s) \|_W \, ds
	< \infty $, 
	and
	\item \label{item:ChainRule}
	it holds for all
	$ t \in [a, b] $ that
	$ \phi(F(t))
	-
	\phi(F(a))
	=
	\int_a^t \phi'(F(s)) f(s) \, ds $.
	\end{enumerate}
\end{lemma}
\begin{proof}[Proof of Lemma~\ref{lemma:integral_chain_rule}]
Throughout this proof let
$ (f_n)_{n \in \N} \subseteq \mathcal{C}( [a,b], V ) $
be functions which satisfy
$ \limsup_{ n \to \infty } 
\int_a^b
\| f_n(s) - f(s) \|_V \, ds = 0 $
and let
$ F_n \colon [a,b] \to V $, $ n \in \N $,
be the functions which satisfy for all $ n \in \N $,
$ t \in [a,b] $ that
$ F_n(t) = F(a) + \int_a^t f_n(s) \, ds $.
Observe that the 
fact that $ f $ is strongly measurable,
the fact that
the function
$ [ a,b ] \ni s \mapsto \phi'( F(s) ) \in L(V,W) $ is continuous,
and 
the fact that
the function
$ V \times L(V,W) \ni (v, A) \mapsto A v \in W $ is continuous
implies that
$ [a,b] \ni s \mapsto  \phi'(F(s)) f(s) \in W $
is strongly measurable.
The fact that the function
$ [a,b] \ni s \mapsto \phi'( F( s ) ) \in L(V,W) $
is continuous
and the assumption that 
$ f $ 
is strongly measurable
and integrable
hence establish items~\eqref{item:L1defined}
and~\eqref{item:L1definedB}. 
Next observe that
Lemma~\ref{lemma:extensions}
(with
$ V = V $,
$ a = a $,
$ b = b $,
$ f = f_n $,
$ F = F_n $
for 
$ n \in \N $
in the notation of Lemma~\ref{lemma:extensions})
ensures that
for every $ n \in \N $ it holds that
$ F_n \in \mathcal{C}^1( [a,b], V ) $
and
$ F_n' = f_n $.
Lemma~\ref{lemma:FundamentaTheoremOfCalculus}
(with 
$ V = W $,
$ a = a $,
$ b = b $,
$ F = \phi \circ F_n $
for $ n \in \N $
in the notation of 
Lemma~\ref{lemma:FundamentaTheoremOfCalculus})
and the chain 
rule 
for Fr\'echet derivatives
therefore
prove that for all
$ n \in \N $, 
$ t \in [a,b] $ it holds that
\begin{equation}
\begin{split}
\label{eq:Basic_chain}
\phi(F_n(t)) - \phi(F_n(a)) =
\int_a^t \phi'(F_n(s)) f_n(s) \, ds
.
\end{split}
\end{equation}
Moreover, note that for all
$ n \in \N $,
$ t \in [a,b] $
it holds that
\begin{equation}
\begin{split}
\| F_n(t) - F(t) \|_V
=
\Big \| \int_a^t [ f_n(s) - f(s) ] \, ds \Big\|_V 
\leq
\int_a^b \| f_n(s) - f(s) \|_V \, ds  
.
\end{split}
\end{equation}
This assures that
\begin{equation}
\begin{split}
\label{eq:unifConv}
\limsup\nolimits_{n \to \infty} 
(
\sup\nolimits_{ s \in [a,b] }
\| F_n( s ) - F( s ) \|_V
)
=
0
.
\end{split}
\end{equation}
The fact that
$ \phi \in \mathcal{C}^1(V, W) $
hence shows that
for every $ t \in [a,b] $ it holds that
\begin{equation}
\begin{split}
\label{eq:cont_Conv}
\limsup\nolimits_{n \to \infty}  
\| \phi( F_n( t ) ) - \phi( F( t ) ) \|_{W} 
=
0 
\end{split}
\end{equation}
and
\begin{equation}
\begin{split}
\limsup\nolimits_{n \to \infty}  
\| \phi'( F_n( t ) ) - \phi'( F( t ) ) \|_{L(V, W)} 
=
0
.
\end{split}
\end{equation}
Next observe that for all
$ n \in \N $,
$ t \in [a,b] $ it holds that
\begin{equation}
\begin{split}
\label{eq:limit_estimate}
&
\int_a^t 
\| \phi'(F_n(s)) f_n(s) - \phi'(F(s)) f(s) \|_W \, ds 
\\
&
\leq 
\int_a^t 
\| \phi'(F_n(s))  \|_{L(V,W)} 
\|  f_n(s) - f(s) \|_V \, ds
+
\int_a^t 
\| \phi'(F_n(s)) - \phi'(F(s)) \|_{L(V,W)} 
\| f(s) \|_V \, ds
\\
&
\leq
\sup\nolimits_{ r \in [a,b] } 
\| \phi'( F(r) + ( F_n(r) - F(r) ) ) \|_{ L(V, W) } 
\int_a^b 
\| f_n(s) - f(s) \|_V
\, ds
\\
&
\quad
+
\sup\nolimits_{ r \in [a,b] }
\| \phi'( F(r) + ( F_n(r) - F(r) ) ) 
- 
\phi'(F(r)) \|_{L(V,W)} 
\int_a^b 
\| f (s) \|_V \, ds
.
\end{split}
\end{equation}
Lemma~\ref{lemma:Help_cont}
(with
$ V = V $,
$ d_V = ( V \times V \ni (v_1,v_2) \mapsto 
\| v_1 - v_2 \|_V \in \R ) $,
$ W = L(V,W) $, 
$ d_W = ( W \times W \ni (w_1, w_2 ) 
\mapsto \| w_1 - w_2 \|_W \in \R ) $,
$ a = a $,
$ b = b $,
$ g = ( [a,b] \times V \ni (t,x) \mapsto \phi'( F( t ) + x ) 
\in L(V, W) ) $
in the notation of 
Lemma~\ref{lemma:Help_cont})
and~\eqref{eq:unifConv}
prove that
for every $ \varepsilon \in (0, \infty ) $
there exists $ N \in \N $
such that
for every $ n \in [N, \infty) \cap \N $
it holds that
\begin{equation}
\begin{split}  
\label{eq:uniform2} 
\sup\nolimits_{ r \in [a,b] }
\| \phi'( F(r) + ( F_n(r) - F(r) ) ) 
- 
\phi'(F(r)) \|_{L(V,W)} 
< \varepsilon
.
\end{split}
\end{equation}
In particular, this 
and the fact that
$ \phi' \circ F \in \mathcal{C} ( [a,b], L(V, W) ) $
imply that
\begin{equation}
\begin{split} 
\label{eq:uniform1}
\sup\nolimits_{ n \in \N } 
\sup\nolimits_{ r \in [a,b] } 
\| \phi'( F(r) + ( F_n(r) - F(r) ) ) \|_{ L(V, W) }
< \infty 
.
\end{split} 
\end{equation} 
Combining~\eqref{eq:limit_estimate}--\eqref{eq:uniform1} 
and the fact that
$ \limsup_{ n \to \infty } 
\int_a^b
\| f_n(s) - f(s) \|_V \, ds = 0 $ 
ensures that
for every $ t \in [a,b] $ it holds that
\begin{equation}
\begin{split} 
\label{eq:int_conv}
\limsup_{ n \to \infty }
\Big\|
\int_a^t \phi'( F_n(s)) f_n(s) \, ds
-
\int_a^t \phi'( F(s)) f(s) \, ds
\Big\|_W
=
0.
\end{split} 
\end{equation}
Moreover, observe that~\eqref{eq:cont_Conv}
shows that for all
$ t \in [a,b] $ it holds that
\begin{equation}
\begin{split}
\limsup\nolimits_{ n \to \infty}
\| \phi ( F_n( t ) ) - \phi( F_n( a ) ) 
-
[ \phi( F( t ) ) - \phi( F( a ) ) ] \|_W
=
0.
\end{split}
\end{equation}
This, \eqref{eq:Basic_chain}, 
and~\eqref{eq:int_conv}
establish
item~\eqref{item:ChainRule}.
The proof of Lemma~\ref{lemma:integral_chain_rule}
is thus completed.
\end{proof} 
\begin{corollary}
	\label{corollary:integral_chain_rule} 
	Let $ ( V, \left\| \cdot \right \|_V ) $
	be a nontrivial $ \R $-Banach space,
	let
	$ ( W, \left\| \cdot \right \|_W ) $
	be an  
	$ \R $-Banach space, 
	let $ a \in \R $,
	$ b \in (a,\infty) $,  
	$ F \colon [a,b] \to V $,
	$ \phi \in \mathcal{C}^1([a, b] \times V, W) $,  
	$ \phi_{1,0} \colon [a,b] \times V \to W $,
	$ \phi_{0,1} \colon [a,b] \times V \to L(V,W) $
	satisfy for all
	$ t \in [a,b] $,
	$ x \in V $ 
	that
	$ \phi_{1,0}(t,x) = (\frac{\partial}{\partial t } \phi)(t,x) $,
	$ \phi_{0,1}(t,x) = (\frac{\partial}{\partial x } \phi)(t,x) $,
	and let
	$ f \colon [a,b] \to V $
	be a strongly measurable function
	which satisfies for all
	$ t \in [a,b] $ that
	$ \int_a^b \| f(s) \|_V \, ds < \infty $
	and
	$ F(t) - F(a)
	= \int_a^t f(s) \,ds $.
	Then 
	\begin{enumerate}[(i)]
	\item \label{item:FirstIt} it holds that
	$ 
	[a,b] \ni s \mapsto  
	 [ 
	\phi_{1,0} (s, F(s)) 
	+
	\phi_{0,1} (s, F(s)) 
	f(s) 
	 ] \in W  
$
is strongly measurable,
	\item 
	it holds that
	$ \int_a^b 	\|
	\phi_{1,0} (s, F(s)) 
	+
	\phi_{0,1} (s, F(s)) 
	f(s) 
	\|_W
	\, ds < \infty $,
	and
	\item \label{item:LastIt}
	it holds for all
	$ t \in [a, b] $ that
	\begin{equation} 
	\label{eq:integral_chain}
	\phi(t, F(t))
	-
	\phi( a, F(a))
	=
	\int_a^t 
	\big[
	\phi_{1,0}(s, F(s)) 
	+ 
	\phi_{0,1}(s, F(s)) f(s) 
	\big] \, ds 
	.
	\end{equation}
\end{enumerate}
\end{corollary}
\begin{proof}[Proof of Corollary~\ref{corollary:integral_chain_rule}]
	Throughout this proof let  
	$ \Phi \in \mathcal{C}^1(\R \times V, W) $
	be a function
	which satisfies for all
	$ t \in [a,b] $, $ x \in V $ 
	that
	$ \Phi(t,x) = \phi(t,x) $.
	Note that
	Lemma~\ref{lemma:integral_chain_rule}
	(with
	$ V = \R \times V $,
	$ W = W $, 
	$ a = a $,
	$ b = b $,
	$ \phi = \Phi $, 
	$ F = ( [a,b] \ni s \mapsto (s, F(s) ) \in \R \times V ) $,
	$ f = ( [a,b] \ni s \mapsto (1, f(s) ) \in \R \times V ) $
	in the notation of
	Lemma~\ref{lemma:integral_chain_rule}) 
	establishes items~\eqref{item:FirstIt}--\eqref{item:LastIt}.
	The proof of 
Corollary~\ref{corollary:integral_chain_rule}
is thus completed.
\end{proof}
\begin{lemma} 
	\label{lemma:ChangeOrder}  
Let 
	$ (V, \left\|\cdot \right\|_V) $ be an 
	$ \R $-Banach space,
	let
	$ a \in \R $,
	$ b \in (a, \infty) $,
	$ t_0 \in [a,b] $, 
	$ (f_n)_{n \in \N } \subseteq \mathcal{C}^1([a,b],V) $,
	and assume that 
	$ ( f_n(t_0) )_{n \in \N} \subseteq V $ 
	and  
	$ ( (f_n)' )_{n \in \N } \subseteq \mathcal{C}( [a,b], V ) $ 
	are convergent.
	Then 
	there exists $ F \in \mathcal{C}^1( [a,b], V ) $
	such that for every $ t \in [a,b] $ it holds that 
	\begin{equation}
	\begin{split}
	\limsup_{ n \to \infty }
	\big( 
	\| f_n(t) - F(t) \|_V
	+ 
	\| (f_n)'(t) - F'(t) \|_V
	\big)
	=
	0
	.
	\end{split}
	\end{equation}
\end{lemma}
\begin{proof}[Proof of Lemma~\ref{lemma:ChangeOrder}]
Throughout this proof let
$ F, g \colon [a,b] \to V $
be the functions which satisfy for all
$ t \in [a,b] $ that
$ \limsup_{ n \to \infty }
\sup_{ s \in [0,T] }
\| (f_n)'( s ) - g( s ) \|_V = 0 $,
$ F(t_0) = \lim_{n \to \infty } f_n(t_0) $,
and
\begin{equation} 
\label{eq:Defined} 
F(t) = F(t_0) + \int_{t_0}^t g(s) \, ds 
.
\end{equation}
Observe that
Lemma~\ref{lemma:FundamentaTheoremOfCalculus}
(with
$ V = V $,
$ a = a $,
$ b = b $,
$ F = f_n $
for $ n \in \N $
in the notation of
Lemma~\ref{lemma:FundamentaTheoremOfCalculus})
shows that for all
$ n \in \N $, $ t \in [a,b] $ it holds that
\begin{equation}
\label{eq:Fund}
f_n(t) = f_n(t_0) + \int_{t_0}^t (f_n)'(s) \, ds
.
\end{equation}
The assumption that
$ ( (f_n)' )_{n \in \N } \subseteq \mathcal{C}( [a,b], V ) $ 
converges ensures that
\begin{equation}
\sup\nolimits_{ n \in \N } 
\sup\nolimits_{ s \in [a,b] }
\| (f_n)'(s) \|_V < \infty.
\end{equation} 
The dominated convergence theorem 
therefore
proves that
for every $ t \in [a,b] $ 
it holds that
\begin{equation}
\begin{split}
\limsup_{ n \to \infty }
\Big\|
\int_{t_0}^t 
g(s) \, ds
- 
\int_{t_0}^t 
(f_n)'(s) \, ds
\Big\|_V =
0. 
\end{split}
\end{equation}
This and~\eqref{eq:Fund} imply that for all
$ t \in [a,b] $ it holds that
\begin{equation}
\begin{split}
\lim_{n \to \infty}
f_n(t)
=
F(t_0)
+
\int_{t_0}^t 
g(s) \, ds  
.
\end{split}
\end{equation}
Equation~\eqref{eq:Defined} 
hence assures for all
$ t \in [a,b] $ that
\begin{equation}
F(t) = \lim_{n \to \infty} f_n(t) .
\end{equation}
The fact that
$ g \in \mathcal{C}( [a,b], V ) $,
\eqref{eq:Defined},
and 
Lemma~\ref{lemma:extensions}
(with
$ V = V $,
$ a = a $,
$ b = b $,
$ f = g $,
$ F = F $
in the notation of Lemma~\ref{lemma:extensions})
establish that  
for every $ t \in [a,b] $ it holds that
$ F
\in 
\mathcal{C}^1( [a,b], V) $
and
$ F'(t) = g(t) $.
The proof
of Lemma~\ref{lemma:ChangeOrder}
is thus completed.
\end{proof}
\begin{proposition}
	\label{proposition:Banach_space_alekseev_grobner}
	Let 
	$ (V, \left\|\cdot \right\|_V) $ 
	be a nontrivial $ \R $-Banach space,
	let 
	$ t_0 \in \R $, 
	$ t \in (t_0, \infty) $, 
	$ \phi \in \mathcal{C}^1(V, V) $,  
	$ F \colon [t_0,t] \to V $,
	$ \Phi \colon \{ (u, r ) \in [t_0, t]^2 \colon u \leq r \} \times V \to V $,
	$ \dot \Phi \colon \{ (u, r ) \in [t_0, t]^2 \colon u \leq r \} \times V \to V $,
	$ \Phi^\star \colon \{ (u,r) \in [t_0,t]^2 \colon u\leq r \} \times V
	\to L(V) $,
let
$ f \colon [t_0, t] \to V $
be a strongly measurable function,
assume 
that
$ ( [t_0,t] \times V \ni (u,x) \mapsto \Phi_{u,t}(x) \in V )
\in \mathcal{C}^1( [t_0,t] \times V, V ) $, 
assume for all
$ x \in V $,
$ t_1 \in (t_0,t) $ 
that
$ ( [t_0, t_1] \ni u \mapsto \Phi_{u,t_1}(x) \in V ) \in \mathcal{C}^1( [t_0, t_1], V ) $,
$ ( [t_0, t_1] \ni u \mapsto \dot \Phi_{u, t_1}(x) \in V ) \in \mathcal{C} ( [t_0, t_1], V ) $,
$ ( \{(u,r)\in(t_0,t)^2 \colon u< r\} \ni (s, \tau) \mapsto
\Phi_{s, \tau}(x) \in V ) \in \mathcal{C}^1 ( \{(u,r)\in(t_0,t)^2 \colon u< r\}, V ) $,
 assume 
	for all  
	$ t_1 \in [t_0, t] $,
	$ t_2 \in [t_1, t] $,
	$ t_3 \in [t_2, t] $,
	$ x \in V $
	that 
	$ \int_{ t_0 }^t \| f( u ) \|_V \, du < \infty $,
	$ F(t_1) = F(t_0) + \int_{t_0}^{t_1} f(u)\, du $, 
	$ \Phi_{t_1, t_1}(x) = x $,
	$ \Phi_{t_1, t_3}(x) = \Phi_{t_2, t_3}( \Phi_{t_1, t_2}(x)) $, 
	assume for all
	$ t_1 \in (t_0,t) $,
	$ t_2 \in (t_1, t) $,
	$ x \in V $ 
	that 
	$ \dot \Phi_{t_1, t_2}(x) = \tfrac{\partial}{\partial t_2} ( \Phi_{t_1, t_2}(x) ) $,
	and assume for all 
	$ t_1 \in [t_0,t] $, 
	$ x \in V $
	that
	$ \Phi_{t_1, t }^\star(x) =
	\tfrac{\partial}{\partial x}
	( \Phi_{t_1, t } (x) ) $.
	Then 
	\begin{enumerate}[(i)]
	\item it holds that
	$ 
	 [t_0, t] \ni s \mapsto
	\phi' ( \Phi_{s, t} ( F(s) ) ) \Phi_{s, t}^\star  ( F(s) ) 
	  [ \dot \Phi_{s,s}  ( F(s) ) - f(s) ]  
	\in V   
	$
	is strongly measurable,
	\item it holds that
	$ \int_{ t_0 }^t 
	\| 
	\phi' ( \Phi_{s, t} ( F(s) ) ) \Phi_{s, t}^\star  ( F(s) ) 
	[ \dot \Phi_{s,s}  ( F(s) ) - f(s) ]   
	\|_V
	\, ds < \infty $,
	and
	\item it holds that
	\begin{equation}
	\begin{split}
	\label{eq:GeneralAlekseev} 
	&
	\phi ( \Phi_{t_0, t}  ( F(t_0) ) ) - \phi ( F(t) )
	=
	\int_{ t_0 }^t \phi'  ( \Phi_{s, t}   ( F(s)  )  ) \Phi_{s, t}^\star  ( F(s)   ) 
	\big[ \dot \Phi_{s,s}  ( F(s) ) - f(s)  \big] 
	\, 
	ds.
	\end{split}
	\end{equation}
\end{enumerate}
\end{proposition}
\begin{proof}[Proof
	of Proposition~\ref{proposition:Banach_space_alekseev_grobner}]
	Throughout this proof let 
	\begin{equation}
	\mathring \Phi
	=
	(
	\mathring \Phi_{t_1, t_2}(x)
	)_{ (t_1, t_2, x) 
\in 
\{ (u,r) \in [t_0, t]^2
\colon u \leq r \} \times V	
}
	\colon 
	\{ (u,r) \in [t_0, t]^2
	\colon u \leq r \} \times V
	\to 
	V
	\end{equation}
	be a function which satisfies
	for all  
	$ t_2 \in (t_0, t] $,
	$ t_1 \in [t_0, t_2] $,
	$ x \in V $ 
	that
	$ \mathring \Phi_{t_1, t_2}(x) = \tfrac{\partial}{\partial t_1} 
	\Phi_{t_1,t_2}(x) $
	and
	let
	$ \varphi \colon [ t_0, t ] \times V \to V $
	be the function
	which satisfies
	for all
	$ s \in [t_0, t] $,
	$ v \in V $
	that
	$ \varphi(s, v) = \Phi_{s, t}(v) $.
	Note that
	the assumption that
	$ ( [t_0,t] \times V \ni (s,x) \mapsto \Phi_{s,t}(x) \in V )
	\in \mathcal{C}^1( [t_0,t] \times V, V ) $
	shows that
	for every $ \tau \in [t_0, t) $
	it holds that
	$ \varphi |_{ [ \tau, t] \times V } \in \mathcal{C}^1( [\tau, t] \times V, V ) $.
	Corollary~\ref{corollary:integral_chain_rule}
	(with
	$ V = V $,
	$ W = V $,
	$ a = \tau $,
	$ b = t $, 
	$ \phi = \varphi |_{ [ \tau, t] \times V } $,
	$ F = F|_{ [\tau, t]} $,
	$ f = f|_{ [\tau, t]} $
	for $ \tau \in [t_0, t) $
	in the notation of
	Corollary~\ref{corollary:integral_chain_rule})
	therefore
	implies
	\begin{enumerate}[(a)]
		\item \label{item:integrability_condition1} that  
	$ [ t_0, t ] \ni s \mapsto 
	\big[
	\Phi^\star_{s, t }( F( s ) ) f( s ) 
	+
	\mathring \Phi_{s, t} ( F( s ) )
	\big] \in V $
	 is strongly measurable, 
	\item that
	$
	\int_{t_0}^t
	 \|
	\Phi^\star_{s, t }( F( s ) ) f( s ) 
	+
	\mathring \Phi_{s, t} ( F( s ) )
	 \|_V
	\, ds < \infty 
	$,
	and 
	\item that
	for every 
	$ \tau \in [t_0, t ] $
	it holds that  
	\begin{equation}
	\begin{split}
	\Phi_{ t , t  }( F( t  ) )
	-
	\Phi_{ \tau, t  }( F( \tau ) )
	=
	F( t  )
	-
	\Phi_{ \tau, t  }( F( \tau ) ) 
	=
	\int_{\tau}^t
	\big[
	\Phi^\star_{s, t }( F( s ) ) f( s ) 
	+
	\mathring \Phi_{s, t} ( F( s ) )
	\big]
	\,
	ds
	.
	\end{split}
	\end{equation}
\end{enumerate}
	Lemma~\ref{lemma:integral_chain_rule}
	(with 
	$ V = V $,
	$ W = V $,
	$ a = t_0 $,
	$ b = t $,
	$ \phi = \phi $,
	$ F = ( [t_0, t ] \ni s \mapsto \Phi_{s,t}( F(s) ) \in V ) $,
	$ f = ( [t_0, t ] \ni s \mapsto \Phi^\star_{s,t}( F( s ) ) f( s ) 
	+
	\mathring \Phi_{s,t} ( F( s ) ) \in V ) $ 
	in the notation of 
	Lemma~\ref{lemma:integral_chain_rule})  
	hence shows 
	\begin{enumerate}[(A)]
	\item \label{item:Banach_Alekseev_integrable}
	that
	$ [t_0, t] \ni s
	 \mapsto
	\phi'  ( \Phi_{s,t}  ( F(s)  ) )
	\big [ 
	\Phi_{s,t}^\star ( F(s) ) f(s) 
	+ 
	\mathring \Phi_{s,t} ( F(s) )
	\big]
	\in V $ 
	is strongly measurable, 
	\item that
	$ \int_{t_0}^t
	\big\|
	\phi'  ( \Phi_{s,t}  ( F(s)  ) )
	\big [ 
	\Phi_{s,t}^\star ( F(s) ) f(s) 
	+ 
	\mathring \Phi_{s,t} ( F(s) )
	\big]
	\big\|_V
	\, ds
	<
	\infty $,
	and
	\item \label{item:Banach_Alekseev_0} that
	\begin{equation} 
	\begin{split}
	& 
	\phi ( F(t) )
	-
	\phi ( \Phi_{ t_0, t} ( F( t_0 ) ) ) 
	= 
	\phi ( \Phi_{t, t} ( F(t) ) )     
	-
	\phi ( \Phi_{ t_0, t} ( F( t_0 )  )  )  
	\\
	&
	= 
	\int_{ t_0 }^t
	\phi'  ( \Phi_{s,t}  ( F(s)  ) )
	\big [ 
	\Phi_{s,t}^\star ( F(s) ) f(s) 
	+ 
	\mathring \Phi_{s,t} ( F(s) )
	\big]
	\,
	ds
	.
	\end{split}
	\end{equation}
	\end{enumerate}
	Next observe that 
		the assumption that
	$ ( [t_0,t] \times V \ni (s,x) \mapsto \Phi_{s,t}(x) \in V )
	\in \mathcal{C}^1( [t_0,t] \times V, V ) $
	and
	the chain rule 
	ensure that
	for all 
	$ u \in (t_0, t) $, 
	$ s \in (t_0, u] $, 
	$ x \in V $ 
	it holds that
	\begin{equation}
	\label{eq:Banach_Alekseev_1}
	\mathring \Phi_{s, t} (x) 
	= 
	\tfrac{ \partial }{\partial s}  
	\Phi_{s, t}(x)  
	=
	\tfrac{ \partial }{\partial s} 
	 ( 
	\Phi_{u, t}
	 ( \Phi_{s,u} (x) ) 
	 )
	=
	\Phi_{u,t}^\star
	 ( \Phi_{s, u}(x)  ) 
	\mathring
	\Phi_{s, u}(x).
	\end{equation}
	Moreover, note that
	the assumption that
	for all
	$ x \in V $
	it holds that
	$ \{(u,r)\in(t_0,t)^2:u< r\} \ni (s, \tau) \mapsto
	\Phi_{s, \tau}(x) \in V $
	is continuously differentiable
	implies that for all 
	$ s \in (t_0, t) $,
	$ n \in [2, \infty) \cap \N $,
	$ x \in V $
	it holds that
	\begin{equation}
	\begin{split}
	\tfrac{ \partial }{ \partial s }
	\Phi_{s-(s-t_0)/n, s}(x)
	=
	( 1 - \nicefrac{1}{n} )
	\mathring \Phi_{s-(s-t_0)/n, s}(x)
	+
	\dot \Phi_{s-(s-t_0)/n,s}(x)
	.
	\end{split}
	\end{equation}
	Combining the fact that 
	for all
	$ \varepsilon \in (0, \nicefrac{ (t-t_0) }{ 6 } ) $,
	$ n \in [2, \infty) \cap \N $,
	$ x \in V $
	it holds that
	$ [t_0 + \varepsilon, t- \varepsilon] \ni s \mapsto \Phi_{ s - ( s - t_0)/n, s }(x) \in V $
	is continuously differentiable,
	Lemma~\ref{lemma:ChangeOrder}
	(with
	$ V = V $,
	$ a = t_0 + \varepsilon $,
	$ b = t - \varepsilon $,
	$ t_0 = \nicefrac{ ( t_0 + t ) }{ 2 } $, 
	$ f_n = ( [t_0 + \varepsilon, t- \varepsilon] \ni s \mapsto \Phi_{ s - ( s - t_0)/n, s }(x) \in V ) $
	for  
	$ \varepsilon \in (0, \nicefrac{ ( t - t_0 ) }{ 6 } ) $,
	$ n \in [2, \infty) \cap \N $,
	$ x \in V $
	in the notation of Lemma~\ref{lemma:ChangeOrder}),
	and the assumptions that 
	$ \forall \, x \in V $, $ t_1 \in (t_0,t) 
	\colon 
	( [t_0, t_1] \ni u \mapsto \Phi_{u,t_1}^x \in V ) \in \mathcal{C}^1( [t_0, t_1], V ) $
	and
	$ ( [t_0, t_1] \ni u \mapsto \dot \Phi_{u, t_1}^x \in V ) \in \mathcal{C} ( [t_0, t_1], V ) $ 
	therefore proves
	that
	for all
	$ s \in (t_0, t) $, 
	$ x \in V $
	it holds
	that
	\begin{equation}
	\begin{split}
	\tfrac{ \partial }{ \partial s }
	\Phi_{s, s}(x)
	=
	\mathring \Phi_{s, s}(x)
	+
	\dot \Phi_{s,s}(x)
	.
	\end{split}
	\end{equation}
	Hence, we obtain that
	for all $ s \in (t_0, t) $, $ x \in V $ it holds that
	\begin{equation}
	\mathring \Phi_{s,s}(x)
	+
	\dot \Phi_{s,s}(x)
	=
	0
	.
	\end{equation}
	This and~\eqref{eq:Banach_Alekseev_1} 
	imply that for all $ s \in (t_0, t) $, $ x \in V $ 
	it holds that
	\begin{equation}
	\label{eq:Banach_Alekseev_2}
	\mathring \Phi_{s,t}(x) 
	=
	\Phi_{s,t}^\star ( \Phi_{s,s} (x) ) \mathring \Phi_{s,s}(x)
	=
	\Phi_{s,t}^\star(x) \mathring \Phi_{s,s}(x)
	=
	- \Phi_{s,t}^\star(x) \dot \Phi_{s,s}(x)
	.
	\end{equation} 
	Combining this with items~\eqref{item:Banach_Alekseev_integrable}--\eqref{item:Banach_Alekseev_0}   
completes the proof of
	Proposition~\ref{proposition:Banach_space_alekseev_grobner}. 
\end{proof}
\section{Continuity of solutions to initial value problems}
\label{subsection:Continuity of solutions}
In this section we prove in Corollary~\ref{Corollary:UnifCont} joint continuity of the solution to a Banach space valued ODE 
with respect to initial value, initial time, and current time.
More precisely,
we first apply
Lemma~\ref{lemma:uniqueness} 
to prove a local existence and uniqueness result
for initial value problems in
Lemma~\ref{lemma:Existence}. 
Then
we combine 
Lemma~\ref{lemma:Existence},
Lemma~\ref{lemma:trivial}, 
Corollary~\ref{corollary:Extension},
and
Lemmas~\ref{lemma:UnifCont}--\ref{lemma:flow_property}
to establish
Corollary~\ref{Corollary:UnifCont}.
\begin{lemma}
	\label{lemma:uniqueness}
	Let
	$ (V, \left\| \cdot \right\|_V) $ 
	be a nontrivial $ \R $-Banach space 
	and
	let
	$ a \in \R $,
	$ b \in [a, \infty) $, 
	$ s \in [a, b] $,  
	$ f \in \mathcal{C} ( [a,b] \times V, V) $, 
	$ X, Y \in \mathcal{C}([a, b], V) $
	satisfy for all  
	$ t \in [a,b ] $, 
	$ x \in V $
	that
	$ X_t -
	\int_s^t f( \tau, X_\tau) \,d\tau 
	= 
	Y_t
	-
	\int_s^t f( \tau, Y_\tau) \,d\tau $
	and
		$ 
		\inf_{ r \in (0, \infty) } 
		\sup_{ \tau \in [a,b] } 
		\sup_{  y \in V, \| x - y \|_V \leq r }
		\sup_{ 
			z \in V \backslash \{ y \}, 
			\| x - z \|_V \leq r  }   
	\tfrac{
		\| f(\tau, y) - f( \tau, z ) \|_V
	}{
		\| y - z \|_V
	}
	< \infty $.
	Then it holds for all
	$ t \in [a, b] $ that
	$ X_t = Y_t $. 
\end{lemma}
\begin{proof}[Proof of Lemma~\ref{lemma:uniqueness}]
	Throughout this proof let
	$ L_{x,r} \in [0,\infty] $, 
	$ x \in V $,
	$ r \in (0, \infty) $,
	be the extended real numbers which satisfy for all
	$ r \in (0, \infty) $,
	$ x \in V $ 
	that
	\begin{equation}
	L_{x,r} =
	\sup\nolimits_{ \tau \in [a, b] } 
		\sup\nolimits_{  y \in V, \| x - y \|_V \leq r }
		\sup\nolimits_{ 
			z \in V \backslash \{ y \}, 
			\| x - z \|_V \leq r  }  
	\tfrac{
		\| f(\tau, y) - f( \tau, z ) \|_V
	}{
		\| y - z \|_V
	} 
,
	\end{equation}
	let
	$ \alpha = \sup ( \{ a \} \cup \{ u \in [a,s] \colon X_u \neq Y_u \} ) $,
	and let
	$ \beta = \inf ( \{ b \} \cup \{ u \in [s,b] \colon X_u \neq Y_u \} ) $.
	Observe that 
	the hypothesis that
	$ X \colon [a,b] \to V $
	and
	$ Y \colon [a,b] \to V $
	are continuous functions ensures that
	there exists 
	a function
	$ \delta \colon [a, b] \times (0, \infty) \to (0, \infty) $
	such that for all
	$ u \in [a, b] $,
	$ \varepsilon \in (0,\infty) $, 
	$ t \in [ u - \delta_{u,\varepsilon}, u + \delta_{u,\varepsilon} ] \cap [a,b] $
	it holds that
	\begin{equation} 
	\| X_t - X_u \|_V < \varepsilon 
	\qquad \text{and} \qquad
	\| Y_t - Y_u \|_V < \varepsilon 
	.
	\end{equation} 
	This implies that for all
	$ u \in [a, b] $
	with 
	$ X_u = Y_u $ 
	there exists 
	$ \varepsilon \in (0, \infty) $
	with
	$ L_{X_u, \varepsilon } < \infty $
	such that for all
	$ t \in 
	[ u - \min \{ \delta_{u, \varepsilon}, 
	\nicefrac{1}{ ( 1 + 2 L_{X_u, \varepsilon} ) } \}, 
	u + \min \{ \delta_{u, \varepsilon}, 
	\nicefrac{1}{ ( 1 + 2 L_{ X_u, \varepsilon } ) }  \} ] 
	\cap [a,b] $ 
	it holds that
	\begin{equation}
	\begin{split}
	&
	\| X_t - Y_t \|_V
	=
	\| ( X_t - Y_t ) - (X_u - Y_u) \|_V
	\\
	&
	=
	\Big\|
	\Big[ 
	\int_s^t
	f( \tau, X_\tau ) \, d\tau 
	-
	\int_s^t 
	f(\tau, Y_\tau ) \, d \tau
	\Big] 
	-
	\Big[
	\int_s^u f( \tau, X_\tau ) \, d \tau 
	-
	\int_s^u f( \tau, Y_\tau ) \, d \tau 
	\Big]
	\Big\|_V
	\\
	&
	=
	\Big\| 
	\int_u^t f( \tau, X_\tau ) \, d \tau 
	-
	\int_u^t f( \tau, Y_\tau ) \, d \tau 
	\Big\|_V
	\\
	&
	=
	\Big\|
	\int_{ \min \{ u, t \} }^{ \max \{ u, t \} } 
	f( \tau, X_{\tau} ) \, d \tau 
	-
	\int_{ \min \{ u, t \} }^{ \max \{ u, t \} } 
	f( \tau, Y_{\tau} ) \, d \tau
	\Big\|_V
	\\
	&
	\leq
	\int_{ \min \{ u, t \} }^{ \max \{ u, t \} }
	\| 
	f( \tau, X_\tau) 
	- 
	f( \tau, Y_\tau ) 
	\|_V 
	\, d \tau
	\leq
	L_{X_u, \varepsilon}
	\int_{ \min \{ u, t \} }^{ \max \{ u, t \} }
	\| X_\tau - Y_\tau \|_V \, d \tau
	\\
	&
	\leq
	L_{X_u, \varepsilon}
	\,
	| t - u |
	\bigg[ 
	\sup_{ \tau \in [ \min \{ u,t \}, 
	\max \{ u, t \} ] }
	\| X_\tau - Y_\tau \|_V
	\bigg]
	\\
	&
	\leq
	L_{X_u, \varepsilon}
	\,
| t - u |
\bigg[ 
	\sup_{ \tau \in [ u - \min \{ \delta_{u, \varepsilon}, 
		\nicefrac{1}{ ( 1 + 2 L_{X_u, \varepsilon} ) } \}, 
		u + \min \{ \delta_{u, \varepsilon}, 
		\nicefrac{1}{ ( 1 + 2 L_{ X_u, \varepsilon } ) }  \} ] 
\cap [a,b]	
}
	\| X_\tau - Y_\tau \|_V
	\bigg]
	\\
	&
	\leq 
	\frac{ L_{X_u, \varepsilon} }{ 1 + 2 L_{X_u, \varepsilon} }
	\bigg[ 
	\sup_{ \tau \in [ u - \min \{ \delta_{u, \varepsilon}, 
		\nicefrac{1}{ ( 1 + 2 L_{X_u, \varepsilon} ) } \}, 
		u + \min \{ \delta_{u, \varepsilon}, 
		\nicefrac{1}{ ( 1 + 2 L_{ X_u, \varepsilon } ) }  \} ] 
\cap [a,b]	
}
	\| X_\tau - Y_\tau \|_V
	\bigg]
	\\
	&
	\leq
	\frac{1}{2}
	\bigg[ 
	\sup_{ \tau \in [ u - \min \{ \delta_{u, \varepsilon}, 
		\nicefrac{1}{ ( 1 + 2 L_{X_u, \varepsilon} ) } \}, 
		u + \min \{ \delta_{u, \varepsilon}, 
		\nicefrac{1}{ ( 1 + 2 L_{ X_u, \varepsilon } ) }  \} ]
\cap [a,b]	
 }
	\| X_\tau - Y_\tau \|_V
	\bigg] 
	.
	\end{split}
	\end{equation}
	Hence, we obtain that for all $ u \in [a,b] $ 
	with $ X_u = Y_u $ 
	there exists $ \varepsilon \in (0, \infty) $
	with
	$ L_{ X_u, \varepsilon } < \infty $
	such that
	\begin{equation}
	\begin{split} 
	&
		\bigg[ 
	\sup_{ \tau \in [ u - \min \{ \delta_{u, \varepsilon}, 
		\nicefrac{1}{ ( 1 + 2 L_{X_u, \varepsilon} ) } \}, 
		u + \min \{ \delta_{u, \varepsilon}, 
		\nicefrac{1}{ ( 1 + 2 L_{ X_u, \varepsilon } ) }  \} ] 
\cap [a,b]	
}
	\| X_\tau - Y_\tau \|_V
	\bigg] 
	\\
	&
	\leq
	\frac{1}{2}
	\bigg[ 
	\sup_{ \tau \in [ u - \min \{ \delta_{u, \varepsilon}, 
		\nicefrac{1}{ ( 1 + 2 L_{X_u, \varepsilon} ) } \}, 
		u + \min \{ \delta_{u, \varepsilon}, 
		\nicefrac{1}{ ( 1 + 2 L_{ X_u, \varepsilon } ) }  \} ] 
\cap [a,b]	
}
	\| X_\tau - Y_\tau \|_V
	\bigg] 
	.
	\end{split} 
	\end{equation}
	This shows that
	for all $ u \in [a,b] $ 
	with $ X_u = Y_u $ 
	there exists $ \varepsilon \in (0, \infty) $
	with
	$ L_{ X_u, \varepsilon } < \infty $
	such that
	\begin{equation}
	\begin{split} 
	\bigg[ 
	\sup_{ \tau \in [ u - \min \{ \delta_{u, \varepsilon}, 
		\nicefrac{1}{ ( 1 + 2 L_{X_u, \varepsilon} ) } \}, 
		u + \min \{ \delta_{u, \varepsilon}, 
		\nicefrac{1}{ ( 1 + 2 L_{ X_u, \varepsilon } ) }  \} ] 
\cap [a,b]	
}
	\| X_\tau - Y_\tau \|_V
	\bigg] 
	=
	0.
	\end{split} 
	\end{equation}
	Therefore, we obtain for all
	$ u \in [0,T] $
	with 
	$ X_u = Y_u $ 
	that
	there exists $ \Delta \in (0, \infty) $
	such that for all
	$ t \in [ u - \Delta, u + \Delta ] \cap [a,b] $ 
	it holds that
%
	\begin{equation} 
	\label{eq:LocUnique}
	X_t = Y_t 
	.
	\end{equation}
	Moreover, 
	observe that
	the fact that 
	$ X $ and $ Y $ are continuous ensures that
	$ X_\alpha = Y_\alpha $
	and 
	$ X_\beta = Y_\beta $.
	Combining this with~\eqref{eq:LocUnique}
	demonstrates
	that
	$ \alpha = a $
	and
	$ \beta = b $.
	The proof of Lemma~\ref{lemma:uniqueness}
	is thus completed.
\end{proof} 
\begin{lemma}
	\label{lemma:Existence} 
	Let
	$ (V, \left\| \cdot \right\|_V) $ 
	be a nontrivial $ \R $-Banach space and
	let 
	$ R, h, \varepsilon \in (0,\infty) $, 
	$ s_0 \in \R $, 
	$ L, M, \delta \in [0,\infty) $,   
	$ x_0 \in V $, 
	$ f \in \mathcal{C} ( \R \times V, V) $ 
	satisfy 
	for 
	all 
	$ x \in V $   
	that  
	\begin{equation} 
	\inf\nolimits_{ r \in (0, \infty) } 
	\sup\nolimits_{ \tau \in [ s_0 - h, s_0 + h ] } 
	\sup\nolimits_{  y \in V, \| x - y \|_V \leq r }
	\sup\nolimits_{
			z \in V \backslash \{ y \}, 
			\| x - z \|_V \leq r  }  
	\tfrac{
		\| f(\tau, y) - f( \tau, z ) \|_V
	}{
		\| y - z \|_V
	}
	< \infty 
	,
	\end{equation} 
	\begin{equation} 
	\label{eq:Define_L}
	 L = \sup\nolimits_{ \tau \in [s_0 - h, s_0 + h] } 
	 \sup\nolimits_{  y \in V,  \| x_0 - y \|_V \leq R + \varepsilon }
	 \sup\nolimits_{ z \in V \backslash \{ y \},
			\| x_0 - z \|_V \leq R + \varepsilon  } 
	\tfrac{
		\| f(\tau, y) - f( \tau, z ) \|_V
	}{
		\| y - z \|_V
	},  
	\end{equation}
	\begin{equation} M = 
	\sup\nolimits_{ \tau \in [s_0-h, s_0+h] } 
	\sup\nolimits_{
		y  \in V, \| x_0 - y \|_V \leq R + \varepsilon  
	} 
	\| f(\tau, y) \|_V 
	,
	\end{equation}
	and
	$ \delta = 
	\min \{ 
	\nicefrac{ \varepsilon }{ ( 2 M + 1 ) }, 
	\nicefrac{ 1 }{ ( 4 L + 1 ) },
	h 
	\} $.
	Then 
	\begin{enumerate}[(i)]
		\item \label{item:existence} 
	it holds for every
	$ s \in [s_0 - \delta,
	s_0 + \delta ] $,
	$ x \in 
	\{ v \in V \colon \| x_0 - v \|_V \leq R \} $
	that
	there exists a unique continuous function 
	$ X_{ s, (\cdot) }^x
	=
	( X_{s,t}^x )_{ t \in [s_0-\delta, s_0+\delta]}
	 \colon 
	[s_0 - \delta, s_0 + \delta ]
	\to 
	V $
	such that
	for all
	$ t\in 
	[s_0 - \delta, s_0 + \delta ] $
	it holds that
	\begin{equation}
	\label{eq:some_equation}
	X_{s,t}^x
	= 
	x
	+
	\int_s^t f( \tau, X_{s,\tau}^x) 
	\, d \tau 
	\end{equation}
	and
	\item \label{item:Bound}
	it
	holds  
	that
	\begin{equation}
	\sup\nolimits_{ s, t \in [s_0- \delta, s_0 + \delta ] }
	\sup\nolimits_{
	x \in V, \| x_0 - x \|_V \leq R 	
	}
	\| X_{s,t}^x - x_0 \|_V \leq R + \varepsilon 
	.
	\end{equation}
\end{enumerate}
\end{lemma}
\begin{proof}[Proof of Lemma~\ref{lemma:Existence}]
	Throughout this proof 
	let
	$ \mathcal{A} $ be the set given by
	\begin{equation}
	\label{eq:Define_A}
	\mathcal{A} 
	= 
	\big\{ \psi \in 
	\mathcal{C}( [s_0 - \delta,
	 s_0 + \delta ], V ) \colon
	 \sup\nolimits_{ t \in [s_0 - \delta,
	 	s_0 + \delta ] }
	 \| \psi(t) - x_0 \|_V \leq R + \varepsilon 
	\big \}
	.
	\end{equation}
	Note that for all
	$ s, t \in 
	[s_0 - \delta,
	s_0 + \delta ] $,
	$ x \in
	\{ v \in V \colon \| x_0 - v \|_V \leq 
	R \} $, 
	$ \psi \in \mathcal{A} $
	it holds that
	\begin{equation}
	\begin{split}
	&
	\Big\|
	x + \int_s^t f( \tau, \psi( \tau)  ) \, d \tau - x_0
	\Big\|_V
	\leq
	\| x - x_0 \|_V
	+
	\int_{ \min \{ s, t \} }^{ \max \{s, t\} }
	\| f( \tau, \psi(\tau) ) \|_V
	\, d \tau 
	\\
	&
	\leq
	R
	+
	M | t -s |
	\leq
	R
	+
	2 M \delta 
	\leq
	R + ( 2 M + 1 ) \delta
	\leq 
    R +	\varepsilon.
	\end{split}
	\end{equation}
	This ensures that there exist functions
	$ B_{s,x} \colon 
	\mathcal{A} \to \mathcal{A} $,
	$ s \in [s_0 - \delta,
	s_0 + \delta ] $,
	$ x \in \{ v \in V \colon \| x_0 - v \|_V \leq R \} $,
	such that for all
	$ s, t \in [s_0 - \delta,
	s_0 + \delta ] $,
	$ x \in \{ v \in V \colon \| x_0 - v \|_V \leq R \} $, 
	$ \psi \in \mathcal{A} $ 
	it holds that
	\begin{equation}
	( B_{s,x}( \psi ) ) (t) = x 
	+
	\int_s^t
	f( \tau, \psi( \tau) ) \, d \tau.
	\end{equation}
	Next observe that~\eqref{eq:Define_L}
	and~\eqref{eq:Define_A}
	demonstrate	
	that for all
	$ s, t \in [s_0 - \delta,
	s_0 + \delta ] $, 
	$ x \in \{ v \in V \colon \| x_0 - v \|_V \leq R \} $, 
	$ \psi_1, \psi_2 \in \mathcal{A} $
	it holds that
	\begin{equation}
	\begin{split}
	&
	\| 
	( B_{s,x} ( \psi_1 ) ) ( t )
	-
	( B_{s,x}( \psi_2 ) ) ( t )
	\|_V
	\leq
	\Big\|
	\int_s^t
	[ 
	f( \tau, \psi_1( \tau ) )
	-
	f( \tau, \psi_2( \tau) )
	]
	\, d \tau
	\Big\|_V
	\\
	&
	\leq
	\int_{ \min \{ s, t \} }^{ \max \{ s, t \} }
	\|
	f( \tau, \psi_1( \tau ) )
	-
	f( \tau, \psi_2( \tau) )
	\|_V
	\, d \tau
	\leq
	L 
	\int_{ \min \{ s, t \} }^{ \max \{ s, t \} }
	\| \psi_1( \tau ) - \psi_2( \tau ) \|_V
	\, d \tau 
	\\
	&
	\leq
	L  | t - s | 
	\bigg[ 
	\sup_{ \tau \in [s_0 - \delta, s_0 + \delta ] }
	\| \psi_1( \tau ) - \psi_2( \tau ) \|_V
	\bigg]
	\leq 
	2 L \delta 
	\bigg[ 
	\sup_{\tau \in [s_0 - \delta, s_0 + \delta ]}
	\| \psi_1( \tau ) - \psi_2( \tau ) \|_V
	\bigg]
	\\
	&
	\leq 
	\frac{( 4 L + 1 ) \delta}{2}
	\bigg[ 
	\sup_{ \tau \in [s_0 - \delta, s_0 + \delta ] }
	\| \psi_1( \tau ) - \psi_2( \tau ) \|_V
	\bigg]
	\leq 
	\frac{1}{2}
	\bigg[ 
	\sup_{ \tau \in [s_0 - \delta, s_0 + \delta ] }
	\| \psi_1( \tau ) - \psi_2( \tau ) \|_V
	\bigg]
	.
	\end{split}
	\end{equation}
	This shows that for all
	$ s \in [s_0 - \delta, s_0 + \delta ] $,
	$ x \in \{ v \in V \colon \| x_0 - v \|_V \leq R \} $,
	$ \psi_1, \psi_2 \in \mathcal{A} $
	it holds that
	\begin{equation}
	\begin{split}
	\bigg[
	\sup_{ t \in [s_0 - \delta, s_0 + \delta ] }
	\| 
	( B_{s,x}( \psi_1 ) )( t ) 
	-
	( B_{s,x}( \psi_2 ) )( t ) 
	\|_V
	\bigg]
	\leq
	\frac{1}{2}
	\bigg[ 
	\sup_{ t \in [s_0 - \delta, s_0 + \delta ] }
	\| \psi_1( t ) - \psi_2( t ) \|_V
	\bigg]
	.
	\end{split}
	\end{equation}
	Banach's fixed point theorem
	hence proves that
	there exist continuous functions 
	$ X_{s,(\cdot)}^x 
	=
	( X_{s,t}^x )_{ t \in [s_0-\delta, s_0+\delta] }
	\colon 
	[s_0 - \delta,
	s_0 + \delta ]
	\to V $, 
	$ s \in [s_0 - \delta,
	s_0 + \delta ] $,
	$ x \in	\{ v \in V \colon \| x_0 - v \|_V \leq R \} $, 
	such that for all
	$ s,t \in [s_0 - \delta,
	s_0 + \delta ] $,
	$ x \in	\{ v \in V \colon \| x_0 - v \|_V \leq R \} $ 
	it holds that
	\begin{equation}
	X_{s,t}^x = x + \int_s^t f( \tau, X_{s,\tau}^x ) \, d \tau 
	\end{equation}
	and
	\begin{equation}
	\| X_{s,t}^x - x_0 \|_V \leq R + \varepsilon.
	\end{equation}
	Combining this and
	Lemma~\ref{lemma:uniqueness} 
	(with
	$ V = V $,
	$ a = s_0 - \delta $, 
	$ b = s_0 + \delta $,
	$ s = s $, 
	$ f = ( [s_0 - \delta, s_0 + \delta] \times V \ni ( \tau, y ) 
	\mapsto f(\tau, y) \in V ) $,
	$ X = ( [s_0 - \delta, s_0 + \delta] \ni t \mapsto X_{s,t}^x \in V ) $
	for 
	$ s \in [s_0 - \delta, s_0 + \delta ] $,
	$ x \in \{ v \in V \colon \| x_0 - v \|_V \leq R \} $
	in the notation of
	Lemma~\ref{lemma:uniqueness})
	completes the proof of Lemma~\ref{lemma:Existence}.
\end{proof}
\begin{lemma}
	\label{lemma:trivial}
	Let
	$ (V, \left\| \cdot \right\|_V) $ 
	be a nontrivial $ \R $-Banach space
	and 
	let 
	$ a \in \R $,
	$ b \in (a, \infty ) $,  
	$ r \in (0, \infty) $,
	$ x \in V $,
	$ f \in \mathcal{C}( [a, b] \times V, V) $
	satisfy  
	that
	\begin{equation} 
	\label{eq:Assumption}
	\sup\nolimits_{ \tau \in [a, b] } 
	\sup\nolimits_{ y \in V, \| x - y \|_V \leq r }
	\sup\nolimits_{ z \in V \backslash \{ y \}, 
			\| x - z \|_V \leq r } 
	\tfrac{
		\| f(\tau, y) - f( \tau, z ) \|_V
	}{
		\| y - z \|_V
	}
	< \infty 
	.
	\end{equation}
	Then   
	\begin{equation} 
	\sup\nolimits_{ \tau \in [a, b] }
	\sup\nolimits_{ y \in V,  \| x - y \|_V \leq r }
	\| f ( \tau, y ) \|_V <  \infty 
	.
	\end{equation}
\end{lemma}
\begin{proof}[Proof of Lemma~\ref{lemma:trivial}]
	Note
	that~\eqref{eq:Assumption}
	and the hypothesis that
	$ f \colon [a,b] \times V \to V $
	is a continuous function 
	ensure that 
	\begin{equation}
	\begin{split}
	&
	\sup\nolimits_{ \tau \in [a,b] }
	\sup\nolimits_{ y \in V, \| x - y \|_V \leq r }
	\| f ( \tau, y ) \|_V
	\\
	&
	\leq
	\big[ 
	\sup\nolimits_{ \tau \in [a,b] }
	\sup\nolimits_{ y \in V, \| x - y \|_V \leq r }
	\| f( \tau, y ) - f ( \tau, x ) \|_V
	\big]
	+
	\big[ 
	\sup\nolimits_{ \tau \in [a,b] }
	\| f ( \tau, x ) \|_V
	\big]
	\\
	&
	=
	\big[
	\sup\nolimits_{ \tau \in [a,b] }
	\sup\nolimits_{ y \in V \backslash \{x\},
		\| x - y \|_V \leq r }
	\| f( \tau, y ) - f( \tau, x ) \|_V 
	\big] 
	+
	\big[ 
	\sup\nolimits_{ \tau \in [a,b] }
	\| f ( \tau, x ) \|_V
	\big]
	\\
	&
	\leq 
	r
	\big[
	\sup\nolimits_{ \tau \in [a,b] }
	\sup\nolimits_{ y \in V \backslash \{ x \}, 
		\| x - y \|_V \leq r }
	\big( 
	\tfrac{ \| f( \tau, y ) - f( \tau, x ) \|_V }
	{ \| y - x \|_V }
	\big)
	\big]
	+
	\big[ 
	\sup\nolimits_{ \tau \in [a,b] }
	\| f ( \tau, x ) \|_V
	\big]
	\\
	&
	\leq
	r
	\bigg[ 
	\sup_{ \tau \in [a,b] }
	\sup_{   z \in V, \| x - z \|_V \leq r }
	\sup_{ 
			y \in V \backslash \{ z \},
			\| x - y \|_V \leq r 
		 } 
	\big(
	\tfrac{
		\| f(\tau, y) - f( \tau, z ) \|_V
	}{
		\| y - z \|_V
	} 
	\big)
	\bigg]
	+
	\big[ 
	\sup\nolimits_{ \tau \in [a,b] }
	\| f ( \tau, x ) \|_V
	\big] 
	< \infty
	.
	\end{split}
	\end{equation}
	The proof of Lemma~\ref{lemma:trivial}
	is thus completed.
\end{proof}
\begin{corollary}
	\label{corollary:Extension} 
	Let
	$ (V, \left\| \cdot \right\|_V) $ 
	be a nontrivial $ \R $-Banach space,  
	let
	$ T, R, \varepsilon \in (0,\infty) $,  
	$ s_0 \in [0,T] $,
	$ x_0 \in V $,
	$ f \in \mathcal{C} ( [0,T] \times V, V) $,  
	for every
	$ x \in V $,
	$ s \in [0,T] $
	let
	$ X_{s, (\cdot)}^x 
	=
	( X_{s,t}^x )_{ t \in [s,T] } 
	\colon 
	[s,T] \to V $
	be a continuous function
	which satisfies 
	for all
	$ t \in [s,T] $ that
$ X_{s, t}^x = 
x
+
\int_s^t f( \tau, X_{s, \tau}^x) \,d\tau $,
and assume for all
$ x \in V $ 
that
	$ \sup_{ \tau \in [0,T] } 
	\sup_{ y \in V, \| x_0 - y \|_V \leq R+\varepsilon } 
	\sup_{ z \in V \backslash \{ y \},
			\| x_0 - z \|_V \leq R+\varepsilon }  
	\tfrac{
		\| f(\tau, y) - f( \tau, z ) \|_V
	}{
		\| y - z \|_V
	}
	< \infty $
	and
	$
	\inf_{ r \in (0, \infty) } 
	\sup_{ \tau \in [0, T] } 
	\sup_{  y \in V, \| x - y \|_V \leq r }
	\sup_{ 
			z \in V \backslash \{ y \}, 
			\| x - z \|_V \leq r  }   
	\tfrac{
		\| f(\tau, y) - f( \tau, z ) \|_V
	}{
		\| y - z \|_V
	}
	< \infty 
$.
	Then 
	\begin{enumerate}[(i)]
		\item \label{item:extension}
	there exists 
	$ \delta \in (0, \infty) $
	such that for every
	$ s \in [s_0 - \delta, s_0 + \delta ] \cap [0,T] $,
	$ x \in \{ v \in V \colon \| x_0 - v \|_V \leq R \} $ 
	there exists a unique continuous function
	$ Y_{s,(\cdot)}^x
	=
	(Y_{s,t}^x)_{ t \in  [ s_0 - \delta, s_0 + \delta ] \cap [0,T] }
	\colon
	 [ s_0 - \delta, s_0 + \delta ] \cap [0,T]
	 \to 
	 V $ 
	such that for all 
	$ t \in [s_0 - \delta, s_0 + \delta ] \cap [0,T] $
	it holds that  
	\begin{equation}
    Y_{s,t}^x 
	=
	x
	+
	\int_s^t f( \tau, Y_{s, \tau}^x ) \, d \tau
	, 
	\end{equation}
	\item 
	\label{item:uniqueness}
	for all
	$ s, t \in [ s_0 - \delta, s_0 + \delta ] \cap [0,T] $,
	$ x \in \{ v \in V \colon \| x_0 - v \|_V \leq R \} $
	with
	$ s \leq t $
	it holds that
	$ X_{s,t}^x = Y_{s,t}^ x $,
	and
	\item\label{item:unif_bound} 
	for all 
	$ s \in [s_0 - \delta, s_0 + \delta ] \cap [0,T] $,
	$ x \in \{ v \in V \colon \| x_0 - v \|_V \leq R \} $ 
	it holds that
	\begin{equation}
	\sup\nolimits_{ t \in [s_0- \delta, s_0 + \delta ] }
	\| Y_{s,t}^x - x_0 \|_V \leq R + \varepsilon 
	.
	\end{equation}
\end{enumerate}
\end{corollary}
\begin{proof}[Proof of Corollary~\ref{corollary:Extension}]
	Throughout this proof let  
	$ F \in \mathcal{C}(\R \times V, V) $
	be the function
	which satisfies for all 
	$ t \in \R $, 
	$ x \in V $ 
	that 
\begin{equation} 
 F(t,x) = f( \min \{ T, \max \{ 0, t \} \}, x ),
\end{equation}
	let
		$ L \in [0, \infty) $
	be the real number 
	given by
	\begin{equation}
	L =
	\sup\nolimits_{ \tau \in [0, T] }
	 \sup\nolimits_{ 
		 y \in V, \| x_0 - y \|_V \leq R + \varepsilon }
	 \sup\nolimits_{
			z \in V \backslash \{ y \},
			\| x_0 - z \|_V \leq R + \varepsilon }   
	\tfrac{
		\| F(\tau, y) - F( \tau, z ) \|_V
	}{
		\| y - z \|_V
	} 
	,
	\end{equation}
	let
	$ M \in [0, \infty] $
	be the extended real number given by
	\begin{equation}
	M =
	\sup\nolimits_{  
		\tau \in [0,T]}
	\sup\nolimits_{
		y \in V, \| x_0 - y \|_V \leq R + \varepsilon  
	}
	\| F(\tau, y) \|_V
	,
	\end{equation}
	and let
	$ \delta \in [0, \infty) $
	be the real number
	given by
	  $ \delta = 
	  \min \{ 
	  \nicefrac{ \varepsilon }{ ( 2 M + 1 ) }, 
	  \nicefrac{ 1 }{ ( 4 L + 1 ) },
	  1 
	  \} $.
	Note that
Lemma~\ref{lemma:trivial}
(with
$ V = V $,
$ a = 0 $,
$ b = T $,
$ r = R + \varepsilon $,
$ x = x_0 $,
$ f = F $
in the notation of Lemma~\ref{lemma:trivial})
proves that 
$ M < \infty $.
This ensures that 
$ \delta \in (0,\infty) $.
	Combining this, 
	the fact that $ M < \infty $, 
	and
	item~\eqref{item:existence} 
	of
	Lemma~\ref{lemma:Existence}
	(with
	$ V = V $, 
	$ R = R $,
	$ h = 1 $,
	$ \varepsilon = \varepsilon $,
	$ s_0 = s_0 $,
	$ L = L $,
	$ M = M $,
	$ \delta = \delta $, 
	$ x_0 = x_0 $, 
	$ f = F $ 
	in the notation of
	item~\eqref{item:existence} 
	of
	Lemma~\ref{lemma:Existence})
	establishes item~\eqref{item:extension}.
	The fact that 
	$ \forall \, 
	s \in [s_0- \delta, s_0 + \delta] \cap [0,T],  
	x \in \{ v \in V \colon \| x_0 - v \|_V \leq R \}  
	\colon 
	Y_{s,s}^x = X_{s,s}^x $
	and 
	Lemma~\ref{lemma:uniqueness}
	(with
	$ V = V $,
	$ a = s $,
	$ b = \min \{s_0 + \delta, T \} $, 
	$ f = ( [s, \min \{s_0 + \delta, T \}] \times V \ni (t,y) \mapsto F(t,y) \in V ) $,
	$ X = ( [s, \min \{s_0 + \delta, T \} ] \ni t \mapsto X_{s,t}^x \in V ) $,
	$ Y = ( [s, \min \{s_0 + \delta, T \} ] \ni t \mapsto Y_{s,t}^x \in V ) $
	for 
	$ s \in [ s_0 - \delta, s_0 + \delta ] \cap [0,T] $,
	$ x \in \{ v \in V \colon \| x_0 - v \|_V \leq R \}   $
	in the notation of Lemma~\ref{lemma:uniqueness})
	hence 
	show that item~\eqref{item:uniqueness}
	holds.
	In addition,  
	item~\eqref{item:Bound} 
	of
	Lemma~\ref{lemma:Existence}
	(with
	$ V = V $, 
	$ R = R $,
	$ h = 1 $,
	$ \varepsilon = \varepsilon $,
	$ s_0 = s_0 $,
	$ L = L $,
	$ M = M $, 
	$ \delta = \delta $, 
	$ x_0 = x_0 $, 
	$ f = F $ 
	in the notation of
	item~\eqref{item:Bound} 
	of
	Lemma~\ref{lemma:Existence})
	establishes item~\eqref{item:unif_bound}. 
	This completes the proof of Corollary~\ref{corollary:Extension}.
\end{proof}
\begin{lemma}
	\label{lemma:UnifCont}
	Let
	$ (V, \left\| \cdot \right\|_V) $ 
	be a nontrivial $ \R $-Banach space,  
	let
	$ T, R \in (0,\infty) $,  
	$ s_0 \in [0,T] $,
	$ x_0 \in V $,
	$ f \in \mathcal{C} ( [0,T] \times V, V) $,  
	for every
	$ x \in V $,
	$ s \in [0,T] $
	let
	$ X_{s,(\cdot)}^x 
	=
	(X_{s,t}^x)_{ t \in [s,T] }
	\colon [s,T]
	\to V $
	be a continuous function 
which satisfies for all 
	$ t \in [s,T] $  
	that 
	$ X_{s, t}^x = 
	x
	+
	\int_s^t f( \tau, X_{s, \tau}^x) \,d\tau $,
	assume 
	that
	$ 
	\inf_{ r \in (0,\infty) }
	\sup_{ \tau \in [0,T] } 
	\sup_{ y \in V, \| x_0 - y \|_V \leq R + r }
	\sup_{ 	z \in V \backslash \{ y \}, 
			\| x_0 - z \|_V \leq R + r }  
	\tfrac{
		\| f(\tau, y) - f( \tau, z ) \|_V
	}{
		\| y - z \|_V
	}
	< \infty $,
    and
    assume for all
    $ x \in V $ that
	$ 
	\inf_{ r \in (0, \infty) } 
	\sup_{ \tau \in [0, T] } 
	\sup_{  y \in V, \| x - y \|_V \leq r  }
	\sup_{	z \in V \backslash \{ y \}, 
			\| x - z \|_V \leq r  }   
	\tfrac{
		\| f(\tau, y) - f( \tau, z ) \|_V
	}{
		\| y - z \|_V
	}
	< \infty 
	$.
	Then there exists 
	$ \delta \in (0, \infty) $
	such that 
	\begin{equation}  
	\label{eq:LemmaUnifCont}
\{
(u,v) \in  
( [ s_0 - \delta, s_0 + \delta ] \cap [0,T] )^2
\colon 
u \leq v
\}
	\times \{ v \in V \colon 
	\| x_0 - v \|_V \leq R \}
	\ni
	(s,t,x)
	\mapsto
	X_{s,t}^x \in V 
	\end{equation}
	is uniformly continuous.
\end{lemma}
\begin{proof}[Proof of Lemma~\ref{lemma:UnifCont}]
	Throughout this proof let
	$ \varepsilon, L \in \R $  
	be real numbers
	which satisfy that
	\begin{equation}
	L =
	\sup\nolimits_{ \tau \in [0, T] } 
	\sup\nolimits_{ y \in V,
			\| x_0 - y \|_V \leq R + \varepsilon }
		\sup\nolimits_{
			z \in V \backslash \{ y \},
			\| x_0 - z \|_V \leq R + \varepsilon }   
	\tfrac{
		\| f(\tau, y) - f( \tau, z ) \|_V
	}{
		\| y - z \|_V
	} 
	\end{equation}
	and
	let
	$ M \in [0, \infty] $
	be the extended real number given by
	\begin{equation}
	M =
	\sup\nolimits_{ \tau \in [0,T] }
	\sup\nolimits_{
		y \in V, \| x_0 - y \|_V \leq R + \varepsilon 
	}
	\| f(\tau, y) \|_V
	.
	\end{equation}
	Note that
	Lemma~\ref{lemma:trivial}
	(with
	$ V = V $,
	$ a = 0 $,
	$ b = T $,
	$ r = R + \varepsilon $,
	$ x = x_0 $,
	$ f = f $
	in the notation of Lemma~\ref{lemma:trivial})
	shows that
	$ M < \infty $.
	Next observe that
	Corollary~\ref{corollary:Extension}
	(with
	$ V = V $,
	$ T = T $,
	$ R = R $,
	$ \varepsilon = \varepsilon $,
	$ s_0 = s_0 $,
	$ x_0 = x_0 $,
	$ f = f $,
	$ X_{s,t}^x = X_{s,t}^x $
	for 
	$ x \in V $,
	$ (s,t) \in [0,T] $
	with
	$ s \leq t $
	in the notation of 
	Corollary~\ref{corollary:Extension})
	ensures that there exists 
	$ \delta \in (0, \infty) $
	such that for all
	$ s \in [s_0- \delta, s_0 + \delta] \cap [0,T] $,
	$ t \in [s, \min \{ s_0 + \delta, T \} ] $,
	$ x \in \{ v \in V \colon \| x_0 - v \|_V \leq R \} $  
	it holds that
	\begin{equation}
	\label{eq:estimate}
	\| X_{s,t}^x - x_0 \|_V 
	\leq 
	R + \varepsilon 
	.
	\end{equation}
	Moreover, note that for all
	$ s,t,u
	\in [ \max \{ s_0 - \delta, 0 \}, 
	\min \{ s_0 + \delta, T \} ] $,
	$ x, y \in
	\{ v \in V \colon \| x_0 - v \|_V \leq
	R \} $
	with
	$ s, u \in [0,t] $
	it holds that
	\begin{equation}
	\begin{split}
	\label{eq:Cont_estimate}
	&
	\| X_{s,t}^x - X_{u, t}^y \|_V
	\leq 
	\| x - y \|_V 
	+
	\Big\| \int_s^t 
	f( \tau, X_{s, \tau}^x ) 
	\, d \tau 
	-
	\int_u^t
	f( \tau, X_{ u, \tau}^y ) \, d \tau 
	\Big\|_V
	\\
	&
	\leq 
	\| x - y \|_V 
	+
	\max \bigg\{
	\Big\|
	\int_s^{ \max \{s,u\}}
	f( \tau, X_{s, \tau}^x ) 
	\, d \tau 
	\Big\|_V
	,
	\Big\|
	\int_u^{ \max \{s,u\}}
	f( \tau, X_{u, \tau}^y ) 
	\, d \tau 
	\Big\|_V
	\bigg\}
	\\
	&
	\quad
	+
	\Big\|
	\int_{ \max \{s,u\} }^t
	f( \tau, X_{s, \tau}^x ) 
	-
	f( \tau, X_{ u, \tau}^y ) 
	\, d \tau 
	\Big\|_V  
	\\
	&
	\leq 
    \| x - y \|_V 
    +
    \max \bigg\{
    \Big\|
    \int_s^{ \max \{s,u\}}
    f( \tau, X_{s, \tau}^x ) 
    \, d \tau 
    \Big\|_V
    ,
    \Big\|
    \int_u^{ \max \{s,u\}}
    f( \tau, X_{u, \tau}^y ) 
    \, d \tau 
    \Big\|_V
    \bigg\}
	\\
	&
	\quad
	+ 
	\int_{ \max \{ s, u \}}^t
	\| 
	f( \tau, X_{s, \tau}^x ) 
	-
	f( \tau, X_{ u, \tau}^{ y } ) 
	\|_V
	\, d \tau   
	.
	\end{split} 
	\end{equation}
	Combining this with~\eqref{eq:estimate}
	proves that for all
	$ s,t,u
	\in [ \max \{ s_0 - \delta, 0 \}, 
	\min \{ s_0 + \delta, T \} ] $,
	$ x, y \in 
	\{ v \in V \colon \| x_0 - v \|_V \leq
	R \} $
	with
	$ s, u \in [0,t] $
	it holds
	that
	\begin{equation}
	\begin{split}
	\| X_{s,t}^x - X_{u, t}^y \|_V
	\leq 
	\| x - y \|_V
	+ 
	M | u - s |  
	+
	L
	\int_u^t
	\| X_{s, \tau}^x - X_{ u, \tau }^y \|_V \, d \tau
	.
	\end{split}
	\end{equation}
	The fact that
	$ M < \infty $
	and
	Gronwall's lemma therefore
	imply that
	for all
	$ s, t, u
	\in [ \max \{ s_0 - \delta, 0 \}, 
	\min \{ s_0 + \delta, T \} ] $,
	$ x, y \in \{ v \in V \colon \| x_0 - v \|_V \leq
	R \} $
	with
	$ s, u \in [0,t] $
	it holds
	that
	\begin{equation}
	\begin{split}
	\label{eq:GronwalEstimate}
	\| X_{s,t}^x - X_{u,t}^y \|_V
	\leq 
	(
	\| x - y \|_V
	+
	M
	|u-s| 
	)
	e^{L | t-u |}
	.
	\end{split}
	\end{equation}
	In addition, note that~\eqref{eq:estimate}
	shows that for all
	$ s, t, \tau
	\in [ \max \{ s_0 - \delta, 0 \}, 
	\min \{ s_0 + \delta, T \} ] $,
	$ x \in \{ v \in V \colon \| x_0 - v \|_V \leq R \} $
	with
	$ s \leq \min \{ t, \tau \} $
	it holds that
	\begin{equation}
	\begin{split}
		\| X_{s,t}^x - X_{s,\tau}^x \|_V 
		\leq
		\int_{ \min \{ \tau, t \} }^{ \max \{ \tau, t \} } 
		\| f(r, X_{s,r}^x ) \|_V \, dr
		\leq
		M
		| t - \tau |
		.
		\end{split}
	\end{equation}
	Combining this with~\eqref{eq:GronwalEstimate} 
	assures that for all
	$ s, t, u, \tau
	\in [ \max \{ s_0 - \delta, 0 \}, 
	\min \{ s_0 + \delta, T \} ] $,
	$ x, y \in
	\{ v \in V \colon \| x_0 - v \|_V \leq R \} $
	with
	$ s \leq t $, $ u \leq \tau \leq t $
	it holds that
	\begin{equation}
	\begin{split}
	\| X_{s,t}^x - X_{u, \tau}^y \|_V
	&
	\leq
	\| X_{s,t}^x - X_{u, t}^y \|_V
	+
	\| X_{u, t}^y - X_{u, \tau}^y \|_V
	\\
	&
	\leq 
	( \| x - y \|_V + M | u - s | )
 e^{L T}
 +
 M
 | t - \tau | 
	. 
	\end{split}
	\end{equation}
	The fact that
	$ M < \infty $
	establishes~\eqref{eq:LemmaUnifCont}.
	The proof of Lemma~\ref{lemma:UnifCont}
	is thus completed.
\end{proof}
\begin{lemma}
	\label{lemma:Nicer_assumption}
	Let $ ( V, \left \| \cdot \right \|_V ) $
	be a nontrivial $ \R $-Banach space 
	and
	let
	$ T \in (0,\infty) $, 
	$ x_0 \in V $,
	$ f \in \mathcal{C}^{0,1}( [0,T] \times V, V) $.
	Then   
	\begin{equation} 
	\inf\nolimits_{ r \in (0,\infty) }
	\sup\nolimits_{ \tau \in [0,T] } 
	\sup\nolimits_{   y \in V, 
			\| x_0-y \|_V \leq r }
	\sup\nolimits_{
			z \in V \backslash \{ y \},
			\| x_0 - z \|_V \leq r }   
	\tfrac{
		\| f(\tau, y) - f( \tau, z ) \|_V
	}{
		\| y - z \|_V
	}
	< \infty 
	.
	\end{equation}
\end{lemma}
\begin{proof}[Proof of Lemma~\ref{lemma:Nicer_assumption}]
	Throughout this proof let
	$ f_{0,1} \colon [0,T] \times V \to L( V ) $
	be the function 
	which satisfies 
	for all
	$ t \in [0,T] $, $ x \in V $ that
	$ f_{0,1}(t,x) = ( \tfrac{ \partial }{\partial x } f )(t,x) $.
Note that the 
	assumption that
	$ f \in \mathcal{C}^{0,1}( [0,T] \times V, V) $ 
	implies that
	there exists a function
	$ \delta \colon [ 0,T]   \to (0,\infty) $ 
	such that for all
	$ x \in V $,
	$ s, t \in [0, T] $ 
	with 
	$ \max \{ |s-t|, \|x-x_0\|_V \}<\delta_{t } $ 
	it holds that
	\begin{equation}
	\label{eq:cont_condition2}
	\| f_{0,1}(s,x) - f_{0,1}(t,x_0) \|_{L(V)}
	<
	\tfrac{ 1 }{2} 
	.
	\end{equation}
	Moreover,
	observe that the fact that
	$ [0, T] $ is compact ensures that 
	there exist
	$ n \in \N $,
	$ t_1, \ldots, t_n \in [0, T] $
	such that 
	\begin{equation}
	0 = t_1 < \ldots < t_n = T
	\qquad
	\text{and}
	\qquad
	[ 0, T ] = \cup_{ i = 1 }^n  
	\{ r \in [0, T] \colon | r - t_i | < \delta_{t_i } \}
	.
	\end{equation}
	This and~\eqref{eq:cont_condition2} demonstrate that 
	there exist
	$ n \in \N $,
	$ t_1, \ldots, t_n \in [a,b] $
	such that for all
	$ x \in V $,
	$ t \in [0, T] $ 
	with
	$ \| x - x_0 \|_V < \min \{ \delta_{t_1 },  \ldots, \delta_{t_n } \} $
	it holds that
	\begin{equation}
	\begin{split}
	&
	\| f_{0,1} (t, x) - f_{0,1} (t, x_0) \|_V 
		\\
		&
	\leq
	\min_{i\in [0,n]\cap \N} 
	\big| 
	\| f_{0,1} (t, x) - f_{0,1} (t_i, x_0) \|_V
	+
	\| f_{0,1} (t_i, x_0) - f_{0,1} (t, x_0) \|_V
	\big|
	%
	<
	\tfrac{ 1 }{2}
	+
	\tfrac{ 1 }{2} 
	=
	1
	.
	\end{split}
	\end{equation}
	Hence, we obtain that for all
	$ t \in [0,T] $,
	$ y, z \in V $ 
	with
	$ \max \{ \| y - x_0 \|_V, 
	\| z - x_0 \|_V \} 
	<
	\min \{ \delta_{t_1 },  \ldots,
	\delta_{t_n } \} $ 
	it holds
	that
	\begin{equation}
	\begin{split}
	&
	\| f(t, y) - f (t,z) \|_V
	=
	\Big\| 
	\int_0^1 f_{0,1}(t, z + s(y-z) )(y-z) \, ds
	\Big\|_V
	\\
	&
	\leq
	\int_0^1 
	\| f_{0,1}(t, z + s(y-z) ) \|_{L(V)}
	\| y - z \|_V 
	\, ds
	\\
	&
	\leq 
	\| f_{0,1}(t, x_0 ) \|_V
	\| y - z \|_V
	+
	\int_0^1 
	\| f_{0,1}(t, z + s(y-z) ) - f_{0,1}(t, x_0 ) \|_{L(V)}
	\| y - z \|_V 
	\, ds	
	\\
	&
	<
	(
	\| f_{0,1}(t, x_0 ) \|_V 
	+
	1
	)
	\| y - z \|_V
	\leq 
	\big(
	\sup\nolimits_{ \tau \in [0,T] }
	\| f_{0,1}(\tau, x_0 ) \|_V 
	+
	1
	\big)
	\| y - z \|_V
	<
	\infty  
	.
	\end{split}
	\end{equation}
The proof of
	Lemma~\ref{lemma:Nicer_assumption}
	is thus completed. 
\end{proof}
\begin{lemma}
	\label{lemma:flow_property}
	Let
	$ (V, \left\| \cdot \right\|_V) $ 
	be a nontrivial $ \R $-Banach space, 
	let
	$ T \in (0, \infty) $,
	$ f \in \mathcal{C}^{0,1}( [0, T] \times V, V) $,   
	and 
	for every
	$ x \in V $,
	$ s \in [0,T] $
	let
	$ X_{s, (\cdot)}^x 
	=
	(X_{s,t}^x)_{ t \in [s,T] }
	\colon [s,T] \to V $
	be a continuous function
which satisfies for all
	$ t \in [s,T] $
	that
	$ X_{s,t}^x = 
	x
	+
	\int_s^t f( \tau , X_{s, \tau }^x) \,d\tau $.
	Then it holds for all 
	$ x \in V $,
	$ t_1 \in [0,T] $,
	$ t_2 \in [t_1, T] $,
	$ t_3 \in [t_2, T] $
	that
	$ X_{t_2, t_3}^{X_{t_1, t_2}^x} = X_{t_1, t_3}^x $.
\end{lemma}
\begin{proof}[Proof of Lemma~\ref{lemma:flow_property}]
	Throughout this proof let
	$ f_{0,1} \colon [0,T] \times V \to L( V ) $
	be the function 
	which satisfies 
	for all
	$ t \in [0,T] $, $ x \in V $ that
	$ f_{0,1}(t,x) = ( \tfrac{ \partial }{\partial x } f )(t,x) $.
	Observe that for all
	$ x \in V $,
	$ t_1 \in [0,T] $,
	$ t_2 \in [t_1, T] $,
	$ t_3 \in [t_2, T] $
	it holds that
	\begin{equation}
	\begin{split}
	X_{t_2, t_3}^{ X_{t_1, t_2}^x }
	=
	X_{t_1, t_2}^x
	+
	\int_{t_2}^{t_3} 
	f \big( \tau, X_{t_2, \tau}^{ X_{t_1, t_2}^x } \big) 
	\, d \tau
	.
	\end{split}
	\end{equation}
	This,
	the assumption  that
	$ \forall \, x \in V, s \in [0,T] \colon
	( [s,T] \ni t \mapsto X_{s,t}^x \in V ) 
	\in \mathcal{C} ( [s,T], V ) $,
	and
	the assumption that
	$ f \in \mathcal{C}^{0,1}( [0,T] \times V, V ) $
	imply that for all
	$ x \in V $,
	$ t_1 \in [0,T] $,
	$ t_2 \in [t_1, T] $,
	$ t_3 \in [t_2, T] $
	it holds that
	\begin{equation}
	\begin{split}
	&
	\| 
	X_{t_2, t_3}^{ X_{t_1, t_2}^x }
	-
	X_{t_1, t_3}^x
	\|_V
	=
	\Big\|
	\int_{t_2}^{t_3}
	f \big( \tau, X_{ t_2, \tau}^{X_{t_1, t_2}^x } \big) \, d \tau
	+
	X_{t_1, t_2}^x
	-
	X_{t_1, t_3}^x
	\Big\|_V
	\\
	&
	=
	\Big\| 
	\int_{t_2}^{t_3}
	f \big( \tau, X_{ t_2, \tau}^{X_{t_1, t_2}^x } \big) \, d \tau
	-
	\int_{t_2}^{t_3}
	f ( \tau, X_{t_1, \tau}^x ) \, d\tau
	\Big\|_V
	\leq
	\int_{t_2}^{t_3}
	\Big\| 
	f \big( \tau, X_{t_2, \tau}^{X_{t_1, t_2}^x } \big)
	-
	f ( \tau, X_{t_1, \tau}^x ) 
	\Big \|_V \, d\tau
	\\
	&
	=
	\int_{t_2}^{t_3}
	\Big\|
	\int_0^1
	f_{0,1}
	\Big( \tau, 
	X_{t_1, \tau}^x
	+
	\big( 
	X_{t_2, \tau}^{X_{t_1, t_2}^x }
	-
	X_{t_1, \tau}^x
	\big)
	r
	\Big)
	\Big( 
	X_{t_2, \tau}^{X_{t_1, t_2}^x }
	-
	X_{t_1, \tau}^x
	\Big)
	\, dr
	\Big\|_V
	\, d\tau
	\\
	&
	\leq
	\int_{t_2}^{t_3} 
	\Big(
	\int_0^1
	\Big \|
	f_{0,1}
	\Big ( \tau, 
	X_{t_1, \tau}^x
	+
	\big( 
	X_{t_2, \tau}^{X_{t_1, t_2}^x }
	-
	X_{t_1, \tau}^x
	\big)
	r
	\Big )
	\Big\|_{L(V)}
	\, dr 
	\Big)
	\big\|
	X_{t_2, \tau}^{X_{t_1, t_2}^x }
	-
	X_{t_1, \tau}^x 
	\big\|_V
	\, d\tau
	\\
	&
	\leq
	\sup_{r \in [0,1], t \in [t_2,T]}
	\Big\|
	f_{0,1}
	\Big( t, 
	X_{t_1, t}^x
	+
	\big( 
	X_{t_2, t}^{X_{t_1, t_2}^x }
	-
	X_{t_1, t}^x
	\big)
	r
	\Big)
	\Big\|_{L(V)}
	\int_{t_2}^{t_3} 
	\big\|
	X_{t_2, \tau}^{X_{t_1, t_2}^x }
	-
	X_{t_1, \tau}^x 
	\big\|_V
	\, d\tau
	< \infty.
	\end{split}
	\end{equation}
	Gronwall's lemma hence shows
	for all
	$ x \in V $,
	$ t_1 \in [0,T] $,
	$ t_2 \in [t_1, T] $,
	$ t_3 \in [t_2, T] $
	that
	$ X_{t_1, t_3}^x = X_{t_2, t_3}^{X_{t_1, t_2}^x} $.
	The proof of Lemma~\ref{lemma:flow_property}
	is thus completed.
\end{proof}
\begin{corollary}
	\label{Corollary:UnifCont}
	Let
	$ (V, \left\| \cdot \right\|_V) $ 
	be a nontrivial $ \R $-Banach space, 
	let
	$ T \in (0,\infty) $,  
	$ f \in \mathcal{C}^{0,1} ( [0,T] \times V, V) $,  
	and
	for every
	$ x \in V $,
	$ s \in [0,T] $
	let
	$ X_{s,(\cdot)}^x \colon [s,T] \to V $
	be a continuous function
	which satisfies for all
	$ t \in [s,T] $ that
	$ X_{s, t}^x = 
	x
	+
	\int_s^t f( \tau, X_{s, \tau}^x) \,d\tau $.
	Then it holds that 
	$ \{ (u, v) \in [0,T]^2 \colon u \leq v \} \times V 
	\ni
	(s,t,x)
	\mapsto
	X_{s,t}^x \in V $
	is a continuous function.
\end{corollary}
\begin{proof}[Proof of Corollary~\ref{Corollary:UnifCont}]
Throughout this proof
we denote by
$ \angle_T \subseteq [0,T]^2 $ 
the set given by
$ \angle_T = \{ (s, t) \in [0,T]^2 \colon s \leq t \} $,
let
$ (s_0, t_0, x_0 ) \in \angle_T \times V $,
and let
$ \varepsilon \in (0, \infty) $.
Note that 
Lemma~\ref{lemma:Nicer_assumption}
(with
$ V = V $,
$ T = T $, 
$ x_0 = x $,
$ f = f $
for 
$ x \in V $
in the notation of 
Lemma~\ref{lemma:Nicer_assumption})
shows that
there exists a function
$ r \colon V \to (0,\infty) $
such that
for every
$ x \in V $
it holds that
	\begin{equation} 
	\label{eq:locLipCondition}
	\sup\nolimits_{ \tau \in [0,T] } 
	\sup\nolimits_{  y \in V, \| x - y \|_V \leq 2 r_x }
	\sup\nolimits_{ z \in V \backslash \{ y \}, 
			\| x-z \|_V \leq 2 r_x}   
	\tfrac{
		\| f(\tau, y) - f( \tau, z ) \|_V
	}{
		\| y - z \|_V
	}
	< \infty 
	.
	\end{equation} 
Lemma~\ref{lemma:UnifCont}  
(with
$ V = V $,
$ T = T $,
$ R = r_x $, 
$ s_0 = s $,
$ x_0 = x $,
$ f = f $,
$ X_{s,t}^x = X_{s,t}^x $
for
$ (s,t) \in \angle_T $,
$ x \in V $
in the notation of
Lemma~\ref{lemma:UnifCont})
hence
ensures that there exists a function
$ \delta \colon [0,T] \times V \to (0,\infty) $
such that for all 
$ s \in  [0,T] $, 
$ x \in V $
it holds that 
\begin{equation} 
\begin{split}
\label{eq:Improved}
&
( \angle_T \cap 
[ s - \delta_{s,x}, s + \delta_{s,x} ]^2  )
\times 
\{ 
v \in V \colon 
\| x - v \|_V \leq 
r_x 
\}
\ni
(u,t,y)
\mapsto
X_{u,t}^y \in V 
\end{split}
\end{equation}
is a uniformly continuous function.
Furthermore, note that
the fact that
$ \forall \, x \in V, s \in [0,T] \colon 
( [s,T] \ni t  \mapsto  X_{ s, t }^x \in V ) \in 
\mathcal{C}( [s,T], V ) $
ensures that there exists a function
$ \Delta \colon [0,T]
\to (0, \infty) $
such that for all  
$ t, \tau \in [s_0,T] $ 
with
$ | t - \tau | < \Delta_t $
it holds that
\begin{equation}
\label{eq:continuity}
\| X_{s_0,t}^{x_0} - X_{s_0, \tau}^{x_0} \|_V
\leq 
\tfrac{1}{2}
r_{ X_{s_0,t}^{x_0} } 
.
\end{equation}
This implies that there exist non-empty intervals
$ I_{s,x} \subseteq [0,T] $, 
$ s \in [0,T] $,
$ x \in V $,
such that for all
$ s \in (0,T) $,
$ x \in V $
it holds that
$ I_{s,x} = ( \max \{ s - \delta_{s,x}, s- \Delta_s, 0 \}, 
\min \{ s + \delta_{s,x}, s+ \Delta_s, T \} ) $,
$ I_{0,x} = [0, \min \{ \delta_{0,x}, \Delta_0 \} ) $,
and
$ I_{T,x} = (T- \min \{ \delta_{T,x}, \Delta_T \}, T] $.
In addition, observe that
the fact that
$ \cup_{ s \in [s_0, t_0] } I_{s, X_{s_0,s}^{x_0} } \supseteq
 [s_0, t_0] $ 
 and
 the fact that
$ [s_0, t_0] $
is a compact set
 assure
 that there exist
$ n \in [2, \infty) \cap \N $,
$ s_1, \ldots, s_n \in (s_0, t_0 ] $
such that
$ s_1 < \ldots < s_n = t_0 $
and
\begin{equation} 
 [s_0, t_0 ] \subseteq 
\cup_{ j = 0 }^n I_{ s_j, X_{s_0, s_j}^{x_0} } 
.
\end{equation}
This demonstrates
that there exist real numbers
$ \tau_j \in (s_0, t_0) $, $ j \in [1,n] \cap \N $,
such that
for all
$ j \in [1,n] \cap \N $
it holds that
\begin{equation}
\begin{split}
\tau_j 
\in 
\big(
( I_{ s_{j-1}, X_{s_0, s_{j-1} }^{x_0} } )
\cap
( I_{ s_{j}, X_{s_0, s_{j} }^{x_0} } )
\big) 
.
\end{split}
\end{equation}
Moreover, note that~\eqref{eq:continuity}
(with 
$ t = s_n $,
$ \tau = \tau_n $) 
ensures that 
$ \| X_{s_0, s_n}^{x_0} - X_{s_0, \tau_n}^{x_0} \|_V 
\leq
\frac{1}{2}
r_{ X_{s_0, s_n }^{x_0} } $.
Combining this 
with~\eqref{eq:Improved}
(with
$ s = s_{n} $,
$ x = X_{s_0, s_n }^{x_0} $) 
and
Lemma~\ref{lemma:flow_property}
(with
$ V = V $,
$ T = T $,
$ f = f $,
$ X_{s,t}^x = X_{s,t}^x $
for 
$ (s, t ) \in \angle_T $,
$ x \in V $
in the notation of Lemma~\ref{lemma:flow_property})
proves that there exists 
$ \varepsilon_n \in (0, \min \{1,\delta_{s_n, X_{s_0, s_n}^{x_0} }, \frac{1}{2}
r_{ X_{s_0, s_n }^{x_0} } \} ) $
such that
for all 
$ s \in [0, \tau_n ] $,
$ t \in [ \tau_n, T ] $, 
$ x \in V $
with
$ \max \{ \| X_{s_0, \tau_n }^{x_0} - X_{s, \tau_n }^x \|_V,
| t_0 - t | \} < \varepsilon_n $
it holds that
$ \| X_{s_0, s_n}^{x_0} - X_{s, \tau_n}^{x} \|_V 
\leq 
r_{ X_{s_0, s_n }^{x_0} } $ 
and
\begin{equation}
\begin{split}
\label{eq:first_step}
\| X_{s_0, t_0}^{x_0} - X_{s,t}^x \|_V
=
\|X^{x_0}_{s_0,s_n}-X^x_{s,t}\|_V
=
\big\| 
X_{ \tau_n, t_0}^{ X_{s_0, \tau_n}^{x_0} } 
-
X_{ \tau_n, t}^{ X_{s, \tau_n}^x } 
\big\|_V
<
\varepsilon
.
\end{split}
\end{equation}
Furthermore, observe that~\eqref{eq:continuity} 
(with
$ t = s_j $,
$ \tau = \tau_j $
for 
$ j \in [1, n-1] \cap \N $) 
ensures that 
$ \| X_{s_0, s_j}^{x_0} - X_{s_0, \tau_j}^{x_0} \|_V 
\leq
\frac{1}{2}
r_{ X_{s_0, s_j }^{x_0} } $.
This, \eqref{eq:Improved}
(with
$ s = s_j $,
$ x = X_{s_0, s_j }^{x_0} $
for 
$ j \in [1, n - 1] \cap \N $), 
and
Lemma~\ref{lemma:flow_property}
(with
$ V = V $,
$ T = T $,
$ f = f $,
$ X_{s,t}^x = X_{s,t}^x $
for 
$ (s, t ) \in \angle_T $,
$ x \in V $
in the notation of Lemma~\ref{lemma:flow_property})
ensure that there exist 
$ \varepsilon_j \in (0, \min \{1,\delta_{s_j, X_{s_0, s_j}^{x_0} }, \frac{1}{2}
r_{ X_{s_0, s_j }^{x_0} } \} ) $,
$ j \in [1, n - 1] \cap \N $,
such that
for all
$ j \in [1, n - 1] \cap \N $,
$ s \in [0, \tau_{j}] $, 
$ x \in V $
with
$ \| X_{s_0, \tau_{j} }^{x_0} - X_{s, \tau_{j} }^x \|_V  
< \varepsilon_j $
it holds that
$ \| X_{s_0, s_j}^{x_0} - X_{s, \tau_j}^{x} \|_V 
\leq 
r_{ X_{s_0, s_j }^{x_0} } $ 
and
\begin{equation}
\begin{split}
\label{eq:general_step}
\| X_{s_0, \tau_{j+1} }^{x_0} - X_{s, \tau_{j+1} }^x \|_V
=
\big \| 
X_{ \tau_j, \tau_{j+1} }^{ X_{s_0, \tau_j }^{x_0} } 
-
X_{ \tau_j, \tau_{j+1} }^{ X_{s, \tau_j }^{x } } 
\big \|_V
<
\varepsilon_{j+1}
.
\end{split}
\end{equation}
In addition, note that~\eqref{eq:Improved}
(with
$ s = s_0 $,
$ x = x_0 $) 
shows that
there exists
$ \varepsilon_0 \in (0, \min \{ 1, \delta_{s_0, x_0 }, r_{x_0} \} ) $
such that
for all 
$ s \in [0, \tau_1] $,
$ x \in V $
with
$ \max \{ \| x_0 - x\|_V, | s_0 - s | \} 
< 
\varepsilon_0 $ 
it holds that
\begin{equation}
\| X_{s_0, \tau_1}^{x_0} - X_{ s, \tau_1 }^x \|_V
< \varepsilon_1
.
\end{equation}
Combining this with~\eqref{eq:first_step}
and~\eqref{eq:general_step}
implies that
for all
$ s \in [0, \tau_1 ] $,
$ t \in [s,T] $,
$ x \in V $ 
with
$ \max \{ \| x_0 - x \|_V, | s_0 - s | \} < 
 \varepsilon_0  $ 
 and 
$ | t_0 - t | 
< 
\min \{ \varepsilon_n,  
| t_0 - \tau_n | \} $
it holds that
\begin{equation}
\| X_{s_0,t_0}^{x_0} - X_{s,t}^x \|_V < \varepsilon.
\end{equation}
Hence, we obtain that there exists $ \delta \in (0, \infty) $ such that for all
$ (s,t,x) \in \angle_T \times V $
with
$ \max \{ | s - s_0|, 
| t - t_0 |,
\| x - x_0 \|_V \}
<
\delta $
it holds that
\begin{equation}
\| X_{s_0,t_0}^{x_0} - X_{s,t}^x \|_V < \varepsilon.
\end{equation}
The proof of Corollary~\ref{Corollary:UnifCont}
is thus completed.
\end{proof} 
\section{Continuous differentiability of solutions to initial value problems}
\label{subsection:ContinuousDifferentiability}
In this section we prove
in 
Lemma~\ref{lemma:Differentiable}
(cf., e.g., Driver~\cite[Theorem~19.13]{Driver2003})
differentiability properties of solutions to
initial value problems.
In order to do so, we recall a few elementary 
auxiliary results in
Lemmas~\ref{lemma:WellConti}--\ref{lemma:help_lemma}
(cf., e.g., Driver~\cite[Theorem~19.14]{Driver2003}),
Lemma~\ref{lemma:converse_chain_rule}
(cf., e.g., Driver~\cite[Theorem~19.7]{Driver2003}), 
and
Lemmas~\ref{lemma:existence_uniqueness}--\ref{lemma:bound}.
Then we combine them 
to establish
continuous differentiability 
of the solution to the considered initial value problem 
with respect to the initial data
in
Lemma~\ref{lemma:diff_initial_value}.
In addition,
we establish 
in
Lemma~\ref{lemma:Differentiable} 
continuous differentiability of the solution to the considered initial
value problem
with respect to 
the initial time as well as the current
time. 
\begin{lemma}
	\label{lemma:WellConti}
	Let
	$ (U, \left\| \cdot \right\|_U) $,
	$ (V, \left\| \cdot \right\|_V) $,
	and
	$ (W, \left\| \cdot \right\|_W) $   
	be 
	$ \R $-Banach spaces  
	and let
	$ T \in (0,\infty) $, 
	$ \angle_T = \{ (s, t) \in [0,T]^2 \colon s \leq t \} $, 
	$ f \in \mathcal{C} ( [0,T] \times U, L(V, W) ) $, 
	$ y \in 
	\mathcal{C}( \angle_T, U ) $, 
	$ h \in 
	\mathcal{C}( \angle_T, V ) $.
 	Then 
 	$ \angle_T \ni (s,t) \mapsto \int_s^t f( \tau, y( s, \tau ) )
 	h( s, \tau ) \, d \tau \in W $
 	is continuous.
\end{lemma}
\begin{proof}[Proof of Lemma~\ref{lemma:WellConti}]
	Throughout this proof let
	$ X \colon \angle_T \to W $
	be the function which satisfies for all
	$ (s,t) \in \angle_T $  
	that 
	$ X_{s, t} = 
	\int_s^t f( \tau, y( s, \tau )) h(s, \tau) \,d \tau $.
	Observe that for all
	$ s \in [0,T] $, $ t, u \in [s,T] $ 
	with
	$ t \leq u $
	it holds that
	\begin{equation}
	\begin{split}
	\label{eq:Eq1}
	\| X_{s, u } - X_{s,t} \|_W
	&=
	\Big\| \int_t^u 
	f( \tau, y( s, \tau ) ) 
	h( s, \tau )
	\, d \tau  
	\Big\|_W
	\\
	&
	\leq
	|t-u|
	\bigg[ 
	\sup_{ (r,\tau) \in \angle_T }   
	\| f( \tau, y( r, \tau ) ) h( r, \tau )  \|_W
	\bigg] 
	.
	\end{split}
	\end{equation}
	Hence, we obtain that for all
	$ s \in [0,T] $, 
	$ t, u \in [s,T] $
	it holds that
	\begin{equation}
	\begin{split}
	\| X_{s,t} - X_{s,u} \|_W
	\leq
	| t - u |
	\bigg[
	\sup_{ ( r, \tau ) \in \angle_T }
	\| f( \tau, y ( r, \tau ) ) h( r, \tau ) \|_W
	\bigg]
	.
	\end{split}
	\end{equation}
	In addition, note that for all
	$ s, u, t \in [0, T] $ 
	with
	$ s, u \in [0,t] $
	it holds that
	\begin{equation}
	\begin{split}
	%
	&
	\| X_{s,t} - X_{u, t} \|_W
	=
	\Big\| \int_s^t 
	f( \tau, y( s, \tau ) ) 
	h( s, \tau )
	\, d \tau 
	-
	\int_u^t
	f( \tau, y( u, \tau ) )
	h( u, \tau )
	\, d \tau 
	\Big\|_W
	\\
	&
	\leq 
	\Big\|
	\int_{ \max \{ s, u \} }^t
	f( \tau, y( \max \{ s, u \}, \tau ) ) 
	h( \max \{ s, u \}, \tau ) \, d \tau 
	\\
	&
	\quad 
	-
	\int_{ \min \{ s, u \} }^t
	f( \tau, y( \min \{ s, u \}, \tau ) )
	h( \min\{s,u\}, \tau )
	\, d \tau 
	\Big\|_W
	\\
	&
	=
	\Big\|
	\int_{ \max \{ s, u \} }^t
	f( \tau, y( \max \{ s, u \}, \tau ) ) 
	h( \max \{ s, u \}, \tau ) \, d \tau 
	\\
	&
	\quad 
	-
	\Big[ 
		\int_{ \min \{ s, u \} }^{ \max \{ s, u\} } 
	f( \tau, y( \min \{ s, u \}, \tau ) )
	h( \min\{s,u\}, \tau )
	\, d \tau
	\\
	&
	\quad 
	+
	\int_{ \max \{ s, u \} }^t
	f( \tau, y( \min \{ s, u \}, \tau ) )
	h( \min\{s,u\}, \tau )
	\, d \tau
	\Big]
	\Big\|_W
	\\
	&
	\leq 
	\int_{ \max \{ s, u \} }^t
	\|
	f( \tau, y( \max \{ s, u \}, \tau ) ) 
	h( \max \{ s, u \}, \tau ) 
	-
	f( \tau, y( \min \{ s, u \}, \tau ) )
	h( \min\{s,u\}, \tau )
	\|_W 
	\, d \tau 
	\\
	&
	\quad 
	+  
	\int_{ \min \{ s, u \} }^{ \max \{ s, u\} } 
	\|
	f( \tau, y( \min \{ s, u \}, \tau ) )
	h( \min\{s,u\}, \tau )
	\|_W
	\, d \tau
	.
\end{split}
\end{equation} 
This implies that for all
$ s, u, t \in [0, T] $ 
with
$ s, u \in [0,t] $
it holds that
\begin{equation}
\begin{split}
\label{eq:Eq2}
\| X_{s,t} - X_{u, t} \|_W
&\leq  
	\int_{ \max \{ s, u \}}^t
	\| 
	f( \tau, y( s, \tau ) ) 
	h( s, \tau )
	-
	f( \tau, y( u, \tau ) )
	h( u, \tau ) 
	\|_W
	\, d \tau   
	\\
	& 
	\quad  
	+
	| s - u |
	\bigg[ 
	\sup_{ (r,\tau) \in \angle_T }   
	\| f( \tau, y( r, \tau ) ) h( r, \tau )  \|_W
	\bigg]
	.
	\end{split} 
	\end{equation}
	Combining~\eqref{eq:Eq1}
	and~\eqref{eq:Eq2} 
	assures that for all
	$ s, t, u, v \in [0,T] $ 
	with
	$ s \leq t $,
	$ u \leq v $,
	and
	$ u \leq t $ 
	it holds that
	\begin{equation}
	\begin{split}
	&
	\| X_{s,t} - X_{u, v} \|_W
	\leq
	\| X_{s,t} - X_{u, t} \|_W
	+
	\| X_{u, t} - X_{u, v} \|_W
	\\
	&
	\leq 
(
| t - v |
+
| s - u |
)
\big[ 
\sup\nolimits_{ (r,\tau) \in \angle_T }   
\| f( \tau, y( r, \tau ) ) h( r, \tau ) \|_V
\big]
\\
&
\quad 
	+ 
	\int_{ \max \{ s, u \}}^t
	\| 
	f( \tau, y( s, \tau ) ) 
	h( s, \tau )
	-
	f( \tau, y( u, \tau ) )
	h( u, \tau ) 
	\|_V
	\, d \tau 
	. 
	\end{split}
	\end{equation}
	The dominated convergence theorem hence
	completes the proof of Lemma~\ref{lemma:WellConti}.
\end{proof}
\begin{lemma}
	\label{lemma:CompactnessArgument}
	Let  $ ( V, \left \| \cdot \right \|_V ) $
	be a nontrivial $ \R $-Banach space,
	let
	$ \varepsilon, T \in (0,\infty) $,
	$ \angle_T = \{ (s, t) \in [0,T]^2 \colon s \leq t \} $,
	$ y \in \mathcal{C} ( \angle_T, V ) $,
	$ f \in \mathcal{C}^{0,1}( [0,T] \times V, V ) $,
	and
	let
	$ f_{0,1} \colon [0,T] \times V \to L( V ) $
	be the function 
	which satisfies 
	for all
	$ t \in [0,T] $, $ x \in V $ that
	$ f_{0,1}(t,x) = ( \tfrac{ \partial }{\partial x } f )(t,x) $.
	Then there exists $ \delta \in (0,\infty) $
	such that for all
	$ h \in \mathcal{C}( \angle_T, V ) $
	with
	$ \sup_{ (s, t) \in \angle_T } 
	\| h(s, t) \|_V < \delta $
	it holds that
	\begin{equation}
	\sup\nolimits_{ r \in [0,1] }
	\sup\nolimits_{ (s,t) \in \angle_T }
	\| f_{0,1}( t, y(s,t) + r h (s,t) ) 
	-
	f_{0,1}( t, y(s,t) ) \|_{L(V)} 
	< \varepsilon 
	.
	\end{equation}
\end{lemma}
\begin{proof}[Proof of Lemma~\ref{lemma:CompactnessArgument}] 
	Throughout this proof let
	$ C = y( \angle_T ) = \{ y(s,t) \in V \colon (s,t) \in \angle_T \} $.
	Note that the assumption that
	$ f \in \mathcal{C}^{0,1}( [0,T] \times V, V ) $
	ensures that 
	there exists a function
	$ \delta \colon [0,T] \times V \to  (0,\infty) $
	such that for all
	$ (t, x), (\tau, \xi) \in [0,T] \times V $
	with
	$ | t - \tau | + \| x - \xi \|_V < \delta_{t,x} $
	it holds that
	\begin{equation}
	\label{eq:cont_cri}
	\| f_{0,1}( \tau, \xi ) - f_{0,1}(t,x) \|_{L(V)}
	< 
	\tfrac{ \varepsilon }{ 2 }
	.
	\end{equation}
	Moreover, observe that
	\begin{equation}
	[0,T]\times C \subseteq  
	\cup_{ (t,x) \in [0,T] \times C }
	\{ (\tau, \xi) \in [0,T] \times V \colon 
	| t - \tau | + \| x - \xi \|_V < \tfrac{1}{2} \delta_{t,x} \}
	.
	\end{equation}
	\sloppy
The fact that
	$ [0,T] \times C $ is a compact set
	therefore implies that
	there exist 
	$ n \in \N $, 
	$ (t_1, x_1), \ldots, (t_n, x_n) \in [0,T] \times C $
	such 
	that
	\begin{equation}
	[0,T] \times C
	\subseteq 
	\cup_{i=1}^n
	\{ (\tau, \xi) \in [0,T] \times V \colon 
	| t_i - \tau | + \| x_i - \xi \|_V < \tfrac{1}{2}  \delta_{t_i, x_i}  \}
	.
	\end{equation}
	This and~\eqref{eq:cont_cri}
	show that
	there exist 
	$ n \in \N $, 
	$ (t_1, x_1), \ldots, (t_n, x_n) \in [0,T] \times C $
	such that
	for all 
	$ (s,t) \in \angle_T $,
	$ r \in [0,1] $,
	$ h \in \mathcal{C}( \angle_T, V) $
	with
	$ \sup_{ (u,\tau) \in \angle_T } 
	\| h(u, \tau) \|_V <
	\min_{i \in [1,n]\cap \N } 
	\nicefrac{ \delta_{ t_i, x_i } }{2} $
	it holds that
	\begin{equation}
	\begin{split}
	&
	\| 
	f_{0,1}( t, y(s,t) + r h (s, t) ) 
	-
	f_{0,1}(t, y(s,t) ) 
	\|_{L(V)}
	\\
	&
	\leq
	\min_{ i \in [1,n] \cap \N }
	\big[
	\| 
	f_{0,1}(t, y(s,t) + r h(s,t) ) 
	-
	f_{0,1}(t_i, x_i)
	\|_{L(V)}
	\\
	&
	\quad 
	+
	\| 
	f_{0,1}(t_i, x_i)
	-
	f_{0,1}(t, y(s,t))
	\|_{L(V)}
	\big]
	\\
	&
	\leq
	\tfrac{ \varepsilon }{ 2 } 
	+
	\tfrac{ \varepsilon }{ 2 }
	=
	\varepsilon
	.
	\end{split}
	\end{equation}
	The proof of 
	Lemma~\ref{lemma:CompactnessArgument} is thus completed.
\end{proof}
\begin{lemma}
	\label{lemma:help_lemma}
	Let  $ ( V, \left \| \cdot \right \|_V ) $
	be a nontrivial $ \R $-Banach space,
	let $ T \in (0, \infty) $, 
	$ \angle_T = \{ (s, t) \in [0,T]^2 \colon s \leq t \} $,
	$ f = ( f(t, x) )_{ (t, x) \in [0,T] \times V } \in \mathcal{C}^{0,1}([0,T] \times V, V  ) $,
	$ F \colon \mathcal{C}(\angle_T, V)
	\to
	\mathcal{C}( \angle_T, V) $
	satisfy for all
	$ ( s, t ) \in \angle_T $,
	$ y \in \mathcal{C}(\angle_T, V) $ 
	that
	$ (F(y))(s, t) = \int_s^t f (\tau, y( s, \tau ) ) \, d \tau $,
	and let 
	$ f_{0,1} \colon [0,T] \times V \to L( V ) $
	be the function 
	which satisfies 
	for all
	$ t \in [0,T] $, $ x \in V $ that
	$ f_{0,1}(t,x) = ( \tfrac{ \partial }{\partial x } f )(t,x) $.
	Then it holds 
	for all 
	$ y, h \in \mathcal{C}( \angle_T, V ) $,
	$ (s,t) \in \angle_T $ 
	that
	$ F \in \mathcal{C}^1( \mathcal{C}( \angle_T, V ), 
	\mathcal{C}( \angle_T, V ) ) $
	and  
	\begin{equation}
	(F'(y) h)(s,t)
	=
	\int_s^t 
	f_{0,1}(\tau, y(s, \tau) ) h(s, \tau ) \, d \tau .
	\end{equation}
\end{lemma}
\begin{proof}[Proof of Lemma~\ref{lemma:help_lemma}]
	Note that for all
	$ y, h \in \mathcal{C}( \angle_T, V) $,
	$ (s, \tau) \in \angle_T $
	it holds that
	\begin{equation}
	\begin{split}
	f( \tau, y( s, \tau ) + h ( s, \tau  ) ) - f ( \tau, y (s,\tau) )
	=
	\int_0^1  
f_{0,1}
	( \tau, y( s, \tau ) + r h(s, \tau ) ) 
	h( s, \tau )
	\, dr
	.
	\end{split}
	\end{equation}
	This ensures that for all
	$ y, h \in \mathcal{C}( \angle_T, V ) $,
	$ (s,t) \in \angle_T $ 
	it holds that
	\begin{equation}
	\begin{split}
	&
	( F ( y + h ) )(s, t) - ( F( y ) )(s, t) - \int_s^t 
	f_{0,1}
	 (\tau, y(s, \tau) ) 
	h( s, \tau ) \, d\tau
	\\
	&
	=
	\int_s^t
	\big[ 
	f( \tau, y(s, \tau) + h( s, \tau ) ) 
	- 
	f( \tau, y( s, \tau ) )
	- 
	f_{0,1}
	(\tau, y( s, \tau ) ) 
	h( s, \tau )
	\big]
	\,
	d\tau
	\\
	&
	=
	\int_s^t
	\bigg(
	\int_0^1
	\big[  
f_{0,1}
	( \tau, y( s, \tau ) + r h( s, \tau ) )
	-  
f_{0,1}
	( \tau, y( s, \tau )  )
	\big]
	h( s, \tau )
	\, dr
	\bigg)
	\, d\tau
	.
	\end{split}
	\end{equation}
	Hence, we obtain for all
	$ y, h \in \mathcal{C}( \angle_T, V) $ 
	that
	\begin{equation}
	\begin{split}
	\label{eq:rought_limit_estimate}
	&
	\sup_{ (s,t) \in \angle_T }
	\bigg\| ( F(y +h ) ) (s, t) - ( F(y ) )(s, t)
	- 
	\int_s^t   
f_{0,1}
	( \tau, y(s, \tau)  )
	h(s, \tau)
	\, d \tau 
	\bigg\|_V
	\leq
	\sup_{ (s,t) \in \angle_T }
	\| h(s, t) \|_V  
	\\
	&
	\cdot
	\sup_{ (s, t) \in \angle_T }
	\bigg[
	\int_s^t
	\bigg(
	\int_0^1
	\big\|  
f_{0,1}
	( \tau, y(s, \tau) + r h(s, \tau) )
	-  
f_{0,1}
	(\tau, y(s, \tau) )
	\big\|_{ L( V ) }
	\,
	dr
	\bigg)
	\, d \tau
	\bigg]
	\\
	&
	\leq
	T
	\bigg[ 
	\sup_{ (s,t) \in \angle_T }
	\| h(s, t) \|_V 
	\bigg] 
	\bigg[
	\sup_{ r \in [0,1] }
	\sup_{ (s, \tau) \in \angle_T } 
	\|  
f_{0,1}
	( \tau, y(s, \tau) + r h(s, \tau) )
	-  
f_{0,1}
	(\tau, y(s, \tau) )
	\|_{ L( V ) } 
	\bigg]
	.
	\end{split}
	\end{equation}
	Moreover, observe that
	Lemma~\ref{lemma:CompactnessArgument}
	(with
	$ V = V $,
	$ \varepsilon = \varepsilon $,
	$ T = T $,
	$ y = y $,
	$ f = f $
	for 
	$ y \in \mathcal{C}( \angle_T, V ) $,
	$ \varepsilon \in (0,\infty) $
	in the notation of 
	Lemma~\ref{lemma:CompactnessArgument}) 
	shows that
	for all
	$ y \in \mathcal{C}( \angle_T, V ) $,
	$ \varepsilon \in (0, \infty) $
	there exists 
	$ \delta \in ( 0, \infty) $ 
	such that for all
	$ h \in \mathcal{C}( \angle_T, V ) $
	with
	$ \sup_{ (u,\tau) \in \angle_T } 
	\| h(u, \tau) \|_V < \delta $
	it holds that
	\begin{equation} 
	\begin{split} 
	\label{eq:sup_estimate}
	\sup\nolimits_{ r \in [0,1] }
	\sup\nolimits_{ (s, \tau) \in \angle_T } 
	\|  
f_{0,1}
	( \tau, y(s, \tau) + r h(s, \tau) )
	-   
f_{0,1}
	(\tau, y(s, \tau) )
	\|_{ L( V ) } 
	<
	\varepsilon
	.
	\end{split}
	\end{equation}
	Combining this with~\eqref{eq:rought_limit_estimate}
	implies that
	for all
	$ y \in \mathcal{C}( \angle_T, V ) $
	it holds that
	\begin{equation}
	\limsup_{ (\mathcal{C}( \angle_T, V ) \backslash \{0\} ) \ni h \to 0 }
	\tfrac{
		\sup_{ (s,t) \in \angle_T }
		\| ( F(y +h ) ) (s,t) - ( F(y ) )(s,t)
		- \int_s^t   
f_{0,1}
		( \tau, y(s, \tau)  )
		\,
		h(s, \tau)
		\, d \tau 
		\|_V
	}
	{
		\sup_{(s,t) \in \angle_T } \| h(s,t) \|_V
	}
	=
	0.
	\end{equation}
	Lemma~\ref{lemma:WellConti}
	(with
	$ U = V $,
	$ V = V $,
	$ W = V $,
	$ T = T $,
	$ f = ( [0,T] \times V \ni (t,x) \mapsto  
	f_{0,1}(t,x) \in L(V) ) $,  
	$ y = y $,
	$ h = h $
	in the notation of Lemma~\ref{lemma:WellConti})
	therefore proves
	that 
	$ F $ is Fr\'echet differentiable and that
	for all
	$ y, h \in \mathcal{C}( \angle_T, V ) $, 
	$ ( s, t ) \in \angle_T $
	it holds that
	\begin{equation}
	( F'(y)h)(s, t)
	=
	\int_s^t 
f_{0,1}
	( \tau, y( s, \tau ) ) \, h(s, \tau) \, d\tau.
	\end{equation}
	This ensures that for all
	$ y, g, h \in \mathcal{C}( \angle_T, V ) $
	it holds
	that
	\begin{equation}
	\begin{split}
	&
	\sup_{ (s,t) \in \angle_T }
	\| (F'(y+h)g)(s, t) - (F'(y)g)(s, t) \|_V
	\\
	&
	\leq
	\sup_{ (s,t) \in \angle_T }
	\int_s^t
	\|  
f_{0,1}
	(\tau, y(s, \tau) + h(s, \tau))
	\, g(s, \tau)
	-  
f_{0,1}
	(\tau, y(s, \tau) )
	\, g(s, \tau)
	\|_V
	\,
	d\tau
	\\
	&
	\leq
	T
	\sup_{ (s, \tau) \in \angle_T}
	\|  
f_{0,1}
	(\tau, y(s, \tau) + h(s, \tau)) 
	-  
f_{0,1}
	(\tau, y(s, \tau) ) 
	\|_{ L(V) } 
	\sup_{(s, \tau ) \in \angle_T}
	\| g( s, \tau ) \|_V
	.
	\end{split}
	\end{equation}
	Combining
	this with~\eqref{eq:sup_estimate}
	shows that 
	for all
	$ y \in \mathcal{C} ( \angle_T, V ) $
	it holds that
	\begin{equation}
	\limsup\nolimits_{ \mathcal{C}( \angle_T, V ) \ni h \to 0 } 
	\| F'(y+h) -  F'(y)  \|_{ L( \mathcal{C}(\angle_T, V), \mathcal{C}(\angle_T, V) ) }
	=
	0.
	\end{equation}
	The proof of Lemma~\ref{lemma:help_lemma}
	is thus completed.
\end{proof}
\begin{lemma} 
	\label{lemma:converse_chain_rule}	
	Let 
	$ ( X, \left \| \cdot \right \|_X ) $,
	$ ( Y, \left \| \cdot \right \|_Y ) $,
	and
	$ ( Z, \left \| \cdot \right \|_Z ) $
	be 
	nontrivial
	$ \R $-Banach spaces, 
	let  
	$ f \in \mathcal{C}(X, Y) $,
	$ g \in \mathcal{C}^1(Y, Z) $,
	assume that
	$ g \circ f \in \mathcal{C}^1(X,Z) $,
	and
	assume  
	for all $ x \in X $  that
	$ g'( f ( x ) ) $ is bijective
	and that
	$ [ g'(f(x)) ]^{-1} \in L(Z,Y) $.
	Then 
	\begin{enumerate}[(i)]
	\item \label{item:Differentiable} it holds that
	$ f \in \mathcal{C}^1(X,Y) $
	and
	\item \label{item:chainRule} it holds for all $ x \in X $
	that
$ f'(x)
	=
	[ g'( f( x ) ) ]^{-1}
	( g \circ f )'( x ) $. 
\end{enumerate}
\end{lemma}
\begin{proof}[Proof of Lemma~\ref{lemma:converse_chain_rule}]
	Throughout this proof 
	let $ o_1, o_2 \colon X \times X \to Z $
	be the functions which satisfy 
	for all $ x, h \in X $
	that
	\begin{equation}
	\begin{split}
	\label{eq:first_smooth}
	o_1( x,  h )
	=
	g( f( x + h ) ) - g(f(x))
	-
	g' ( f(x) )
	( f(x+h) - f(x) ) 
	\end{split}
	\end{equation}
	and
	\begin{equation}
	\begin{split}
	\label{eq:second_smooth}
	o_2( x, h)
	=
	g(f(x+h))
	-
	g(f(x))
	-
	(g \circ f)'(  x ) h 
	.
	\end{split}
	\end{equation}
	Observe that~\eqref{eq:first_smooth}
	and~\eqref{eq:second_smooth}
	imply
	for all $ x, h \in X $ that
	\begin{equation}
	\begin{split}
	\label{eq:some_ide}
	&
	f(x + h ) - f(x)
	=
	[ g'( f ( x ) ) ]^{-1}
	\big(
	( g \circ f )'(x) h + o_2( x, h)
	-
	o_1( x, h )
	\big)
	\\
	&
	=
	[ g'( f ( x ) ) ]^{-1}
	( g \circ f)'( x ) h
	+ 
	[ g'( f ( x ) ) ]^{-1}
	o_2( x, h)
	-
	[ g'( f ( x ) ) ]^{-1}
	o_1( x, h )
	.
	\end{split}
	\end{equation}
	Moreover, note that
	the fact that $ f $ is continuous 
	and the assumption that 
	$ g $ is differentiable  
	assure 
	that for every $ x \in X $ 
	there exists a function
	$ w \colon X \to [0,\infty) $
	such that
	for every $ h \in X $ it holds that
	$ \limsup_{ X \ni u \to 0 }  w(u) = 0 $
	and
	\begin{equation}
	\label{eq:diff_condition} 
	\| o_1( x, h ) \|_Z = w(h) \cdot \|f(x + h)- f(x) \|_Y 
	.
	\end{equation}
	This shows that for every $ x \in X $ there 
	exists $ \delta \in (0,\infty) $
	such that for all
	$ h \in X $
	with $ \| h \|_X < \delta $
	it holds that
	\begin{equation}
	\|
	[ g'( f ( x ) ) ]^{-1}
	\|_{L(Z, Y)} 
	\| o_1( x, h ) \|_Z
	\leq
	\tfrac{1}{2}
	\| f(x+h) - f(x) \|_Y
	.
	\end{equation}
	Equation~\eqref{eq:some_ide}
	hence
	proves that
	for every $ x \in X $
	there exists $ \delta \in (0,\infty) $
	such that
	for all
	$ h \in X $
	with
	$ \| h \|_X < \delta $
	it holds that
	\begin{equation}
	\begin{split}
	&
	\| f(x + h ) - f(x) \|_Y
	\leq
	\|
	[ g'( f ( x ) ) ]^{-1}
	( g \circ f)'( x )
	\|_{L( X, Y )}
	\| h \|_X
	\\
	&
	+ 
	\|
	[ g'( f ( x ) ) ]^{-1}
	\|_{L(Z, Y)}
	\|
	o_2( x, h)
	\|_Z
	+
	\tfrac{1}{2} 
	\| f(x+h) - f(x) \|_Y
	.
	\end{split}
	\end{equation}
	Therefore, we establish that
	for every $ x \in X $
	there exists $ \delta \in (0,\infty) $
	such that for all
	$ h \in X $ with
	$ \| h \|_X < \delta $
	it holds that
	\begin{equation}
	\begin{split}
	\label{eq:use_later}
	&
	\tfrac{
	\| f(x + h ) - f(x) \|_Y
}{2}
	%
	\leq 
	\|
	[ g'( f ( x ) ) ]^{-1}
	( g \circ f)'( x )
	\|_{L(X, Y)}
	\| h \|_X
	+  
	\|
	[ g'( f ( x ) ) ]^{-1}
	\|_{L(Z, Y)}
	\|
	o_2( x, h )
	\|_Z
	.
	\end{split}
	\end{equation}
	Furthermore, note that
	the fact that 
	\begin{equation}
	\label{eq:diff_condition2}
	\forall  \, x \in X \colon
	\limsup_{( X \backslash \{0\} ) \ni h \to 0} 
	\tfrac{ \| o_2( x, h ) \|_Z }{ \| h \|_X } 
	=
	0
	\end{equation}
	ensures that for every $ x \in X $ 
	there 
	exists $ \delta \in (0,\infty) $
	such that for all
	$ h \in X $
	with $ 0 < \| h \|_X <  \delta $
	it holds that
	\begin{equation}
	\tfrac{
		\|
		[ g'( f ( x ) ) ]^{-1}
		\|_{L(Z, Y)}
		\|
		o_2( x, h )
		\|_Z 
	}
	{ \| h \|_X}
	\leq
	\tfrac{1}{2} 
	.
	\end{equation}
	Equation~\eqref{eq:use_later}
	hence proves that 
	for every $ x \in X $
	there exists $ \delta \in (0,\infty) $
	such that
	for all
	$ h \in X $
	with $ 0 < \| h \|_X < \delta $
	it holds that
	\begin{equation}
	\begin{split}
	\tfrac{ 
		\| f(x + h ) - f(x) \|_Y
	}
	{\| h \|_X}
	&
	\leq 
	2
	\|
	[ g'( f ( x ) ) ]^{-1}
	( g \circ f)'( x )
	\|_{L(X, Y)} 
	+  
	2
	\tfrac{
		\|
		[ g'( f ( x ) ) ]^{-1}
		\|_{L(Z, Y)}
		\|
		o_2( x, h )
		\|_Z 
	}
	{ \| h \|_X}
	\\
	&
	\leq
	2
	\|
	[ g'( f ( x ) ) ]^{-1}
	( g \circ f)'( x )
	\|_{L(X, Y)} 
	+
	1
	. 
	\end{split}
	\end{equation}
	This 
	shows that 
	for every $ x \in X $
	there exists $ \delta \in (0,\infty) $
	such that for all
	$ h \in \{ y \in X \colon f(x+y) \neq f(x) \} $
	with $ 0 < \| h \|_X < \delta $
	it holds that
	\begin{equation}
	\begin{split}
	&
	\tfrac{ \| o_1( x, h ) \|_Z }{ \|f(x+h)- f(x) \|_Y }
	=
	\tfrac{ \| o_1( x, h ) \|_Z }{ \| h \|_X}
	\cdot
	\tfrac{ \| h \|_X }{ \|f(x+h)- f(x) \|_Y }
	\geq
	\tfrac{ \| o_1( x, h ) \|_Z }{ \| h \|_X }
	\cdot
	\tfrac{ 1 }
	{ 
		2
		\|
		[ g'( f ( x ) ) ]^{-1}
		( g \circ f)'( x )
		\|_{L( X, Y)} 
		+
		1 
	}
	.
	\end{split}
	\end{equation}
	Combining this with~\eqref{eq:diff_condition}
	implies that for every $ x \in X $ it holds that
	\begin{equation}
	\label{eq:gives_corollary}
	\limsup_{ ( X \backslash \{0\} ) \ni h \to 0 }
	\tfrac{ \| o_1( x, h ) \|_Z}{ \| h \|_X }
	=
	0.
	\end{equation}
	Equation~\eqref{eq:some_ide}
	therefore
	establishes
	that for 
	every $ x \in X $ 
	there exists $ \delta \in (0,\infty) $ such that 
	for all
	$ h \in X $
	with
	$ 0 < \| h \|_X < \delta $
	it holds that
	\begin{equation}
	\begin{split}
	&
	\tfrac{
		\| f(x + h ) 
		- 
		f(x)  
		-
		[ g'( f ( x ) ) ]^{-1} 
		( g \circ f )'(x) h  
		\|_Y
	}
	{ \| h \|_X }
	=
	\tfrac{
		\| 
		[ g'( f ( x ) ) ]^{-1} 
		( o_2( x, h )
		-
		o_1( x, h ) ) 
		\|_Y
	}
	{
		\| h \|_X }
	\\
	&	 
	\leq
	\| 
	[ g'( f ( x ) ) ]^{-1} 
	\|_{L(Z, Y)}
	\tfrac{ 
		\| 
		o_2( x, h )
		-
		o_1( x, h )  
		\|_Z
	}
	{
		\| h \|_X }
	\leq
	\| 
	[ g'( f ( x ) ) ]^{-1} 
	\|_{L(Z, Y)}
	\Big[
	\tfrac{ 
		\| o_2( x, h ) \|_Z
	}
	{ 
		\| h \|_X
	}
	+
	\tfrac{ \| 
		o_1( x, h )  
		\|_Z
	}
	{
		\| h \|_X }
	\Big]
	.
	\end{split}
	\end{equation}
	This, \eqref{eq:diff_condition2},
	and~\eqref{eq:gives_corollary} 
	demonstrate
	that $ f $ is differentiable 
	and that
	item~\eqref{item:chainRule}  
	holds.
	The fact that
	$ \{ B \in L( Y, Z ) \colon B \text{ is invertible} \}
	\ni C \mapsto C^{-1} 
	\in \{ B \in L( Z, Y ) \colon B \text{ is invertible} \} $
	is continuous
	(cf., e.g., Deitmar \& Echterhoff~\cite[Lemma~2.1.5]{DeitmarEchterhoff2014})
	and
	the fact that
	$ X \ni x \mapsto g'(f(x)) \in L(Y,Z) $
	is continuous
	assure that
	$ X \ni x \mapsto [ g'(f(x)) ]^{-1} \in L(Z,Y) $
	is continuous.
	Item~\eqref{item:chainRule} hence establishes item~\eqref{item:Differentiable}.
	The proof of Lemma~\ref{lemma:converse_chain_rule}
	is thus completed.
\end{proof}
\begin{lemma}
	\label{lemma:existence_uniqueness}
	Let $ (V, \left\| \cdot \right\|_V) $ 
	be an $ \R $-Banach space 
	and let
	$ T \in (0,\infty) $, 
	$ \angle_T = \{ (s, t) \in [0,T]^2 \colon s \leq t \} $,
	$ \phi \in \mathcal{C}( \angle_T, V ) $,
	$ A \in \mathcal{C}( \angle_T, L(V) ) $.
	Then there exists a unique function
	$ y \in \mathcal{C}(\angle_T, V ) $
	such that for all 
	$ (s,t) \in \angle_T $ 
	it holds that
	\begin{equation}
	\label{eq:solution}
	y(s, t) = \phi(s, t) + \int_s^t A(s, \tau) y(s, \tau) \, d \tau.
	\end{equation}
\end{lemma}
\begin{proof}[Proof of Lemma~\ref{lemma:existence_uniqueness}]
	Throughout this proof let
	$ \mathcal{A} \colon \R \times [0,T] \to L(V) $
	be a  continuous function
	which satisfies for all
	$ (s,\tau) \in \angle_T $ that
	$ \mathcal{A}(s,\tau) = A(s,\tau) $,
	let
	$ f_n \colon \mathcal{C}( \angle_T, V ) 
	\to \mathcal{C}( \angle_T, V ) $,
	$ n \in \N $,
	be the functions which satisfy for all
	$ n \in \N $,
	$ x \in \mathcal{C}( \angle_T, V ) $,
	$ (s,t ) \in \angle_T $
	that
	$ ( f_1 ( x ) ) (s,t) = \phi(s,t) + \int_s^t A ( s, \tau ) x ( s, \tau ) \, d \tau $ 
	(see Lemma~\ref{lemma:WellConti} (with
	$ U = \R $,
	$ V = V $,
	$ W = V $,
	$ T = T $,
	$ f = ( [0, T] \times \R \ni ( \tau, s ) 
	\mapsto \mathcal{A} (s, \tau ) \in L(V) ) $,
	$ y = ( \angle_T \ni ( s, \tau )
	\mapsto 
	s \in \R ) $,
	$ h = x $
	for
	$ x \in \mathcal{C} ( \angle_T, V ) $
	in the notation of Lemma~\ref{lemma:WellConti}))
	and
	$ ( f_{n+1}(x) )(s,t) 
	= ( f_n(f_1(x)) ) (s,t) $,
	and let
	$ \mathcal{N} \subseteq \N $
	satisfy that
	\begin{equation}  
	\begin{split} 
	\mathcal{N} =
	\left\{ 
	\begin{array}{rr} 
	n \in \N \colon
	\big(
	\forall \, x, y \in \mathcal{C}( \angle_T, V ), 
	(s,t) \in \angle_T
	\colon 
	\| ( f_n(x) )( s,t) - ( f_n (y) )( s,t) \|_V 
	\\
	\leq 
	\tfrac{
		[\sup_{ (u, \tau) \in \angle_T } 
		\| A( u, \tau ) \|_{L(V)}
		(t-s) ]^{ n }
	} { n ! } 
	\sup_{ ( u, \tau ) \in \angle_T } \| x( u, \tau ) - y( u, \tau ) \|_V
	\big)
	\end{array}
	\right\} 
	.
	\end{split}
	\end{equation}
	Note that for all
	$ x, y \in \mathcal{C}( \angle_T, V ) $,
	$ (s,t) \in \angle_T $
	it holds that
	\begin{equation}
	\begin{split}
	& 
	\| ( f_1 (x) )( s, t ) - ( f_1 (y) )( s, t ) \|_V
	=
	\Big\|
	\int_s^t
	A( s, \tau ) [ x( s, \tau ) - y( s, \tau ) ] \, d\tau
	\Big\|_V
	\\
	&
	\leq
	\int_s^t
	\| A( s, \tau ) [ x( s, \tau ) - y( s, \tau ) ] \|_V \, d\tau
	\\
	&
	\leq
	(t-s)
	\sup\nolimits_{ (u,\tau) \in \angle_T } 
	\| A(u, \tau) \|_{L(V)}
	\sup\nolimits_{ (u,\tau) \in \angle_T}
	\| x(u,\tau) - y(u,\tau) \|_V
	.
	\end{split}
	\end{equation}
	This implies that $ 1 \in \mathcal{N} $
	and that
	$ f_1 $ is continuous.
	Moreover, observe that for all
	$ n \in \mathcal{N} $,
	$ x, y \in \mathcal{C}( \angle_T, V ) $,
	$ (s,t) \in \angle_T $
	it holds that
	\begin{equation}
	\begin{split}
	& 
	\| ( f_{n+1}(x) )( s, t ) - ( f_{n+1}(y) )( s, t ) \|_V
	\\
	&
	= 
	\Big\|
	\int_s^t
	A( s, \tau ) [ ( f_n(x) ) (s, \tau) - ( f_n(y) )(s, \tau) ] \, d\tau
	\Big\|_V
	\\
	&
	\leq
	\int_s^t
	\| A( s, \tau ) [ ( f_n (x) ) (s,\tau) - ( f_n (y) )(s,\tau) ] \|_V \, d\tau
	\\
	&
	\leq 
	\sup_{ (u,\tau) \in \angle_T } 
	\| A(u, \tau) \|_{L(V)}
	\int_s^t
	\| ( f_n(x) ) (s, \tau) - ( f_n(y) )(s, \tau)  \|_V \, d\tau
	\\
	&
	\leq 
	\sup_{ (u, \tau) \in \angle_T } 
	\| A(u, \tau) \|_{L(V)}
	\sup_{ (u,\tau) \in \angle_T } \| x(u, \tau) - y(u, \tau) \|_V
	\int_s^t
	\tfrac{
		[
		\sup_{ (u,v) \in \angle_T } 
		\| A( u, v ) \|_{L(V)}
		(\tau-s)
		]^n
	} { n ! } 
	\, d\tau
	\\
	&
	=
	\tfrac{
		[ \sup_{ (u, \tau) \in \angle_T } 
		\| A( u, \tau ) \|_{L(V)}
		(t-s) ]^{ n + 1 }
	} { ( n + 1 ) ! } 
	\sup\nolimits_{ (u, \tau) \in \angle_T } \| x( u,\tau ) - y( u,\tau ) \|_V
	.
	\end{split}
	\end{equation}
	The fact that $ 1 \in \mathcal{N} $ hence shows that $ \mathcal{N} = \N $.
	This proves that for all
	$ n \in \N $,
	$ x, y \in \mathcal{C}( \angle_T, V ) $,
	$ (s,t) \in \angle_T $
	it holds that
	\begin{equation}
	\label{eq:m=0}
	\| ( f_n(x) )( s, t ) - ( f_n(y) )( s, t ) \|_V 
	\leq
	\tfrac{
		[ \sup_{ (u, \tau) \in \angle_T} 
		\| A( u, \tau ) \|_{L(V)}
		(t-s) ]^{ n }
	} { n ! } 
	\sup\nolimits_{ (u, \tau ) \in \angle_T } 
	\| x( u, \tau ) - y( u, \tau ) \|_V
	.
	\end{equation}
	In particular,
	for all
	$ m, n \in \N $, 
	$ x \in \mathcal{C}( \angle_T, V ) $,
	$ (s,t ) \in \angle_T $
	it holds that
	\begin{equation}
	\begin{split}  
	&
	\| ( f_n(x) )( s, t ) - ( f_{n + 1}(x) )( s, t ) \|_V 
	\\
	&
	\leq
	\tfrac{
		[ \sup_{ (u, \tau) \in \angle_T} 
		\| A( u, \tau ) \|_{L(V)}
		(t-s) ]^{ n }
	} { n ! } 
	\sup\nolimits_{ (u, \tau ) \in \angle_T } 
	\| x( u, \tau ) - (f_1(x))( u, \tau ) \|_V
	.
	\end{split} 
	\end{equation}
	Therefore, we establish that
	for all
	$ m, n \in \N $,
	$ x, y \in \mathcal{C}( \angle_T, V ) $ 
	it holds that
	\begin{equation}
	\begin{split}
	\label{ex:Sum_estimate}
	& 
	\sup\nolimits_{ (s,t ) \in \angle_T }
	\| ( f_n(x) )( s, t ) - ( f_{m+n}(y) )( s, t ) \|_V 
	\\
	&
	\leq
	\sum_{k=0}^{m-1} 
	\sup\nolimits_{ (s,t ) \in \angle_T }
	\| ( f_{n+k}(x) )( s, t ) - ( f_{n + k + 1}(y) )( s, t ) \|_V 
	\\
	&
	\leq 
	\big( 
	\sup\nolimits_{ (u, \tau) \in \angle_T } 
	\| x(u, \tau) - ( f_1(x) )(u, \tau) \|_V
	\big) 
	\sum_{k = 0}^{m - 1}
	\tfrac{ [ \sup_{ (u, \tau) \in \angle_T } 
		\| A(u, \tau) \|_{L(V)}
		T ]^{ n + k } }
	{ ( n + k ) ! }
	\\
	&
	\leq 
		\tfrac{
	( 
	\sup_{ (u, \tau) \in \angle_T }  
	\| x(u, \tau) - ( f_1(x) )(u, \tau) \|_V
	)
	[ \sup\nolimits_{ (u, \tau) \in \angle_T } 
	\| A(u, \tau) \|_{L(V)}
	T  ]^{ n }
}{ n! }
	\sum_{k = 0}^{m - 1}
	\tfrac{ [ \sup_{ (u, \tau) \in \angle_T } 
		\| A(u, \tau) \|_{L(V)}
		T ]^{ k } }
	{ k ! }
	\\
	&
	\leq 
	\tfrac{
		( 
		\sup_{ (u, \tau) \in \angle_T }  
		\| x(u, \tau) - ( f_1(x) )(u, \tau) \|_V
		)
		[ \sup\nolimits_{ (u, \tau) \in \angle_T } 
		\| A(u, \tau) \|_{L(V)}
		T  ]^{ n }
	}{ n! }
	\exp\big( 
	\sup\nolimits_{ (u, \tau) \in \angle_T } 
	\| A( u, \tau ) \|_{L(V)}
	T
	\big)
	.
	\end{split}
	\end{equation}
	This ensures that for every
	$ x \in \mathcal{C}( \angle_T, V ) $
	the sequence
	$ ( f_n(x) )_{n \in \N} \subseteq \mathcal{C}( \angle_T, V ) $
	is Cauchy.
	The fact that
	$ \mathcal{C}( \angle_T, V ) $ is a 
	Banach space
	hence shows that for every
	$ x \in \mathcal{C}( \angle_T, V ) $
	there exists 
	$ z \in \mathcal{C}( \angle_T, V ) $
	such that
	\begin{equation} 
	\label{eq:lim_exists}
	\limsup\nolimits_{n \to \infty} 
	\big(
	\sup\nolimits_{(s,t) \in \angle_T }  
	\| (f_n(x))(s,t) - z(s,t) \|_V 
	\big) 
	= 
	0 
	.
	\end{equation}
	\sloppy 
	Moreover,  
	note that
	for all
	$ x, z \in \mathcal{C}( \angle_T, V ) $
	with
	$ \limsup_{n \to \infty} (
	\sup_{(s,t) \in \angle_T }  
	\| f_n(x)(s,t) - z(s,t) \|_V ) = 0 $
	it holds that
	\begin{equation} 
	\begin{split} 
	&
	\sup\nolimits_{(s,t) \in \angle_T }
	\| ( f_1(z) )(s,t) - z(s,t) \|_V
	\\
	&
	\leq
	\limsup\nolimits_{n \to \infty} 
	\big(
	\sup\nolimits_{(s,t) \in \angle_T }   
	\| ( f_1(z) ) (s,t) - ( f_1(f_n(x)) ) (s,t)  \|_V
	\big)
	\\
	&
	\quad
	+
	\limsup\nolimits_{n \to \infty} 
	\big(
	\sup\nolimits_{(s,t) \in \angle_T }  
	\| ( f_1(f_n(x)) )(s,t) - z(s,t) \|_V
	\big)
	\\
	&
	=
	\limsup\nolimits_{n \to \infty} 
	\big(
	\sup\nolimits_{(s,t) \in \angle_T }   
	\| ( f_1(z) )(s,t) - ( f_1(f_n(x)) )(s,t) \|_V
	\big)
	\\
	&
	\quad
	+
	\limsup\nolimits_{n \to \infty} 
	\big(
	\sup\nolimits_{(s,t) \in \angle_T }  
	\| (f_{n+1}(x))(s,t) - z(s,t) \|_V
	\big)
	\\
	&
	=
	\limsup\nolimits_{n \to \infty} 
	\big(
	\sup\nolimits_{(s,t) \in \angle_T }   
	\| ( f_1(z) )(s,t) - ( f_1(f_n(x)) )(s,t) \|_V
	\big)
	. 
	\end{split} 
	\end{equation}
	Combining this with~\eqref{eq:lim_exists}
	and
	the fact that
	$ f_1 \colon \mathcal{C}( \angle_T, V ) 
	\to \mathcal{C}( \angle_T, V ) $ is continuous  
	demonstrates that there exists $ z \in \mathcal{C}( \angle_T, V ) $
	such that 
	\begin{equation} 
	\label{eq:ProvesUniquenes}
	f_1(z) = z.
	\end{equation} 
	In addition, note that~\eqref{eq:m=0}
	implies that for all
	$ N \in \N $,
	$ x,y \in \mathcal{C}( \angle_T, V ) $
	with
	$ f_1(x) = x $,
	$ f_1(y) = y $
	it holds that
	\begin{equation}
	\begin{split}
	&
	\sum_{n=1}^N
	\sup_{ (s,t ) \in \angle_T }
	\| x( s, t ) - y( s, t ) \|_V 
	=
	\sum_{n=1}^N
	\sup_{ (s,t ) \in \angle_T }
	\| ( f_n(x) )( s, t ) - ( f_n(y) )( s, t ) \|_V 
	\\
	&
	\leq
	\sum_{n=1}^N
		\tfrac{
		[ \sup_{ (u, \tau) \in \angle_T} 
		\| A( u, \tau ) \|_{L(V)}
		(t-s) ]^{ n }
	} { n ! } 
	\sup\nolimits_{ (u, \tau ) \in \angle_T } 
	\| x( u, \tau ) - y( u, \tau ) \|_V
	\\
	&
	\leq
	\big(
	\sup\nolimits_{ (u, \tau) \in \angle_T } \| x( u, \tau ) - y( u, \tau ) \|_V 
	\big)
	\exp\big( 
	\sup\nolimits_{ (u, \tau) \in \angle_T } 
	\| A( u, \tau ) \|_{L(V)}
	T
	\big)
	< \infty
	.
	\end{split}
	\end{equation}
	This and~\eqref{eq:ProvesUniquenes} show that
	there exists a unique
	$ z \in \mathcal{C}( \angle_T, V) $
	such that
	$ f_1(z) = z $.
	The proof of Lemma~\ref{lemma:existence_uniqueness}
	is thus completed.
\end{proof}
\begin{lemma}
	\label{lemma:bound}
	Let $ (V, \left\| \cdot \right\|_V) $ 
	be an $ \R $-Banach space 
	and let
	$ T \in (0,\infty) $, 
	$ \angle_T = \{ (s, t) \in [0,T]^2 \colon s \leq t \} $,
	$ \phi \in \mathcal{C}( \angle_T, V ) $,
	$ A \in \mathcal{C}( \angle_T, L(V) ) $, 
	$ y \in \mathcal{C}(\angle_T, V ) $
	satisfy for all
	$ (s,t) \in \angle_T $ that
	\begin{equation} 
	y(s, t) = \phi(s, t) + \int_s^t A(s, \tau) y(s, \tau) \, d \tau.
	\end{equation}
	Then 
	\begin{equation}
	\sup\nolimits_{(s,t) \in \angle_T }
	\| y(s,t) \|_V 
	\leq
	\big(
	\sup\nolimits_{(s,t) \in \angle_T }
	\| \phi(s,t) \|_V
	\big)
	\exp 
	\big(
	T
	\sup\nolimits_{(s,t) \in \angle_T }
	\| A( s,t) \|_{L(V)}
	\big)
	.
	\end{equation}
\end{lemma}
\begin{proof}[Proof of Lemma~\ref{lemma:bound}]
	Note that for all
	$ (s,t) \in \angle_T $
	it holds that
	\begin{equation}
	\begin{split} 
	\| y(s,t) \|_V 
	&
	\leq 
	\| \phi(s,t) \|_V 
	+
	\int_s^t \| A( s, \tau ) y( s, \tau ) \|_V \, d \tau 
	\\
	&
	\leq
	\| \phi(s,t) \|_V 
	+
	\sup\nolimits_{(u,v) \in \angle_T }
	\| A(u, v ) \|_{L(V)}
	\int_s^t \| y (s, \tau) \|_V \, d \tau 
	\\
	&
	\leq 
	\sup\nolimits_{ (u,v) \in \angle_T } 
	\| \phi( u, v ) \|_V 
	+
	\sup\nolimits_{(u,v) \in \angle_T }
	\| A(u, v ) \|_{L(V)}
	\int_s^t \| y (s, \tau) \|_V \, d \tau 
	< \infty.
	\end{split} 
	\end{equation} 
	Gronwall's lemma hence shows that
	for all
	$ (s,t ) \in \angle_T $
	it holds that
	\begin{equation}
	\| y( s,t) \|_V 
	\leq 
	\big(
	\sup\nolimits_{(u,v) \in \angle_T}
	\| \phi(u,v) \|_V 
	\big)
	\exp \big( 
	\sup\nolimits_{(u,v) \in \angle_T} 
	\| A(u,v) \|_{L(V)}
	(t-s)
	\big)
	.
	\end{equation}
	The proof of Lemma~\ref{lemma:bound}
	is thus completed.
\end{proof}  
\begin{lemma}
	\label{lemma:diff_initial_value}
	Let
	$ (V, \left\| \cdot \right\|_V) $ 
	be a nontrivial $ \R $-Banach space,
	let
	$ T \in (0,\infty) $, 
	$ \angle_T = \{ (s, t) \in [0,T]^2 \colon s \leq t \} $, 
	$ f \in \mathcal{C}^{0,1}( [0,T] \times V, V) $,   
	let
	$ f_{0,1} \colon [0,T] \times V \to L( V ) $
	be the function 
	which satisfies 
	for all
	$ t \in [0,T] $, $ x \in V $ that
	$ f_{0,1}(t,x) = ( \tfrac{ \partial }{\partial x } f )(t,x) $,
	and 
	for every
	$ x \in V $ 
	let
	$ ( X_{s,t}^x )_{ (s,t) \in \angle_T }
	\colon \angle_T 
	\to V $
	be a continuous function
	which satisfies
	for all 
	$ (s,t) \in \angle_T $ 
	that
	$ X_{s, t}^x = 
	x
	+
	\int_s^t f( \tau, X_{s, \tau}^x) \,d\tau $. 
	Then
	\begin{enumerate}[(i)]
		\item \label{item:continuity}
		it holds 
		that
		$ ( \angle_T \times V \ni (s,t,x) \mapsto X_{s, t}^x \in V ) \in \mathcal{C}^{0,0,1}
		( \angle_T \times V, V ) $
		and
		\item \label{item:identity} 
		it holds for all 
		$ x,y \in V $, $ ( s, t ) \in \angle_T $ that
		\begin{equation}
		 \big( 
		 \tfrac{\partial}{\partial x }
		  X_{s, t}^x
		  \big)  y
		=
		y
		+
		\int_s^t 
		f_{0,1}
		(\tau, X_{s, \tau}^x)
		\big(
		\tfrac{\partial}{\partial x } X_{s, \tau}^x 
		\big)
		y
		\, d\tau
		.
		\end{equation}
	\end{enumerate}
\end{lemma}
\begin{proof}[Proof of Lemma~\ref{lemma:diff_initial_value}]
	\sloppy
	Throughout this proof  
	let 
	$ F \colon \mathcal{C}( \angle_T, V )
	\to  
	\mathcal{C}( \angle_T, V ) $
	and
	$ G, H \colon V \to \mathcal{C}( \angle_T, V ) $
	be the functions
	which satisfy for all 
	$ v \in V $,  
	$ (s,t) \in \angle_T $,
	$ z \in \mathcal{C}( \angle_T, V ) $ 
	that
	$ ( F ( z ) )(s,t) = z(s, t) - \int_s^t f ( \tau, z( s, \tau) ) \, d\tau $,
	$ ( G ( v ) )(s,t) = X_{s, t}^v $,
	and
	$ ( H ( v ) )(s, t) = v $.
	Note that
	Lemma~\ref{lemma:help_lemma}
	(with
	$ V = V $,
	$ T = T $, 
	$ f = f $, 
	$ F = ( 
	\mathcal{C}( \angle_T, V) \ni z \mapsto
	(
	\angle_T \ni (s,t) \mapsto \int_s^t f( \tau, z(s, \tau) ) \, d \tau
	)
	\in 
	\mathcal{C}( \angle_T, V )
	)  $,
	$ f_{0,1} = f_{0,1} $
	in the notation of
	Lemma~\ref{lemma:help_lemma})
	proves that
	for all   
	$ y, h \in \mathcal{C}( \angle_T, V ) $,
	$ (s,t) \in \angle_T $
	it holds that
	$ F \in \mathcal{C}^1( \mathcal{C}( \angle_T, V ), 
	\mathcal{C}( \angle_T, V ) ) $
	and
	\begin{equation} 
	\label{eq:first_der_identity}
	( F'( y ) h )(s, t)
	=
	h( s, t ) - \int_s^t   f_{0,1}   
	( \tau, y ( s, \tau) ) h ( s, \tau ) \, d\tau 
	.
	\end{equation}
	Therefore, we obtain that
	for all   
	$ v \in V $,
	$ h \in \mathcal{C}( \angle_T, V ) $,
	$ (s,t) \in \angle_T $
	it holds that
	\begin{equation} 
	h( s, t )
	=
	( F'( G( v ) ) h )(s, t)
	+
	\int_s^t f_{0,1}  ( \tau, G( v ) ( s, \tau) ) 
	h( s, \tau ) \, d\tau 
	.
	\end{equation}
	Combining this with
	Lemma~\ref{lemma:existence_uniqueness}
	(with
	$ V = V $,
	$ T = T $, 
	$ \phi = ( \angle_T \ni (s, \tau )
	\mapsto  ( F'( G( v ) ) h ) ( s, \tau ) \in V ) $,
	$ A = \big( \angle_T \ni ( s, \tau ) 
	\mapsto 
	f_{0,1}  ( \tau, G( v ) ( s, \tau ) ) \in L(V) \big) $ 
	for 
	$ h \in \mathcal{C}( \angle_T, V ) $,
	$ v \in V $ 
	in the notation of
	Lemma~\ref{lemma:existence_uniqueness})
	shows that
	for every
	$ v \in V $ it holds that 
	$ F'( G( v ) ) \in L
	(
	\mathcal{C} ( \angle_T, V ),
	\mathcal{C} ( \angle_T, V )
	) $ 
	is invertible.
	In addition, 
	Lemma~\ref{lemma:bound}
	(with
	$ V = V $,
	$ T = T $, 
	$ \phi = ( \angle_T \ni (s, \tau )
	\mapsto  ( F'( G( v ) ) h ) ( s, \tau ) \in V ) $,
	$ A = \big( \angle_T \ni ( s, \tau ) 
	\mapsto 
	f_{0,1} ( \tau, G( v ) ( s, \tau ) ) \in L(V) \big) $,
	$ y = h $ 
	for 
	$ h \in \mathcal{C}( \angle_T, V ) $,
	$ v \in V $ 
	in the notation of Lemma~\ref{lemma:bound})
	ensures that for every $ v \in V $ it holds that
	\begin{equation} 
	\label{eq_bounded_inverse}
	[ F'(G(v))]^{-1} \in L
	(
	\mathcal{C} ( \angle_T, V ),
	\mathcal{C} ( \angle_T, V )
	) .
	\end{equation}
	Moreover, observe that for all
	$ v \in V $, 
	$ (s,t) \in \angle_T $
	it holds that
	\begin{equation}
	\begin{split} 
	\label{eq:important_identity}
	( F ( G ( v ) ) ) ( s, t )
	&
	=
	( G ( v ) )(s, t) 
	- 
	\int_s^t f ( \tau, ( G ( v ) )(s, \tau) ) \, d\tau
	\\
	&
	=
	X_{s, t}^v
	-
	\int_s^t
	f( \tau, X_{s, \tau}^v ) 
	\, d\tau
	=
	v
	=
	( H( v ) ) ( s, t )
	.
	\end{split}
	\end{equation}
	Next we intend to prove that
	\begin{equation}
	\label{eq:Important}
	G \in \mathcal{C}( V, \mathcal{C}( \angle_T, V) ) 
	.
	\end{equation}
	For this note that
	Corollary~\ref{Corollary:UnifCont}
	(with
	$ V = V $,
	$ T = T $,
	$ f = f $,
	$ X_{s,t}^x = X_{s,t}^x $
	for
	$ (s,t) \in \angle_T $,
	$ x \in V $
	in the notation of 
	Corollary~\ref{Corollary:UnifCont})
	shows that
	$ \angle_T \times V 
	\ni
	(s,t,x)
	\mapsto
	X_{s,t}^x \in V $
	is continuous. 
	This implies that there exists
	a function
	$ \delta \colon (0, \infty) \times \angle_T \times V \to (0, \infty) $
	such that for all
	$ \varepsilon \in (0, \infty) $,
	$ (s_0, t_0), (s,t) \in \angle_T $,
	$ v_0, v \in V $
	with
	$ | s_0 - s | + |t_0 - t |+ \| v_0 - v \|_V <
	\delta_{ s_0, t_0, v_0 }^{\varepsilon} $
	it holds that
	\begin{equation} 
	\label{eq:SomeCont}
	\| X_{s_0, t_0}^{v_0} - X_{s,t}^v \|_V 
	<
	\nicefrac{ \varepsilon }{2 } 
	.
	\end{equation}
	The fact that for all
	$ v_0 \in V $,
	$ \varepsilon \in (0, \infty) $ it holds that
	\begin{equation}
	\angle_T 
	\subseteq 
	\cup_{ (s_0, t_0) \in \angle_T }
	\{ (s,t) \in \angle_T \colon 
    | s_0 - s |
    +
    | t_0 - t | 
    < 
    \nicefrac{ \delta_{ s_0, t_0, v_0 }^{\varepsilon}
    }{2}
    \}
	\end{equation}
	and the compactness of $ \angle_T $
	prove that for all
	$ v_0 \in V $,
	$ \varepsilon \in (0, \infty) $ 
	there exist 
	$ N \in \N $,
	$ (s_1, t_1), \ldots, (s_N, t_N) \in \angle_T $ such that
	\begin{equation}
	\angle_T 
	=
	\cup_{ n = 1 }^N
	\{ (s,t) \in \angle_T \colon 
	| s_n - s |
	+
	| t_n - t | 
	< 
	\nicefrac{ \delta_{ s_n, t_n, v_0 }^{\varepsilon}
	}{2}
	\}.
	\end{equation} 
	Combining this with~\eqref{eq:SomeCont}
	demonstrates that for all
	$ v_0 \in V $,
	$ \varepsilon \in (0, \infty) $
	there exist
	$ N \in \N $,
	$ (s_1, t_1), \ldots, (s_N, t_N) \in \angle_T $
	such that for all
	$ v \in V $,
	$ (s, t) \in \angle_T $
	with
	$ \| v_0 - v \|_H 
	<
	\frac{1}{2}
	\min_{ n \in [1,N] \cap \N } \delta_{s_n, t_n, v_0}^{  \varepsilon  } $
	it holds that
	\begin{equation}
	\| X_{s, t}^{v_0} - X_{s, t}^v \|_H
	\leq
	\min\nolimits_{ n \in [1,N] \cap \N } 
	\big( 
	\| X_{s, t}^{ v_0 } - X_{ s_n, t_n }^{ v_0 } \|_H
	+
	\| X_{ s_n, t_n }^{ v_0 } - X_{s, t}^v \|_H
	\big) 
	<
	\nicefrac{ \varepsilon }{2 }
	+
	\nicefrac{\varepsilon}{2}
	=
	\varepsilon
	.
	\end{equation}
	This ensures that~\eqref{eq:Important} holds.
	Furthermore, note that 
	$ H $ is continuously differentiable and that
	for all  
	$ v, w \in V $ it holds that
	\begin{equation} 
	\label{eq:H_diff} 
	H'(v) w = ( \angle_T \ni (s, \tau) \mapsto w \in V ) 
	.
	\end{equation}
	Equation~\eqref{eq:important_identity}
	hence shows that 
	$ F \circ G $ is continuously differentiable.
	This, \eqref{eq_bounded_inverse}, 
	\eqref{eq:Important},
	the facts that
	for every 
	$ v \in V $
	it holds that
	$ F $ is continuously differentiable
	and 
	$ F'( G( v ) ) \in L
	(
	\mathcal{C} ( \angle_T, V ),
	\mathcal{C} ( \angle_T, V )
	) $ 
	is invertible,
	and
	Lemma~\ref{lemma:converse_chain_rule}
	(with
	$ X = V $,
	$ Y = \mathcal{C}( \angle_T, V ) $,
	$ Z = \mathcal{C}( \angle_T, V ) $,  
	$ f = G $,
	$ g = F $
	for  
	$ v \in V $ 
	in the notation of Lemma~\ref{lemma:converse_chain_rule})
	imply that
	$ G $ is continuously differentiable and that 
	for every $ v \in V $ it holds that
	\begin{equation}
	\label{eq:crucial_identity}
	G'(v) = [ F'(G(v)) ]^{-1} H'(v) 
	.
	\end{equation} 
	Hence, we obtain that
	\begin{equation}
	\label{eq:solution_diff}
	( V \ni x \mapsto X^x \in \mathcal{C}(\angle_T, V) )
	\in 
	\mathcal{C}^1 ( V,
	\mathcal{C}( \angle_T, V ) )
	.
	\end{equation}
	This 
	implies that
	for every $ x \in V $, $ \varepsilon \in (0,\infty) $
	there exists $ \Delta \in (0,\infty) $ such that for all
	$ s \in [0,T] $,
	$ t, u \in [s,T] $,
	$ \tau \in [u,T] $,
	$ y \in V $ 
	with 
	$ | s-u| + |t-\tau | + \| x-y \|_V < \Delta $ 
	it holds that
	\begin{equation}
	\begin{split}
	\label{eq:joint_cont}
	&
	\| X_{s,t}^x - X_{u,\tau}^y \|_V
	\leq 
	\| X_{s,t}^x - X_{s,\tau}^x \|_V
	+
	\| X_{s,\tau}^x - X_{u,\tau}^x \|_V
	+
	\| X_{u,\tau}^x - X_{u,\tau}^ y \|_V
	\\
	&
	\leq
	\| X_{s,t}^x - X_{s,\tau}^x \|_V
	+
	\| X_{s,\tau}^x - X_{u,\tau}^x \|_V
	+
	\sup\nolimits_{(v,r ) \in \angle_T}
	\| X_{v,r}^x - X_{v,r}^ y \|_V
	< \varepsilon.
	\end{split}
	\end{equation}
	Moreover, 
	observe that~\eqref{eq:solution_diff}
	proves that
	for every $ x \in V $, $ \varepsilon \in (0,\infty) $
	there exists $ \Delta \in (0,\infty) $ such that for all
	$ s \in [0,T] $,
	$ t, u \in [s,T] $,
	$ \tau \in [u,T] $,
	$ y \in V $ 
	with 
	$ | s-u| + |t-\tau | + \| x-y \|_V < \Delta $ 
	it holds that
	\begin{equation}
	\begin{split}
	&
	\big\| 
	\tfrac{\partial}{\partial x} X_{s,t}^x 
	- 
	\tfrac{\partial}{\partial x} X_{u,\tau}^y 
	\big\|_{ L(V) }
	\\
	&
	\leq 
	\big\| 
	\tfrac{\partial}{\partial x} X_{s,t}^x 
	- 
	\tfrac{\partial}{\partial x} X_{s,\tau}^x 
	\big\|_{ L(V) }
	+
	\big\| 
	\tfrac{\partial}{\partial x} X_{s,\tau}^x 
	- 
	\tfrac{\partial}{\partial x} X_{u,\tau}^x 
	\big\|_{ L(V) }
	+
	\big\| 
	\tfrac{\partial}{\partial x} X_{u,\tau}^x 
	- 
	\tfrac{\partial}{\partial x} X_{u,\tau}^ y 
	\big\|_{ L(V) }
	\\
	&
	\leq
	\big\| 
	\tfrac{\partial}{\partial x} X_{s,t}^x 
	- 
	\tfrac{\partial}{\partial x} X_{s,\tau}^x 
	\big\|_{ L(V) }
	+
	\big\| 
	\tfrac{\partial}{\partial x} X_{s,\tau}^x 
	- 
	\tfrac{\partial}{\partial x} X_{u,\tau}^x 
	\big\|_{ L(V) }
	+
	\sup_{(v,r ) \in \angle_T}
	\big\| 
	\tfrac{\partial}{\partial x} X_{v,r}^x 
	- 
	\tfrac{\partial}{\partial x} X_{v,r}^ y 
	\big\|_{ L(V) }
	< 
	\varepsilon.
	\end{split}
	\end{equation}
	This and~\eqref{eq:joint_cont}
	establish item~\eqref{item:continuity}. 
	In addition, 
	note that~\eqref{eq:H_diff} and~\eqref{eq:crucial_identity} 
	ensure that
	for all $ v, w \in V $,  
	$ (s,t) \in \angle_T $
	it holds that
	$ ( F'( G( v ) ) G'( v ) w ) ( s, t )= w  $.
	Display~\eqref{eq:first_der_identity} 
	hence demonstrates that for all 
	$ (s,t) \in \angle_T $,
	$ v, w \in V $ 
	it holds that
	\begin{equation}
	\begin{split}
	w 
	&
	=
	( F'( G( v ) )  G'( v ) w ) ( s, t )
	=
	( G'( v ) w )( s, t )
	-
	\int_s^t 
	f_{0,1}
	(\tau, G( v )( \tau ) ) ( G'( v ) w)( s, \tau ) \, d\tau
	\\
	&
	= 
	\tfrac{\partial}{\partial x}
	X_{s, t}^{v}
	w 
	-
	\int_s^t 
	f_{0,1}
	(\tau, X_{s, \tau}^v ) 
	\big( \tfrac{\partial}{\partial x} X_{s, \tau}^v 
	w 
	\big)   \, d\tau
	.
	\end{split}
	\end{equation}
	The
	proof  
	of Lemma~\ref{lemma:diff_initial_value}
	is thus completed.
\end{proof}
\begin{lemma}
	\label{lemma:Differentiable}
	Let
	$ (V, \left\| \cdot \right\|_V) $ 
	be a nontrivial $ \R $-Banach space,  
	let
	$ T \in (0,\infty) $,  
	$ \angle_T = \{ (s, t) \in [0,T]^2 \colon s \leq t \} $, 
	$ f \in \mathcal{C}^{0,1}( [0,T] \times V, V) $, 
	let
	$ f_{0,1} \colon [0,T] \times V \to L( V ) $
	be the function 
	which satisfies 
	for all
	$ t \in [0,T] $, $ x \in V $ that
	$ f_{0,1}(t,x) = ( \tfrac{ \partial }{\partial x } f )(t,x) $, 
	and
	for every 
	$ x \in V $,
	$ s \in [0,T] $
	let
	$ X_{s,(\cdot)}^x = (X_{s,t}^x)_{ t \in [s,T]} \colon [s,T]
	\to V $
	be a continuous function
	which satisfies
	for all
	$ t \in [s,T] $
	that
		$ X_{s, t}^x = 
	x
	+
	\int_s^t f( \tau, X_{s, \tau}^x) \,d\tau $.
	Then 
	\begin{enumerate}[(i)] 
	\item \label{item:Diff_s}
	it holds for all
	$ x \in V $, $ t \in (0,T] $ 
	that
	$ ( [0,t] \ni s \mapsto X_{s,t}^x \in V ) \in \mathcal{C}^1( [0,t], V ) $,
	\item \label{item:JointContinuous}
	it holds that
	$ \{ (r,u) \in [0,T] \times (0,T] 
	\colon r \leq u \}
	\times V  
	\ni (s,t,x)
	\mapsto 
	\frac{ \partial }{ \partial s }
	X_{s,t}^x
	\in V $
	is continuous, 
	\item \label{item:JointContinuous2}
	it holds that there exists a unique continuous function
	$ C \colon \{ ( r, u ) \in [0,T]^2 \colon r \leq u \} \times V 
	\to V $
	which satisfies for all
	$ t \in (0, T] $, $ s \in [0,t] $, $ x \in V $
	that
	$ C_{s,t}^x = \frac{ \partial }{ \partial s }
	X_{s,t}^x $,
	\item \label{item:derSolves} it holds for all  
	$ x \in V $,
	$ t \in (0,T] $, 
	$ s \in [0,t] $
	that  
	\begin{equation}
	\begin{split} 
	\tfrac{ \partial }{ \partial s }
	X_{s,t}^x
	=
	- 
	f(s,x)
	+
	\int_s^t
	f_{0,1}
	( \tau, X_{s,\tau}^x ) 
	 (
	\tfrac{ \partial }{ \partial s }
	X_{s, \tau}^x 
	 )
	\, d \tau
	, 
	\end{split}
	\end{equation}
		\item \label{item:initialTimeDiff}
	it holds that
	$ ( \angle_T \times V \ni (s,t,x) \mapsto X_{s,t}^x \in V )
	\in \mathcal{C}^{0,0,1}( \angle_T \times V, V ) $,
	\item \label{item:identity11} 
	it holds for all 
	$ x,y \in V $, $ ( s, t ) \in \angle_T $ that
	\begin{equation}
	( 
	\tfrac{\partial}{\partial x }
	X_{s, t}^x 
	)  
	y
	=
	y
	+
	\int_s^t 
	f_{0,1}
	(\tau, X_{s, \tau}^x)
	(
	\tfrac{\partial}{\partial x } X_{s, \tau}^x 
	)y
	\, d\tau
	, 
	\end{equation} 
	\item \label{item:NiceConection}
	it holds for all
	$ x \in V $, 
	$ t \in (0,T] $,
	$ s \in [0,t] $
	that
	\begin{equation}
	\begin{split}
	\tfrac{ \partial }{ \partial s }
	X_{s,t}^x =
	-
	\big(
	\tfrac{\partial}{\partial x }
	X_{s, t}^x 
	\big)
	f(s,x)
	,
	\end{split}
	\end{equation}
	\item \label{item:Differentiable_t}
		it holds for all
		$ x \in V $,
		$ s \in [0,T) $
		that 
		$
		( [s,T] \ni t \mapsto X_{s,t}^x \in V) \in \mathcal{C}^1 ( [s,T], V )
		$,
	\item \label{item:Diff_t}
	it holds that
	$ \{ (r,u) \in [0,T) \times [0,T] 
	\colon r \leq u
	\} \times V
	\ni (s,t,x) 
	\mapsto 
	\tfrac{\partial }{\partial t }
	X_{s,t}^x
	\in V $
	is continuous, 
	\item \label{item:Diff_t2}
	it holds that there exists a unique continuous function
	$ D \colon \{ ( r, u ) \in [0,T]^2 \colon r \leq u \} \times V 
	\to V $
	which satisfies for all
	$ s \in [0,T ) $, $ t \in [s,T] $, $ x \in V $ 
	that
	$ D_{s,t}^x = \frac{ \partial }{ \partial t }
	X_{s,t}^x $,
	and
	\item \label{item:Formula_t}
	it holds for all
	$ x \in V $,
	$ s \in [0,T) $,
	$ t \in [s,T] $
	that
	$
	\tfrac{\partial }{\partial t }
	X_{s,t}^x = f(t, X_{s,t}^x) 
	$.
\end{enumerate}
\end{lemma}
\begin{proof}[Proof of Lemma~\ref{lemma:Differentiable}]
	Throughout this proof  
	let  
	$ g \colon [-T, T]^3 \times V \to \R $
	be a function which satisfies for all
	$ s \in [0,T] $, 
	$ h \in [-s,T-s] \backslash \{ 0 \} $, 
	$ \tau \in [ \max \{ s, s+h \}, T] $,
	$ x \in V $
	with
	$ X_{s+h, \tau}^x - X_{s,\tau}^x \neq 0 $
	that
	$ g(s, h, \tau, x)
	=
	\tfrac{ \| f(\tau,X_{s+h, \tau}^x ) - f( \tau, X_{s,\tau}^x )
		- 
		f_{0,1}
		( \tau, X_{s, \tau}^x )( X_{s+h, \tau}^x - X_{s, \tau}^x ) \|_V }
	{ \| X_{s+h, \tau}^x - X_{s, \tau}^x \|_V } 
	$,
	which satisfies for all
	$ s \in [0,T] $, 
	$ h \in [-s,T-s] $, 
	$ \tau \in [ \max \{ s, s+h\} , T] $,
	$ x \in V $
	with
	$ X_{s+h, \tau}^x - X_{s,\tau}^x = 0 $
	that
	$ g(s, h, \tau, x ) = 0 $,
	and which satisfies
	for all
	$ s, h \in [-T,T] $,  
	$ \tau \in [ - T, \max \{ s, s+h\} ) $,
	$ x \in V $
	that
	$ g(s, h, \tau, x ) = 0 $.
	Observe that
	Corollary~\ref{Corollary:UnifCont}
	(with
	$ V = V $,
	$ T = T $,
	$ f = f $,
	$ X_{s,t}^x = X_{s,t}^x $
	for $ (s,t) \in \angle_T $, $ x \in V $
	in the notation of 
	Corollary~\ref{Corollary:UnifCont})
	ensures that 
	\begin{equation} 
	\label{eq:Xcontinuous}
	( \angle_T \times V \ni (s,t,x) \mapsto X_{s,t}^x \in V )
	\in \mathcal{C} ( \angle_T \times V, V) 
	.
	\end{equation} 
	This
	and
	the assumption that
	$ f \in \mathcal{C}^{0,1}( [0,T] \times V, V ) $
	show that
	for every $ x \in V $ it holds that
	$ ( \angle_T \ni (s,t) \mapsto 
	f_{0,1}
	( s, X_{s, t}^x ) \in L(V) ) \in 
	\mathcal{C} ( \angle_T, L(V) ) $.
	Lemma~\ref{lemma:existence_uniqueness}
	(with
	$ V = V $,
	$ T = T $,
	$ \phi = ( \angle_T \ni (s,t) \mapsto - f(s,x) \in V ) $,
	$ A = ( \angle_T \ni (s,t) \mapsto f_{0,1}
	( s, X_{s, t}^x ) \in L(V) ) $
	for $ x \in V $
	in the notation of
	Lemma~\ref{lemma:existence_uniqueness}) therefore proves 
	that for every $ x \in V $ 
	there exists 
	a unique function
	$ Y^x \in \mathcal{C}(\angle_T, V) $ 
	such that for all
	$ (s,t) \in \angle_T $
	it holds that
	\begin{equation}
	\begin{split}
	\label{eq:Der_satisfies_equation}
	Y_{s,t}^x
	=
	- 
	f(s,x)
	+
	\int_s^t
	f_{0,1}
	( \tau, X_{s,\tau}^x ) 
	Y_{s, \tau}^x \, d \tau
	.
	\end{split}
	\end{equation}
	This ensures that there exists a function
	$ q \colon [-T,T]^3 \times V \to \R $ 
	which satisfies
	for all
	$ s \in [0,T] $,
	$ h \in [-s, T-s] \backslash \{ 0 \} $,
	$ \tau \in [ \max \{ s, s+h \},T] $,
	$ x \in V $
	that
	$ q(s, h, \tau, x ) 
	=
	\| \nicefrac{( X_{s+h, \tau}^x - X_{s, \tau}^x) }{h} - Y_{s, \tau}^x \|_V $
	and which satisfies for all
	$ s \in [0,T] $, 
	$ h \in [-T, T] $, 
	$ \tau \in [ -T, \max \{ s, s+h \} ) $,
	$ x \in V $
	that
	$ q( s, h, \tau, x ) = 0 $.
	In the next step we note that
	the triangle inequality implies
	that for all
	$ x \in V $,
	$ s \in [0,T] $,
	$ h \in [-s, T-s] \backslash \{ 0 \} $,
	$ t \in [ \max \{s+h,s\}, T] $
	it holds that
	\begin{equation}
	\begin{split}
	\label{eq:ApplyGronwalFirst}
	&
	\Big\| 
	\tfrac{ X_{s+h,t}^x - X_{s,t}^x }{ h }  
	-
	Y_{s,t}^x
	\Big\|_V
	=
	\Big\|-
	\tfrac{1}{|h|} 
	\int_{ \min \{ s, s + h \} }^{ \max \{ s, s+h \} }  f(\tau, X_{\min \{ s, s + h \}, \tau}^x ) \, d\tau 
	+ 
	f(s,x)
	\\
	&
	\quad
	+
	\int_{ \max \{ s, s+h \} }^t 
	\tfrac{ 
		f(\tau, X_{ s+h, \tau}^x ) 
		-
		f(\tau, X_{ s, \tau}^x )
	}{h}
	\, d\tau 
	-
	\int_s^t f_{0,1}
	( \tau, X_{s, \tau}^x ) Y_{s, \tau}^x \, d \tau 
	\Big\|_V
	\\
	&
	\leq
	\Big\| 
	-
	\int_{ \min \{s, s+h \} }^{
		\max \{s, s+h \} } 
	\tfrac{1}{|h|} 
	f(\tau, X_{ \min \{s, s + h \}, \tau}^x)
	\, d \tau 
	+
	f(s,x)
	\Big\|_V
	+
	\Big\|
	\int_s^{ \max \{ s, s + h \} }
	f_{0,1}
	( \tau, X_{s, \tau}^x ) 
	Y_{s, \tau}^x
	\, d \tau 
	\Big\|_V
	\\
	&
	\quad
	+
	\int_{ \max \{ s, s + h \} }^t
	\Big\|
	\tfrac{ f ( \tau, X_{ s+h, \tau}^x ) - f( \tau, X_{s, \tau}^x ) }
	{ h } 
	-
	f_{0,1}
	( \tau, X_{s, \tau}^x ) 
	Y_{s, \tau}^x
	\Big\|_V 
	\, d \tau 
	.
	\end{split}
	\end{equation}
	Next we intend to prove an upper bound
	for the l.h.s.\ of~\eqref{eq:ApplyGronwalFirst}.
	To do so, we will estimate
	the terms on the
	r.h.s.\ of~\eqref{eq:ApplyGronwalFirst}
	separately.
	For this note that the fact that 
	$ \forall \, x \in V \colon ( \angle_T \ni (s,t) \mapsto X_{s,t}^x \in V )
	\in \mathcal{C} ( \angle_T , V) $
	shows that for all
	$ x \in V $,
	$ s \in [0,T] $
	it holds that 
	\begin{equation}
	\begin{split}
	\label{eq:sup1}
	&
	\sup_{h \in [-s,T-s] 
		\backslash \{ 0 \}}
	\Big\| 
	\int_{ \min \{s, s+h\}}^{ \max \{s, s+h\} } 
	\tfrac{1}{|h|} 
	f(\tau, X_{ \min \{s, s+h\}, \tau}^x)
	\, d \tau
	-
	f(s,x)
	\Big\|_V
	< 
	\infty
	.  
	\end{split}
	\end{equation} 
	In addition,
	the assumption that
	$ \forall \, x \in V, s \in [0,T] \colon
	( [s,T] \ni t \mapsto X_{s,t}^x \in V )
	\in
	\mathcal{C}( [s,T], V ) $ 
	and
	the fact that
	$ \forall \, x \in V \colon 
	(\angle_T \ni (s,t) \mapsto Y_{s,t}^x \in V)
	\in \mathcal{C}( \angle_T, V ) $  
	ensure for all
	$ x \in V $,
	$ s \in [0,T] $
	that
	\begin{equation}
	\begin{split}
	\label{eq:sup2}
	\sup_{ h \in [-s, T-s] 
		\backslash \{ 0 \} }
	\Big\| 
	\int_s^{ \max \{ s, s + h \} }
	f_{0,1}
	( \tau, X_{s, \tau}^x ) 
	Y_{s, \tau}^x
	\, d \tau 
	\Big\|_V  
	< \infty.
	\end{split}
	\end{equation}
	Moreover, the triangle inequality proves that
	for all 
	$ x \in V $,
	$ s \in [0,T] $,
	$ h \in [-s, T-s] \backslash \{0\} $,
	$ t \in [ \max \{ s, s + h \} , T] $ 
	it holds that
	\begin{equation} 
	\begin{split} 
	\label{eq:triangle}
	&
	\int_{ \max \{ s, s + h \} }^t
	\Big\|
	\tfrac{ f ( \tau, X_{s+h, \tau}^x ) - f( \tau, X_{s, \tau}^x ) }
	{ h } 
	-
	f_{0,1}
	( \tau, X_{s, \tau}^x ) 
	Y_{s, \tau}^x
	\Big\|_V 
	\, d \tau
	\\
	&
	\leq
	\int_{ \max \{ s, s + h \} }^t
	\Big\|
	\tfrac{ f(\tau, X_{s+h, \tau}^x ) - f( \tau, X_{s,\tau}^x )
		- 
		f_{0,1}
		( \tau, X_{s, \tau}^x )( X_{s+h, \tau}^x - X_{s, \tau}^x ) }
	{ h }
	\Big\|_V
	\, d \tau
	\\
	&
	\quad 
	+
	\int_{ \max \{ s, s + h \} }^t 
	\Big\|
	\tfrac{ 
		f_{0,1}
		(\tau, X_{s, \tau}^x)( X_{s+h, \tau}^x - X_{s, \tau}^x 
		-
		Y_{s, \tau}^x h) 
	}
	{ h }
	\Big\|_V 
	\, d \tau
	.
	\end{split}
	\end{equation}
	Furthermore,
	note that~\eqref{eq:Xcontinuous}
	assures that there exists 
	$ \delta \colon (0,\infty) \times V \to (0,\infty) $
	such that for all
	$ (s_1, t_1), (s_2, t_2) \in \angle_T $, 
	$ \varepsilon \in (0, \infty) $,
	$ x \in V $
	with
	$ \max \{ | s_1 - s_2 |, | t_1 - t_2 | \} < \delta_\varepsilon^x $
	it holds that
	\begin{equation} 
	\label{eq:UniformContinuous}
	\| X_{s_1, t_1}^x - X_{s_2, t_2}^x \|_V < \varepsilon 
	.
	\end{equation} 
	This
	implies that for all
	$ \varepsilon \in (0, \infty ) $,
	$ x \in V $,
	$ s \in [0,T] $, 
	$ h \in [-s, T-s] \backslash \{0\} $,
	$ t \in [ \max \{ s, s + h \}, T] $ 
	with
	$ | h | < \delta_{ \varepsilon }^x $
	it holds
	that
	\begin{equation} 
	\begin{split} 
	\label{eq:triangle1}
	& 
	\int_{ \max \{ s, s + h \} }^t
	\Big\|
	\tfrac{ f(\tau, X_{s+h, \tau}^x ) - f( \tau, X_{s,\tau}^x )
		- 
	f_{0,1}
		( \tau, X_{s, \tau}^x )( X_{s+h, \tau}^x - X_{s, \tau}^x ) }
	{ h }
	\Big\|_V
	\, d \tau
	\\
	&
	\leq
	\int_{ \max \{ s, s + h \} }^t 
	\Big(
	\tfrac{ \| f(\tau, X_{s+h, \tau}^x ) - f( \tau, X_{s,\tau}^x ) \|_V}
	{ | h | }
	+
	\tfrac{  
		\| f_{0,1}
		( \tau, X_{s, \tau}^x ) \|_{L(V)} 
		\| X_{s+h, \tau}^x - X_{s, \tau}^x \|_V }
	{ | h | }
	\Big) 
	\, d \tau
	\\
	&
	\leq  
\bigg(
\sup_{ \tau \in [0,T] } 
\sup_{  y \in V, \| x - y \|_V \leq \varepsilon }
\sup_{ 
		z \in V \backslash \{ y \}, 
		\| x-z \|_V \leq \varepsilon } 
	\tfrac{ \| f(\tau, y ) - f( \tau, z ) \|_V}
	{ \| y - z \|_V }
	+
\sup_{ \tau \in [s,T] }
\| f_{0,1}
( \tau, X_{s, \tau}^x ) \|_{L(V)}
\bigg) 
\\
&
\quad 
\cdot 
	\int_{ \max \{ s, s + h \} }^t  
	\tfrac{ \| X_{s+h, \tau}^x - X_{s, \tau}^x \|_V }
	{ | h | } 
	\, d \tau 
	.
	\end{split}
	\end{equation}
	Therefore, it holds for all
	$ \varepsilon \in (0, \infty) $,
	$ x \in V $,
	$ s \in [0,T] $,
	$ h \in [-s, T-s] \backslash \{ 0 \} $,
	$ t \in [\max \{ s, s+h\}, T] $
	with
	$ | h | < \delta_\varepsilon^x $  
	that
	\begin{equation} 
	\begin{split} 
	\label{eq:triangle_more}
	& 
	\int_{ \max \{ s, s + h \} }^t
	\Big\|
	\tfrac{ f(\tau, X_{s+h, \tau}^x ) - f( \tau, X_{s,\tau}^x )
		- 
		f_{0,1}
		( \tau, X_{s, \tau}^x )( X_{s+h, \tau}^x - X_{s, \tau}^x ) }
	{ h }
	\Big\|_V
	\, d \tau
	\\
	& 
	\leq 
\bigg(
\sup_{ \tau \in [0,T] } \sup_{   y \in V, \| x - y \|_V \leq \varepsilon}
\sup_{ 
		z \in V \backslash \{ y \}, 
		\| x - z \|_V \leq \varepsilon } 
\tfrac{ \| f(\tau, y ) - f( \tau, z ) \|_V}
{ \| y - z \|_V }
+
\sup_{ \tau \in [s,T] }
\| f_{0,1}
( \tau, X_{s, \tau}^x ) \|_{L(V)}
\bigg) 
\\
&
\quad 
\cdot 
	\Big( 
	\int_{ \max \{ s, s + h \} }^t  
	\tfrac{ \| X_{s+h, \tau}^x - X_{s, \tau}^x
		- Y_{s, \tau}^x h \|_V }
	{ | h | } 
	\, d \tau
	+
	\int_{ \max \{ s, s + h \} }^t
	\| Y_{s, \tau}^x \|_V \, d \tau
	\Big)
	.
	\end{split}
	\end{equation}
	Furthermore, observe that 
	for all
	$ x \in V $,
	$ s \in [0,T] $,
	$ h \in [-s, T-s] \backslash \{0\} $,
	$ t \in [\max \{ s, s + h \}, T] $ 
	it holds that
	\begin{equation} 
	\begin{split} 
	\label{eq:triangle2}
	&
	\int_{ \max \{ s, s + h \} }^t 
	\Big\|
	\tfrac{ 
		f_{0,1}
		(\tau, X_{s, \tau}^x)( X_{s+h, \tau}^x - X_{s, \tau}^x 
		-
		Y_{s, \tau}^x h) 
	}
	{ h }
	\Big\|_V 
	\, d \tau
	\\
	&
	\leq
	\int_{ \max \{ s, s + h \} }^t
	 \| 	
	f_{0,1}
	(\tau, X_{s, \tau}^x) 
	 \|_{L(V)}   
	\tfrac{ 
		\| X_{s+h, \tau}^x - X_{s, \tau}^x 
		-
		Y_{s, \tau}^x h \|_V
	}
	{ | h | } 
	\, d \tau
	\\
	&
	\leq
	\sup\nolimits_{ \tau \in [s,T] }
	\| f_{0,1}( \tau, X_{s, \tau }^x ) \|_{ L(V) }
	\int_{ \max \{ s, s + h \} }^t  
	\tfrac{ 
		\| X_{s+h, \tau}^x - X_{s, \tau}^x 
		-
		Y_{s, \tau}^x h \|_V
	}
	{ | h | } 
	\, d \tau
	.
	\end{split}
	\end{equation}
	Combining~\eqref{eq:ApplyGronwalFirst},
	\eqref{eq:triangle},
	and
	\eqref{eq:triangle_more}
	hence
	proves that
	for all
	$ \varepsilon \in (0, \infty ) $,
	$ x \in V $,
	$ s \in [0,T] $,
	$ h \in [-s, T-s] \backslash \{0\} $,
	$ t \in [ \max \{ s, s + h \}, T] $ 
	with
	$ | h | < \delta_{ \varepsilon }^x $ 
	it holds that
	\begin{equation}
	\begin{split} 
	&
	\Big\| 
	\tfrac{ X_{s+h,t}^x - X_{s,t}^x }{ h }  
	-
	Y_{s,t}^x
	\Big\|_V
	\\
	&
	\leq 
	\Big\| 
	\int_{ \min \{ s, s+h\} }^{ \max \{s, s+h\} } 
	\tfrac{1}{ | h | } 
	f(\tau, X_{ \min \{s, s+h\}, \tau}^x)
	\, d \tau
	-
	f(s,x)
	\Big\|_V
	+ 
	\Big\| 
	\int_s^{ \max \{ s, s + h \} }
	f_{0,1}
	( \tau, X_{s, \tau}^x ) 
	Y_{s, \tau}^x
	\, d \tau  
	\Big\|_V
	\\
	&
	\quad 
	+
\bigg(
\sup_{ \tau \in [0,T] } 
\sup_{   y \in V, \| x - y \|_V \leq \varepsilon }
\sup_{ 
		z \in V \backslash \{ y \},
		\| x-z \|_V \leq \varepsilon } 
\tfrac{ \| f(\tau, y ) - f( \tau, z ) \|_V}
{ \| y - z \|_V }
+
2
\sup_{ \tau \in [s,T] }
\| f_{0,1}
( \tau, X_{s, \tau}^x ) \|_{L(V)}
\bigg)  
\\
&
\quad
\cdot 
	\Big(
	\int_{ \max \{s, s+h\}  }^t
	\Big\|
	\tfrac{ X_{s+h, \tau}^x - X_{s, \tau}^x }{ h } - Y_{s, \tau}^x
	\Big\|_V 
	\, d \tau 
	+
	\int_{ \max \{s, s+h\} }^t \| Y_{s,\tau}^x \|_V \, d \tau 
	\Big) 
	%
	. 
	\end{split}
	\end{equation}
The fact that
$ \forall \, x \in V \colon 
(\angle_T \ni (s,t) \mapsto Y_{s,t}^x \in V)
\in \mathcal{C}( \angle_T, V ) $,
 \eqref{eq:sup1},
 \eqref{eq:sup2},
 and
 Lemma~\ref{lemma:Nicer_assumption}
 (with
 $ V = V $,
 $ T = T $,
 $ x_0 = x_0 $,
 $ f = f $
 in the notation of
 Lemma~\ref{lemma:Nicer_assumption})
	therefore imply that
	for every
	$ x \in V $, 
	$ s \in [0, T] $ 
	there exist
	$ \varepsilon, C \in (0,\infty) $
	such that for all
	$ h \in [-s, T-s] \backslash \{0\} $,
	$ t \in [ \max \{ s, s+h\}, T ] $
	with
	$ | h | < \delta_\varepsilon^x $
	it holds that
	\begin{equation}
	\begin{split} 
	&
	\Big\| 
	\tfrac{ X_{s+h,t}^x - X_{s,t}^x }{ h }  
	-
	Y_{s,t}^x
	\Big\|_V
	\leq
    C
	+
	C 
	\int_{\max \{s,s+h\}}^t
	\Big\|
	\tfrac{ X_{s+h, \tau}^x - X_{s, \tau}^x }{ h } - Y_{s, \tau}^x
	\Big\|_V 
	\, d \tau
	< \infty
	. 
	\end{split}
	\end{equation}
	Gronwall's lemma 
	hence ensures that
	for all
	$ x \in V $, 
	$ s \in [0,T ] $
	there exist
	$ \varepsilon, C \in (0,\infty) $
	such that for all
	$ h \in [-s, T - s] \backslash \{0\} $,
	$ t \in [ \max \{ s, s+h\}, T ] $ 
	with
	$ | h | < \delta_{ \varepsilon }^x $
	it holds that
	\begin{equation}
	\begin{split}
	\label{eq:bounded}
	&
	\Big\| 
	\tfrac{ X_{s+h,t}^x - X_{s,t}^x }{ h }  
	-
	Y_{s,t}^x
	\Big\|_V
	\leq
	C
	\exp ( C ( t - \max \{s, s + h\} ) )
	\leq
	C e^{C t}
	.
	\end{split}
	\end{equation}
	In the next step
	we intend to establish that
	for all
	$ x \in V $,
	$ t \in (0,T] $,
	$ s \in [0,t] $
	it holds that
	\begin{equation}
	\label{eq:LimExists}
	\limsup_{ ( [-s, t-s] \backslash \{0\} ) \ni h \to 0 } 
	\Big\|
	\tfrac{ X_{s+h, t}^x - X_{s, t}^x }{ h } - Y_{s, t}^x
	\Big\|_V 
	=
	0.
	\end{equation}
	For this we will analyze
	the terms on the
	r.h.s.\ of~\eqref{eq:ApplyGronwalFirst}
	separately.
	First, note that
	Lemma~\ref{lemma:Lebesgues_diff} 
	(with
	$ V = V $,
	$ a = s $,
	$ b = T $,
	$ f = ( [s,T] \ni \tau \mapsto f ( \tau, X_{s, \tau}^x ) \in V ) $
	for
	$ s \in [0,T) $,
	$ x \in V $ 
	in the notation of
	Lemma~\ref{lemma:Lebesgues_diff})
	shows that for all
	$ x \in V $,
	$ s \in [0,T) $ 
	it holds that
	\begin{equation}
	\begin{split}
	\label{eq:limsup1}
	&
	\limsup_{ (0, T-s ] \ni h \to 0}
	\Big\| 
	\int_s^{s+h} 
	\tfrac{1}{h} 
	f(\tau, X_{s, \tau}^x)
	\, d \tau
	-
	f(s,x)
	\Big\|_V
	=
	0
	.  
	\end{split}
	\end{equation}
	In addition, 
	\eqref{eq:Xcontinuous} 
	implies  
	for all 
	$ x \in V $,
	$ s \in (0,T] $   
	that
	\begin{equation}
	\begin{split}
	& 
	\limsup_{ [-s,0) \ni h \to 0 }
	\Big\| 
	\int_{s+h}^s 
	\tfrac{1}{h}  
	f(\tau, X_{s+h, \tau}^x) 
	\, d \tau 
	+ f(s,x)
	\Big\|_V
	\\
	&
	=
	\limsup_{ [-s,0) \ni h \to 0 }
	\Big\| 
	\int_{s+h}^s 
	\tfrac{1}{h}
	(  
	f(\tau, X_{s+h, \tau}^x)  
	- f(s,x)
	)
	\, d \tau
	\Big\|_V
	\\
	&
	\leq
	\limsup_{ [-s,0) \ni h \to 0 }
	\Big(
	\tfrac{ 1 }{ | h | }
	\int_{s+h}^s  
	\, d\tau
	\cdot 
	\sup_{ \tau \in [s+h, s] }
	\| f( \tau, X_{s+h, \tau}^x ) - f(s,x) \|_V
	\Big)
	\\
	&
	=
	\limsup_{ [-s,0) \ni h \to 0 } 
	\sup_{ \tau \in [s+h, s] }
	\| f( \tau, X_{s+h, \tau}^x ) - f(s,x) \|_V
	=
	0.
	\end{split} 
	\end{equation} 
	This and~\eqref{eq:limsup1} prove for all
	$ x \in V $,
	$ s \in [0,T] $
	that
	\begin{equation}
	\begin{split}
	\label{eq:limsup1a}
	&
	\limsup_{ ( [-s, T-s ]\backslash \{0\} )\ni h \to 0}
	\Big\| 
	\int_{ \min \{s, s+h\}}^{ \max \{s, s+h\} } 
	\tfrac{1}{ |h|} 
	f(\tau, X_{ \min \{s, s+h\}, \tau}^x)
	\, d \tau
	-
	f(s,x)
	\Big\|_V
	=
	0
	.  
	\end{split}
	\end{equation}
	Furthermore, observe that the fact that
	$ \forall \, x \in V \colon 
	(\angle_T \ni (s,t) \mapsto Y_{s,t}^x \in V)
	\in \mathcal{C}( \angle_T, V ) $  
	and~\eqref{eq:Der_satisfies_equation}
	prove for all
	$ x \in V $,
	$ s \in [0,T] $
	that
	\begin{equation}
	\begin{split}
	\label{eq:limsup2}
	&
	\limsup_{ ( [-s, T-s ]\backslash \{0\} )\ni h \to 0}
	\Big\|  
	\int_s^{ \max \{ s, s + h \} }
	f_{0,1}
	( \tau, X_{s, \tau}^x ) 
	Y_{s, \tau}^x
	\, d \tau 
	\Big\|_V 
	\\
	&
	=
	\limsup_{ ( [-s, T-s ]\backslash \{0\} )\ni h \to 0}
	\| Y_{s, \max \{ s, s + h \} }^x 
	+ f(s,x) \|_V
	=
	0.
	\end{split}
	\end{equation}
	Moreover, note that
	the triangle inequality,
	\eqref{eq:UniformContinuous},
	and
	Lemma~\ref{lemma:Nicer_assumption}
	(with
	$ V = V $,
	$ T = T $,
	$ x_0 = x $,
	$ f = f $
	for
	$ x \in V $
	in the notation of
	Lemma~\ref{lemma:Nicer_assumption})
	demonstrate that
	for all
	$ s \in [0,T] $, $ x \in V $
	there exists 
	$ \varepsilon \in (0, \infty) $
	such that
	for all
	$ h \in [-s,T-s] $, $ t \in [s, T] $
	with
	$ | h | < \delta_{ \varepsilon }^x $
	it holds that
	\begin{equation}
	\begin{split} 
	&
	| g(s,h,t,x) |
	\leq
	\bigg(
	\sup_{ \tau \in [0,T] } 
	\sup_{ \substack{  y \in V, \| x - y \|_V \leq \varepsilon,
			\\
			z \in  V \backslash \{ y \},
			\| x-z \|_V \leq \varepsilon } }
	\tfrac{ \| f(\tau, y ) - f( \tau, z ) \|_V}
	{ \| y - z \|_V }
	+
	\sup_{ \tau \in [s,T] }
	\| f_{0,1}
	( \tau, X_{s, \tau}^x ) \|_{L(V)}
	\bigg) 
	<
	\infty 
	.
	\end{split} 
	\end{equation}
	Fatou's lemma 
	and~\eqref{eq:bounded} 
therefore
	ensure that
	for all
	$ x \in V $, 
	$ s \in [0,T] $
	there exist 
	$ \varepsilon, C \in (0,\infty) $
	such that for all
	$ t \in [s,T] $
	with
	$ | t | + |s | > 0 $
	 it holds that
	\begin{equation} 
	\begin{split} 
	\label{eq:Limit0}
	& 
	\limsup_{ ( [- \min \{ s, \delta_\varepsilon^x \}, 
		\min \{ t - s, \delta_\varepsilon^x \} ] \backslash \{0\} )\ni h \to 0}
	\int_{ \max \{s,s+h\}}^t
	\Big\|
	\tfrac{ f(\tau, X_{s+h, \tau}^x ) - f( \tau, X_{s,\tau}^x )
		- 
	f_{0,1}
		( \tau, X_{s, \tau}^x )( X_{s+h, \tau}^x - X_{s, \tau}^x ) }
	{ h }
	\Big\|_V
	\, d \tau
	\\
	&
	\leq  
\limsup_{ ( [- \min \{ s, \delta_\varepsilon^x \}, 
	\min \{ t - s, \delta_\varepsilon^x \} ] \backslash \{0\} )\ni h \to 0}
	\int_s^t  
	g( s, h, \tau, x)
	\, d \tau
	\\
	&
	\quad 
	\cdot 
	\bigg(
	\sup_{ v \in [- \min \{ s, \delta_\varepsilon^x \}, 
		\min \{ t - s, \delta_\varepsilon^x \} ]
		\backslash \{0\}, \tau \in [\max \{ s, s + v \}, T]}
	\tfrac{ \| X_{s + v, \tau}^x - X_{s, \tau}^x \|_V }
	{ | v | }
	\bigg)
	\\
	&
	\leq  
	\int_s^t  
	\limsup_{ ( [-s, t-s ]\backslash \{0\} )\ni h \to 0}
	g( s, h, \tau, x)
	\, d \tau
	\cdot 
	\big(
	C e^{C T} + 
	\sup\nolimits_{ \tau \in [s, T]}
	\| Y_{s, \tau}^x \|_V
	\big)
	=
	0
	.
	\end{split}
	\end{equation}
	Combining~\eqref{eq:ApplyGronwalFirst},
	\eqref{eq:triangle}, 
	\eqref{eq:triangle2},
	\eqref{eq:bounded},
	\eqref{eq:limsup1a},
	\eqref{eq:limsup2},  
	and Fatou's lemma
	hence 
	proves that 
	for all
	$ x \in V $,
	$ s \in [0,T] $, 
	$ t \in [s,T] $ 
	with
	$ | t | + |s | > 0 $
	it holds that
	\begin{equation}
	\begin{split} 
	&
	\limsup_{  ( [-s, t-s] \backslash \{0\} ) \ni h \to 0 }
	q( s, h, t, x)  
	\\
	&            
	\leq  
\sup_{ \tau \in [s,T] }
\| f_{0,1}
( \tau, X_{s, \tau}^x ) \|_{L(V)}
	\limsup_{  ( [-s, t-s] \backslash \{0\} ) \ni h \to 0 }
	\int_{ \max \{s,s+h\}}^t  
	\tfrac{ 
		\| X_{s+h, \tau}^x - X_{s, \tau}^x 
		-
		Y_{s, \tau}^x h \|_V
	}
	{ | h | } 
	\, d \tau 
	\\
	&
	=
\sup_{ \tau \in [s,T] }
\| f_{0,1}
( \tau, X_{s, \tau}^x ) \|_{L(V)}
	\limsup_{  ( [-s, t-s] \backslash \{0\} ) \ni h \to 0 }
	\int_s^t  
	q(s,h,\tau,x)
	\, d \tau 
	\\
	&
	\leq 
	\sup_{ \tau \in [s,T] }
	\| f_{0,1}
	( \tau, X_{s, \tau}^x ) \|_{L(V)}
	\int_s^t
	\limsup_{  ( [-s, t-s] \backslash \{0\} ) \ni h \to 0 }
	q( s, h, \tau, x )  
	\, d \tau
	< \infty
	. 
	\end{split}
	\end{equation}
	Gronwall's lemma hence establishes~\eqref{eq:LimExists}.
	In particular, we obtain that
	for all
	$ x \in V $,
	$ t \in (0,T] $,
	$ s \in [0,t] $
	it holds that
	$ [0,t] \ni u \mapsto X_{u, t}^x \in V $
	is differentiable
	and that
	\begin{equation} 
	\label{eq:CrucialIdentity}
	\tfrac{\partial}{\partial s }
	X_{s,t}^x = Y_{s,t}^x 
	.
	\end{equation}
	This
	and~\eqref{eq:Der_satisfies_equation} establish items~\eqref{item:Diff_s}
	and~\eqref{item:derSolves}.
	Moreover,
	note that~\eqref{eq:Xcontinuous} and
	Lemma~\ref{lemma:diff_initial_value}
	(with
	$ V = V $,
	$ T = T $,
	$ f = f $,
	$ X_{s,t}^x = X_{s,t}^x $,
	$ y = - f(s,x) $
	for
	$ x \in V $,
	$ (s,t) \in \angle_T $
	in the notation of Lemma~\ref{lemma:diff_initial_value})
	prove that items~\eqref{item:initialTimeDiff}
	and~\eqref{item:identity11} hold 
	and 
	that for all
	$ x \in V $, $ ( s, t ) \in \angle_T $ 
	it holds that
	\begin{equation}
	\label{eq:Uniqueness}
	-
	\big(
	\tfrac{\partial}{\partial x }
	X_{s, t}^x 
	\big)
	f(s,x)
	=
	-
	f(s,x)
	+
	\int_s^t 
	f_{0,1}
	(\tau, X_{s, \tau}^x)
	\big(
	-
	\big(
	\tfrac{\partial}{\partial x } X_{s, \tau}^x 
	\big)
	f(s,x)
	\big)
	\, d\tau
	.
	\end{equation}
	Furthermore, observe that
	the fact that
	$ ( \angle_T \times V \ni (s,t,x) \mapsto
	\frac{ \partial }{ \partial x } X_{s,t}^x \in L(V) )
	\in 
	\mathcal{C}( \angle_T \times V, L(V) ) $,
	the fact that
	$ ( \angle_T \times V \ni (s,t,x) \mapsto f(t,x) \in V) \in \mathcal{C}( \angle_T \times V, V ) $,
	and the fact that
	$ ( L(V) \times V \ni (A, x) \mapsto Ax \in V )
	\in \mathcal{C}( L(V) \times V, V ) $
	ensure that
	\begin{equation} 
		\label{eq:JointCOnt}
		\big( 
		\angle_T \times V \ni (s,t,x) \mapsto -
		\big(
		\tfrac{\partial}{\partial x }
		X_{s, t}^x 
		\big)
		f(s,x) 
		\big)
		\in \mathcal{C}( \angle_T \times V, V )
		.
		\end{equation}  
	Combining this with~\eqref{eq:Der_satisfies_equation}
	and~\eqref{eq:Uniqueness}
	demonstrates that for all
	$ x \in V $,
	$ (s,t) \in \angle_T $ it holds that
	\begin{equation}
	\begin{split}
	Y_{s,t}^x =
	-
	\big(
	\tfrac{\partial}{\partial x }
	X_{s, t}^x 
	\big)
	f(s,x)
	.
	\end{split}
	\end{equation}
	Equations~\eqref{eq:CrucialIdentity}
	and~\eqref{eq:JointCOnt}
	therefore
	establish items~\eqref{item:JointContinuous},
	\eqref{item:JointContinuous2},    
	and~\eqref{item:NiceConection}.
	Next observe that
	Lemma~\ref{lemma:extensions}
	(with
	$ V = V $,
	$ a = s $,
	$ b = T $,
	$ f = ( [s,T] \ni t \mapsto f( t, X_{s,t}^x ) \in V ) $,
	$ F = ( [s,T] \ni t \mapsto X_{s,t}^x \in V ) $
	for 
	$ s \in [0,T) $,
	$ x \in V $
	in the notation of Lemma~\ref{lemma:extensions}) 
	proves that for all
	$ x \in V $,
	$ s \in [0,T) $,
	$ t \in [s,T] $
	it holds that
	\begin{equation}
	\begin{split}
	\label{eq:used_twice}
	\limsup_{ ( [ s-t,T - t ] \backslash \{ 0 \} ) \ni h \to 0 }
	\tfrac{
		\| X_{s, t + h }^x - X_{s,t}^x
		-
		f(t, X_{s,t}^x ) h
		\|_V
	}{ |h| }
	=
	0
	.
	\end{split}
	\end{equation}
	This ensures that for all
	$ x \in V $,
	$ s \in [0,T) $ it holds that
	$ [s,T] \ni t \mapsto X_{s,t}^x \in V $
	is differentiable
	and that
	for all 
	$ x \in V $,
	$ s \in [0,T) $,
	$ t \in [s,T] $
	it holds that
	\begin{equation} 
	\label{eq:smooth1}
	\tfrac{\partial }{\partial t }
	X_{s,t}^x = f(t, X_{s,t}^x) 
	.
	\end{equation} 
	Combining this with~\eqref{eq:Xcontinuous}
	establishes
	items~\eqref{item:Differentiable_t},
	\eqref{item:Diff_t},
	\eqref{item:Diff_t2},
	and~\eqref{item:Formula_t}. 
	The proof of Lemma~\ref{lemma:Differentiable}
	is thus completed.
\end{proof}
\section{Alekseev-Gr\"obner formula}
\label{subsection:AG}
In this section we
combine
Proposition~\ref{proposition:Banach_space_alekseev_grobner},
Lemma~\ref{lemma:flow_property},
and 
Lemma~\ref{lemma:Differentiable} 
to establish
in Corollary~\ref{corollary:Alekseev_grobner} an extension of the Alekseev-Gr\"obner formula (cf., e.g.,
 Hairer et al.\ \cite[Theorem~14.5 in Chapter~I]{HairerNonstiffProblems}) for Banach space valued functions.
\begin{lemma}
	\label{lemma:ExtendedF}  
	Let 
	$ ( V, \left\| \cdot \right \|_V ) $
	be a nontrivial $ \R $-Banach space,
	let
	$ ( W, \left\| \cdot \right \|_W ) $
	be an $ \R $-Banach space,
	and
	let $ a \in \R $,
	$ b \in (a,\infty) $,
	$ \phi \in  \mathcal{C}^{0, 1}([a, b] \times V, W) $
	satisfy
	for all $ x \in V $ that
	$ ( [a,b] \ni t \mapsto \phi(t,x) \in W )
	\in \mathcal{C}^1( [a,b], W ) $
	and
	$ ( [a,b] \times V \ni (t,y) \mapsto
	(\frac{ \partial }{ \partial t} \phi)(t,y)
	\in W )
	\in \mathcal{C}( [a,b] \times V, W ) $.
	Then there exists
	$ \Phi \in  \mathcal{C}^1(\R \times V, W) $
	such that for all
	$ t \in [a,b] $,   
	$ x \in V $ 
	it holds that
	$ \phi(t,x) = \Phi(t,x) $.
\end{lemma}
\begin{proof}[Proof of Lemma~\ref{lemma:ExtendedF}]
	Throughout this proof let  
	$ \Phi \colon \R \times V \to W $
	be the function which satisfies  
	for all 
	$ t \in \mathbb{R} $, 
	$ x \in V $,
	$ k \in \N $
	that
	\begin{equation}
	\begin{split}
	\label{eq:DefineFirst}
	&
	\Phi(t,x) 
	=
	\begin{cases}
	\phi(t,x) 
	&\colon 
	(t,x) \in [a,b] \times V 
	\\
	2 \phi( a, x ) 
	- 
	\phi( 2 a - t, x) 
	&\colon (t,x) 
	\in [ a - ( b - a ), a ) \times V 
	\\
	2 \phi( b, x) 
	- 
	\phi( 2 b - t, x) 
	&\colon (t,x) 
	\in 
	( b, b + ( b - a ) ] \times V
	\end{cases} 
	\end{split}
	\end{equation}
	and
	{\small
		\begin{equation} 
		\begin{split}
		\label{eq:DefineAll}
		%
		\Phi(t,x)
		&=
		\begin{cases}
		\phi(t,x) 
		&\colon 
		(t,x) \in [a,b] \times V 
		\\
		2 \Phi( b - k (b - a), x ) 
		- 
		\Phi( 2(b-k(b-a)) - t, x) 
		&\colon (t,x) 
		\in [ a-k(b-a), b-k(b-a) ) \times V 
		\\
		2 \Phi( a+k(b-a), x) 
		- 
		\Phi( 2(a+k(b-a)) - t, x) 
		&\colon (t,x) 
		\in 
		( a + k(b-a), b+k(b-a) ] \times V
		\end{cases} 
		,
		\end{split}
		\end{equation}
	}let
	$ \Phi_{1,0} \colon \R \times V \to W $ be the function
	which satisfies  
	for all
	$ t \in \mathbb{R} $, 
	$ x \in V $,
	$ k \in \N $
	that
	\begin{equation}
	\begin{split}
	\label{eq:def_per_1}
	\Phi_{1,0}(t,x)
	=
	\begin{cases} 
	(  \tfrac{\partial}{\partial t } \phi)(t,x) 
	&\colon 
	(t,x) \in [a,b] \times V 
	\\ 
	\Phi_{1,0}( 2 ( b - k (b-a) ) - t, x) 
	&\colon (t,x) 
	\in [ a - k( b - a ), b - k( b - a ) ) \times V 
	\\ 
	\Phi_{1,0}( 2( a + k (b-a) ) - t, x) 
	&\colon (t,x) 
	\in 
	( a + k(b-a), b + k(b-a) ] \times V
	\end{cases} 
	,
	\end{split}
	\end{equation}
	and let
		$ \Phi_{0,1}\colon \R \times V \to L(V, W) $
		be the function
		which satisfies 
		for all
	$ t \in \R $,
	$ x \in V $,
	$ k \in \N $
	that
	{\scriptsize
		\begin{equation}
		\begin{split}
		\label{eq:def_per_2}
		\Phi_{0,1}(t,x)
		&=
		\begin{cases} 
		(  \frac{\partial}{\partial x } \phi)(t,x) 
		&\colon 
		(t,x) \in [a,b] \times V 
		\\
		2 \Phi_{0,1}( b - k(b-a), x ) 
		- 
		\Phi_{0,1}( 2 ( b - k (b-a) ) - t, x) 
		&\colon (t,x) 
		\in [ a - k(b-a), b - k (b-a) ) \times V 
		\\
		2 \Phi_{0,1}( a + k(b-a), x) 
		- 
		\Phi_{0,1}( 2( a + k(b-a) ) - t, x) 
		&\colon (t,x) 
		\in 
		( a + k ( b - a ), b + k ( b - a ) ] \times V
		\end{cases} 
		.
		\end{split}
		\end{equation}	
}Note that 
the fact that
$ [a,b] \times V \ni (t,x) 
\mapsto \phi(t,x) \in W $
is continuous 
ensures that
\begin{equation} 
\label{eq:Extension is continuous}
%
\Phi
\in 
\mathcal{C}( \R \times V, W ) 
.
\end{equation} 
Next observe that for all
	$ t \in \R $, $ x \in V $ it holds that
	$ \R \ni s \mapsto \Phi(s,x) \in W $
	is differentiable
	and that
	\begin{equation} 
	\label{eq:derivatives 1,0 coincide}
	\Phi_{1,0}(t,x) 
	= 
	( \tfrac{\partial}{\partial t } \Phi )(t,x) 
	.
	\end{equation}
	Furthermore, note that the fact that
	$ ( [a,b] \times V \ni (t,x) \mapsto ( \frac{\partial }{\partial t }\phi )(t,x) \in W ) \in \mathcal{C} ( [a,b] \times V, W ) $
	assures that
	\begin{equation} 
	\label{eq:Cont1}
	\Phi_{1,0} 
	\in \mathcal{C} ( \R \times V , W ) 
	.
	\end{equation}
	Combining this and~\eqref{eq:derivatives 1,0 coincide}
	proves that 
	\begin{equation}
	\label{eq:derivative 1,0 is continuous}
	( \R \times V  \ni (t ,x) \mapsto 
	( \tfrac{ \partial }{ \partial t } \Phi )(t,x)
	\in W )
	\in \mathcal{C} ( \R\times V , W ) 
	.
	\end{equation}
	In addition, note that
	for all 
	$ t \in \R $,
	$ x \in V $ 
	it holds that
	$ V \ni v \mapsto \Phi(t,v) \in W $
	is differentiable
	and
	that 
	\begin{equation} 
	\label{eq:derivatives 0,1 coincide}
	\Phi_{0,1}(t,x) 
	= 
	( \tfrac{\partial }{\partial x } \Phi)(t,x) 
	.
	\end{equation}
	Moreover, observe that the fact that
	$ ( [a,b] \times V \ni (t,x) \mapsto 
	( \frac{\partial}{\partial x } \phi) (t,x) \in L(V,W) ) 
	\in \mathcal{C} ( [a,b] \times V, L(V, W) ) $
	implies that
	\begin{equation} 
	\label{eq:Cont2}
	\Phi_{0,1} 
	\in \mathcal{C} ( \R \times V , L(V, W) ) 
	.
	\end{equation}
	This and~\eqref{eq:derivatives 0,1 coincide} 
	show that
	\begin{equation}
	\label{eq:derivative 0,1 is continuous} 
	( \R \times V \ni (t,x) \mapsto 
	( \tfrac{\partial}{\partial x } \Phi) (t,x) \in L(V,W) ) 
	\in \mathcal{C} (\R \times V, L(V, W) )
	.
	\end{equation}
Combining~\eqref{eq:Extension is continuous},
\eqref{eq:derivative 1,0 is continuous},
\eqref{eq:derivative 0,1 is continuous},
	and, e.g., 
	Coleman~\cite[Corollary~3.4]{Coleman12}
	completes the proof of Lemma~\ref{lemma:ExtendedF}.
\end{proof}
\begin{corollary}
\label{corollary:Alekseev_grobner} 
Let 
$ (V, \left\| \cdot \right\|_V) $ 
be a nontrivial $ \R $-Banach space, 
let
$ T \in (0, \infty) $,
$ f \in \mathcal{C}^{0,1}( [0, T] \times V, V) $,  
let
$ Y, E \colon [0,T] \to V $
be 
strongly measurable functions, 
for
every 
$ x \in V $,
$ s \in [0,T] $
let
$ X_{s,(\cdot)}^x
=
(X_{s,t}^x)_{t \in [s,T]} 
\colon [s,T] \to V $
be a continuous function
which satisfies
for all
$ t \in [s,T] $
that  
$ X_{s,t}^x = 
x
+
\int_s^t f( \tau , X_{s, \tau }^x) \,d\tau $,
and assume for all
$ t \in [0,T] $ that
$ \int_0^T 
[
\| f ( \tau, Y_\tau ) \|_V 
+
\| E_\tau \|_V
]
\, d\tau < \infty $ 
and 
$ Y_t 
= Y_0 + \int_0^t [ f(\tau, Y_\tau   ) + E_\tau  ] 
\,
d\tau $.
Then
\begin{enumerate}[(i)]
\item \label{item:0} it holds that
$ ( \{(u,r)\in[0,T]^2:u \leq r\}
\times V \ni (s,t,x) \mapsto X_{s,t}^x \in V )
\in \mathcal{C}^{0,0,1}( \{(u,r)\in[0,T]^2:u \leq r\}
\times V, V ) $,	
\item \label{item:1} it holds 
that 
$ \{(u,r)\in(0,T)^2:u< r\}
\times V \ni (s,t,x) \mapsto X_{s,t}^x \in V $
is continuously
differentiable,
\item \label{item:2} it holds 
for all 
$ t \in [0 ,T] $
that
$ [0,t] \ni \tau
\mapsto
 (
\tfrac{\partial}{\partial x} X_{\tau, t}^{Y_\tau} 
 ) E_\tau
\in V $
is strongly
measurable,
\item \label{item:3} it holds for all 
$ t \in [0,T] $
that
$ \int_{0}^t
 \| 
 (
\tfrac{\partial}{\partial x} X_{\tau, t}^{Y_\tau} 
 ) E_\tau
 \|_V
\, 
d\tau < \infty $,
and
\item \label{item:4}
it holds for all
$ s \in [0,T] $, $ t \in [s,T] $
that
\begin{equation}
Y_t 
= 
X_{s, t}^{Y_s}
+
\int_{s}^t 
 (
\tfrac{\partial}{\partial x} X_{\tau, t}^{Y_\tau} 
 ) E_\tau
\, 
d\tau. 
\end{equation}
\end{enumerate}
\end{corollary}
\begin{proof}[Proof
of Corollary~\ref{corollary:Alekseev_grobner}]
Throughout this proof let
$ \angle_T \subseteq [0,T]^2 $ be the set given by
$ \angle_T = \{ (s, t) \in [0,T]^2 \colon s \leq t \} $ 
and let
$ \Phi \colon \angle_T \times V \to V $
be the function which satisfies
for all
$ (s,t) \in \angle_T $, 
$ x \in V $
that
$ \Phi_{s,t} (x) = X_{s,t}^x $.   
Note that,
e.g., 
Coleman~\cite[Corollary~3.4]{Coleman12}
and
items~\eqref{item:Diff_s},
\eqref{item:JointContinuous},
\eqref{item:initialTimeDiff}, 
\eqref{item:Differentiable_t},
and~\eqref{item:Diff_t} 
of 
Lemma~\ref{lemma:Differentiable}
(with
$ V = V $,
$ T = T $, 
$ f = f $,
$ X_{s,t}^x = X_{s,t}^x $
for 
$ (s, t ) \in \angle_T $, 
$ x \in V $
in the notation of
items~\eqref{item:Diff_s},
\eqref{item:JointContinuous},
\eqref{item:initialTimeDiff}, 
\eqref{item:Differentiable_t},
and~\eqref{item:Diff_t}  of
Lemma~\ref{lemma:Differentiable}) 
establish items~\eqref{item:0} and~\eqref{item:1}.
It thus remains to prove items~\eqref{item:2}--\eqref{item:4}.
For this observe that 
 item~\eqref{item:Diff_t2} 
 of 
 Lemma~\ref{lemma:Differentiable}
 (with
 $ V = V $,
 $ T = T $, 
 $ f = f $,
 $ X_{s,t}^x = X_{s,t}^x $
 for 
 $ (s, t ) \in \angle_T $, 
 $ x \in V $
 in the notation of
 item~\eqref{item:Diff_t2} of
 Lemma~\ref{lemma:Differentiable}) 
 ensures that
there exists a unique
continuous function 
$ \dot \Phi \colon \angle_T \times V \to V $
 which satisfies
for all
$ s \in [0,T) $,
$ t \in [s,T] $,
$ x \in V $ 
that
\begin{equation} 
\dot \Phi_{s, t}(x) 
= \tfrac{\partial}{\partial t} ( \Phi_{s, t}(x) ) 
.
\end{equation} 
Next observe that item~\eqref{item:0}
ensures that there exists a function
$ \Phi^\star  \colon \angle_T \times V \to L(V) $  
which satisfies 
for all
$ (s,t) \in \angle_T $, 
$ x \in V $
that
\begin{equation} 
\Phi_{s, t}^\star(x) =
\tfrac{\partial}{\partial x}
(
\Phi_{s, t}(x) 
).
\end{equation} 
Moreover, note that 
items~\eqref{item:Diff_s},
\eqref{item:JointContinuous}, 
and~\eqref{item:initialTimeDiff}
of
Lemma~\ref{lemma:Differentiable}
(with
$ V = V $,
$ T = T $, 
$ f = f $,
$ X_{s,t}^x = X_{s,t}^x $
for 
$ (s, t ) \in \angle_T $, 
$ x \in V $
in the notation of
items~\eqref{item:Diff_s},
\eqref{item:JointContinuous}, 
and~\eqref{item:initialTimeDiff} of
Lemma~\ref{lemma:Differentiable})
and
Lemma~\ref{lemma:ExtendedF}
(with
$ V = V $,
$ W = V $,
$ a = s $,
$ b = t $,
$ \phi =  ( [s,t] \times V \ni (u,x) \mapsto X_{u,t}^x \in V ) $
for
$ s \in [0,T) $,
$ t \in (s,T] $
in the notation of
Lemma~\ref{lemma:ExtendedF})
ensure that
for all
$ s \in [0,T) $,
$ t \in (s,T] $
it holds that
\begin{equation} 
( [s,t] \times V \ni (u,x) \mapsto X_{u,t}^x \in V )
\in \mathcal{C}^1( [s,t] \times V, V ) 
.
\end{equation} 
Combining item~\eqref{item:1}, 
Lemma~\ref{lemma:flow_property}
(with
$ V = V $,
$ T = T $,
$ f = f $,
$ X_{s, t}^x = X_{s, t}^x $
for
$ (s, t ) \in \angle_T $,
$ x \in V $
in the notation of Lemma~\ref{lemma:flow_property}),
items~\eqref{item:Diff_s}  
and~\eqref{item:Differentiable_t} 
of
Lemma~\ref{lemma:Differentiable}
(with
$ V = V $,
$ T = T $, 
$ f = f $,
$ X_{s,t}^x = X_{s,t}^x $
for 
$ (s, t ) \in \angle_T $, 
$ x \in V $
in the notation of
items~\eqref{item:Diff_s}    
and~\eqref{item:Differentiable_t}  
of
Lemma~\ref{lemma:Differentiable}),
and
Proposition~\ref{proposition:Banach_space_alekseev_grobner}
(with
$ V = V $,
$ t_0 = s $,
$ t = t $,
$ \phi = \operatorname{Id}_V $,
$ F = Y |_{[s, t]} $,
$ \Phi = \Phi |_{  \{ (u, r) 
	\in [s, t]^2\colon u \leq r \} \times V} $,
$ \dot \Phi_{v,w} = \dot \Phi_{v,w} $,
$ \Phi_{v,w}^\star = \Phi_{v,w}^\star $,
$ f = ( [s,t] \ni \tau \mapsto f(\tau, Y_\tau) + E_\tau \in V ) $
for 
$ s \in [0,T) $,
$ t \in (s,T] $,
$ v \in [s,t] $,
$ w \in [v, t] $
in the notation of 
Proposition~\ref{proposition:Banach_space_alekseev_grobner})
hence
demonstrates
that for all 
$ (s,t) \in \angle_T $
it holds that
\begin{equation}
\label{eq:Strongly}
[s,t]
\ni 
\tau
\mapsto
\Phi_{\tau, t}^\star(Y_\tau)
\big[
\dot
\Phi_{\tau, \tau}(Y_\tau)
-
f( \tau, Y_\tau ) - E_\tau
\big]
\in V
\text{ is strongly measurable},
\end{equation}
\begin{equation}
\int_s^t
\big\|
\Phi_{\tau, t}^\star(Y_\tau)
\big[
\dot
\Phi_{\tau, \tau}(Y_\tau)
-
f( \tau, Y_\tau ) - E_\tau
\big]
\big\|_V
\,
d \tau 
< \infty,
\text{ and}
\end{equation}
\begin{equation}
\begin{split}
\label{eq:application_Alekseev}
&
X_{s,t}^{Y_s } - Y_t 
=
\int_s^t 
\Phi_{\tau, t}^\star(Y_\tau)
\big[
\dot
\Phi_{\tau, \tau}(Y_\tau)
-
f( \tau, Y_\tau ) - E_\tau
\big]
\, d\tau
.
\end{split}
\end{equation}
Moreover, note that
Lemma~\ref{lemma:extensions}
(with $ V = V $,
$ a = 0 $,
$ b = t $,
$ f = ( [0,t] \ni s \mapsto f( s, X_{s, t }^x ) \in V ) $,
$ F = ( [0,t] \ni s \mapsto X_{s,t}^x \in V ) $
for
$ t \in [0,T] $,
$ x \in V $
in the notation of
Lemma~\ref{lemma:extensions})
shows that
for all 
$ \tau \in [0, T) $,
$ t \in [\tau, T] $,
$ x \in V $
it holds that
\begin{equation}
\begin{split}
\label{eq_derivatives}
\dot
\Phi_{\tau, t}(x)
=
f(t, X_{\tau, t}^x )
.
\end{split}
\end{equation}
Combining this, \eqref{eq:Strongly}--\eqref{eq:application_Alekseev}, 
and the fact that
$ \forall\, \tau \in [0,T) $, $ t \in [\tau,T] $, $ x \in V \colon \Phi_{\tau, t}^\star(x)
=
\tfrac{\partial}{\partial x}
X_{\tau, t}^x  $
establishes items~\eqref{item:2}--\eqref{item:4}.
The proof
of Corollary~\ref{corollary:Alekseev_grobner}
is thus completed.
\end{proof}
%
%
%
%
\subsubsection*{Acknowledgements}
This project has been partially supported through the SNSF-Research project $ 200021\_156603 $ ''Numerical 
approximations of nonlinear stochastic ordinary and partial differential equations''.
\newpage
\bibliographystyle{acm}
\bibliography{../Bib/bibfile}

\end{document}